\documentclass{amsart}

\title[KP solitons and total positivity for the Grassmannian]
{KP solitons and total positivity for the Grassmannian}

\author{Yuji Kodama and Lauren Williams} 
\date{\today}
\thanks{The first author was partially
supported by NSF grants DMS-0806219 and DMS-1108813.  The second author was 
partially supported by the NSF grant DMS-0854432 and an 
Alfred Sloan Fellowship.}
\address{Department of Mathematics, Ohio State University,
Columbus, OH 43210}
\email{kodama@math.ohio-state.edu}
\address{Department of Mathematics, University of California,
Berkeley, CA 94720-3840}
\email{williams@math.berkeley.edu}
\subjclass[2000]{}

\usepackage{url}
\usepackage{color}

\usepackage{rotating}
\newcommand{\rotxc}[1]{\begin{sideways}#1\end{sideways}}
\newcommand{\invert}[1]{\rotxc{\rotxc{#1}}}

\def\Le{\hbox{\invert{$\Gamma$}}}
\def\vblack(#1, #2)#3{\cnode*[linecolor=black](#1, #2){3}{#3}}
\def\vwhite(#1,#2)#3{\cnode[linecolor=black,fillcolor=white,fillstyle=solid](#1,#2){3}{#3}}
\countdef\x=23
\countdef\y=24
\countdef\z=25
\countdef\t=26

\def\tbox(#1,#2)#3{
\x=#1 \y=#2 
\multiply\x by 12 
\multiply\y by 12 
\z=\x \t=\y
\advance\z by 12 
\advance\t by 12 
\psline(\x,\y)(\x,\t)(\z,\t)(\z,\y)(\x,\y)
\advance\x by 6
\advance\y by 6 
\rput(\x,\y){{\bf #3}}}

\usepackage{amssymb}
\usepackage{graphicx}
\usepackage{epstopdf}
\usepackage{amscd}
\usepackage{graphics}

\def\proof{\par{\it Proof}. \ignorespaces}
\def\endproof{{\ \vbox{\hrule\hbox{%
     \vrule height1.3ex\hskip0.8ex\vrule}\hrule }}\par}

\theoremstyle{definition}

\theoremstyle{remark}

\numberwithin{equation}{section}

\let\trueint=\int
\let\truesum=\sum
\def\int{\mathop{\textstyle\trueint}\limits}
\def\sum{\mathop{\textstyle\truesum}\limits}
\let\<=\langle
\let\>=\rangle

\def\S{{\mathcal S}}

\def\rank{\mathop{\rm rank}\nolimits}

\def\t{{\bf t}}
\def\a{{\bf a}}

\usepackage{amsmath, amsthm, amssymb, amsbsy}
\usepackage{amsfonts, latexsym, stmaryrd, amscd, xy}
\usepackage[mathscr]{eucal}
\usepackage{epsfig}
\newtheorem{theorem}{Theorem}[section]

\newtheorem{definition}[theorem]{Definition}
\newtheorem{proposition}[theorem]{Proposition}
\newtheorem{lemma}[theorem]{Lemma}
\newtheorem{example}[theorem]{Example}
\newtheorem{corollary}[theorem]{Corollary}

\newtheorem{remark}[theorem]{Remark}
\newtheorem{algorithm}[theorem]{Algorithm}
\usepackage{amsfonts}
\usepackage{xy}
\usepackage{amssymb}
\usepackage{amsmath}

\newcommand{\R}{\mathbb R}

\newcommand{\Grkn}{(Gr_{k,n})_{\geq 0}}
\newcommand{\Grkntop}{(Gr_{k,n})_{> 0}}

\newcommand{\ttt}{\mathbf{t}}

\DeclareMathOperator{\CC}{\mathcal C}

\DeclareMathOperator{\GL}{GL}

\DeclareMathOperator{\M}{\mathcal M}
\DeclareMathOperator{\SSS}{\mathcal S}

\newcommand{\thmrefer}[1]{\renewcommand\thetheorem
  {\protect\ref{#1}}\addtocounter{theorem}{-1}}

\xyoption{all}
\CompileMatrices

\begin{document}

\begin{abstract}
Soliton solutions of the KP equation have been studied since 1970,
when Kadomtsev and Petviashvili proposed a two-dimensional nonlinear dispersive
wave equation now known as the KP equation.
It is well-known that one can use
the Wronskian method to construct a soliton solution
to the KP equation
from each point of the real Grassmannian $Gr_{k,n}$.
More recently, several authors \cite{BK,K04, BC06, CK1,CK3}
have studied the {\it regular}
solutions that one obtains in this way:
these come from points of the totally non-negative part of the
Grassmannian $(Gr_{k,n})_{\geq 0}$. 

In this paper we exhibit a surprising 
connection between the theory of
total positivity for the Grassmannian, 
and the structure of regular soliton solutions to the KP equation.
By exploiting this connection, we obtain 
new insights into the structure of KP solitons,
as well as new interpretations of the combinatorial objects
indexing cells of $(Gr_{k,n})_{\geq 0}$ \cite{Postnikov}.
In particular, 
we completely classify  the spatial patterns of the soliton solutions 
coming from $(Gr_{k,n})_{\geq 0}$ when the absolute
value  of the time parameter 
is sufficiently large.
We demonstrate an intriguing connection between soliton graphs
for $(Gr_{k,n})_{>0}$ and the {\it cluster algebras} of Fomin and Zelevinsky \cite{FZ},
and we use this connection
to solve the {\it inverse problem} for generic KP solitons coming from 
$(Gr_{k,n})_{>0}$.
Finally we construct all the soliton graphs
for $(Gr_{2,n})_{>0}$ using the triangulations of an $n$-gon.
\end{abstract}

\maketitle

\setcounter{tocdepth}{1}
\tableofcontents

\section{Introduction}

The KP equation is a two-dimensional
nonlinear dispersive wave equation given by
\begin{equation}\label{eq:KP}
\frac{\partial}{\partial x}\left(-4\frac{\partial u}{\partial t}+6u\frac{\partial u}{\partial x}+\frac{\partial^3u}{\partial x^3}\right)+3\frac{\partial^2u}{\partial y^2}=0,
\end{equation}
where $u=u(x,y,t)$ represents the wave amplitude at the point $(x,y)$ in the $xy$-plane for fixed time $t$.
The equation was proposed by Kadomtsev and Peviashvili
in 1970 to study the transversal stability of the soliton solutions of the Korteweg-de Vries (KdV)
equation \cite{KP70}.  The KP equation can also be used to describe shallow water waves, and
in particular, the equation provides an excellent model for the resonant interaction of those waves
(see \cite{K10} for recent progress). The equation has  a rich
mathematical structure, and is now considered
to be the prototype of an integrable nonlinear
dispersive wave equation with two spatial dimensions (see for example  \cite{NMPZ84,AC91,D91,MJD00,H04}).

One of the main breakthroughs in the KP theory was given by Sato \cite{Sato}, who 
realized that solutions of the KP equation could be written in terms of points on an 
infinite-dimensional Grassmannian.  The present paper deals with a real, finite-dimensional
version of the Sato theory; in particular, we are interested in solutions that 
are regular in the entire $xy$-plane, where they are localized along
certain rays.  
We call such solution \emph{line-soliton solution}, and they
can be constructed from a point $A$ of the real Grassmannian
\cite{Sato,Sa79,FN,H04}.
In this paper, we denote by $u_A(x,y,t)$ the solution associated to $A$.


Recently several authors
have worked on classifying the regular line-soliton solutions
\cite{BK,K04, BC06, CK1,CK3}.  These solutions come
from points of the {\it totally non-negative part of the Grassmannian}, that is,
those points of the real Grassmannian 
whose Pl\"ucker coordinates are all non-negative.  They found 
a large variety of soliton solutions which were previously overlooked by those using the
Hirota method of a perturbation expansion \cite{H04}.
In the generic situation, the asymptotic pattern at $y\to \pm \infty$ of the solution
consists of $n$ line-solitons.
However, because of the nonlinearity in the KP equation, the interaction pattern of the soliton solutions are very complex. Figure \ref{contour-example} illustrates the time evolution of the pattern of a line-soliton solution.
Each figure shows the contour plot of the solution at a fixed time $t$
in the $xy$-plane with $x$ in the horizontal
and $y$ in the vertical directions.
\begin{figure}[h]
\centering
\includegraphics[height=1.7in]{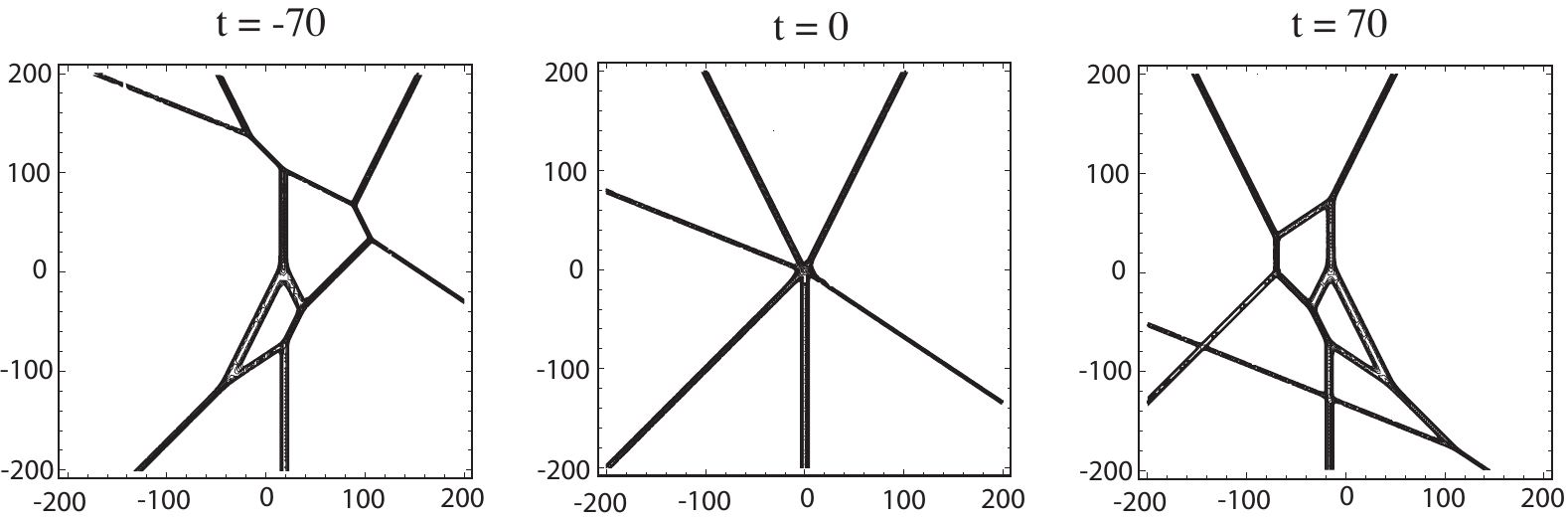}
\caption{Time evolution of the spatial pattern of a soliton solution. The figures illustrate the contour plots of the solution (see Example \ref{exercise}
to reconstruct the figures).
\label{contour-example}}
\end{figure}
One of the main goals of this paper is to give a combinatorial classification
of the patterns generated by the line-soliton solutions as in the figures.

Recently Postnikov \cite{Postnikov} studied the totally non-negative part of the 
Grassmannian $(Gr_{k,n})_{\geq 0}$ from a combinatorial point of view.  
Total positivity has attracted a lot of interest in the last two decades, largely due to work 
of Lusztig \cite{Lusztig2, Lusztig3}, who introduced 
the totally positive and non-negative parts of real reductive groups and
flag varieties (of which the Grassmannian is an important example).
Postnikov gave a decomposition of $(Gr_{k,n})_{\geq 0}$ into {\it positroid cells}, by 
specifying which Pl\"ucker coordinates  are strictly positive
and which are zero.
He also introduced several remarkable families of combinatorial objects, including
{\it decorated permutations, $\Le$-diagrams, plabic graphs,} and
{\it Grassmann necklaces,}
in order to index the cells and describe their properties.

We are interested in the wave pattern generated by a soliton solution $u_A(x,y,t)$
in the $xy$-plane for fixed time $t$.  We then consider a \emph{contour plot} $\CC(u_A,t)$ for each $t$ in the $xy$-plane which
is a {\it tropical curve}  approximating
the positions in the plane where the corresponding wave has a peak, see Figure
\ref{contour-example}.

While the local interactions of arbitrary contour plots are extremely 
complicated, it is possible to  understand their asymptotic structure
for $|y|\gg0$.  Moreover, if we take the limit as the time
variable $t$ goes to infinity (or more generally, some of the symmetry parameters of the KP equation, denoted by
$\t=(t_3, t_4, t_5, \dots)$ with $t_3=t$, also go to infinity),  and rescale $x$ and $y$ accordingly, 
we obtain an 
\emph{asymptotic contour plot}, whose combinatorial structure
is much more tractable.  
We then associate to each asymptotic contour plot a {\it soliton graph} 
by forgetting the metric structure of the pattern but remembering the topological
structure.

In this paper we establish a tight connection between 
total positivity on the Grassmannian and the regular soliton solutions of the KP equation.
This allows us to apply machinery from total positivity to understand 
soliton solutions of the KP equation. In particular:
\begin{itemize}
\item  we
classify the soliton graphs and asymptotic contour plots
coming from $(Gr_{k,n})_{\geq 0}$  when $t \to \pm \infty$ 
(Theorems \ref{t<<0} and \ref{t>>0}); 
\item we demonstrate an intriguing connection between soliton graphs
for $(Gr_{k,n})_{>0}$ and the {\it cluster algebras} of Fomin and Zelevinsky \cite{FZ}
(Theorem \ref{soliton-cluster});
\item we solve the {\it inverse problem} for KP solitons coming from 
$(Gr_{k,n})_{>0}$, and from $(Gr_{k,n})_{\geq 0}$ when $|t|\gg 0$ (Theorems \ref{inverse1}
and \ref{inverse2}).
\item we classify all soliton graphs and asymptotic contour
plots coming from $(Gr_{2,n})_{>0}$, and show 
that these soliton graphs are in bijection with triangulations of a 
polygon 
(Theorem \ref{theorem:maintriangulation}).
\end{itemize}

Note that prior to our work almost nothing was known about the classification of soliton graphs,
except in the cases of $(Gr_{1,n})_{>0}$ \cite{DMH}, and $(Gr_{2,4})_{\geq 0}$ \cite{CK3}.

In the other direction, we give a KP soliton interpretation to nearly all of 
Postnikov's combinatorial objects, as well as a new characterization of 
{\it reduced plabic graphs} (Theorem \ref{th:reduced}).

The structure of the paper is as follows.
In Sections \ref{sec:TP} and \ref{soliton-background} we provide background on total positivity
on the Grassmannian, and soliton solutions to the KP equation.  In Section \ref{sec:solgraph} 
we explain how to associate soliton graphs to soliton solutions of the KP equation.
In the next four sections (Sections \ref{permutations}, \ref{sec:necklace}, \ref{sec:sol=plabic},
and \ref{plabic-soliton})
we explain the relationships  between combinatorial objects labeling positroid cells
and the corresponding soliton solutions.  In particular, we explain how
(decorated) permutations and Grassmann necklaces control 
the asympototics of soliton graphs when $|y|\gg0$, and how $\Le$-diagrams control the soliton graphs
at $|t|\gg0$.  We also explain the connection between plabic graphs and soliton graphs.
In Section
\ref{X-crossing} we explain how the existence of $X$-crossings
in contour plots corresponds to ``two-term" Pl\"ucker relations.
In Section \ref{Reduced-Cluster} we prove that generically, the dominant exponentials 
labeling the regions of a soliton graph for 
$(Gr_{k,n})_{>0}$ comprise a  {\it cluster} for the {\it cluster algebra} of $Gr_{k,n}$.  
In Section \ref{sec:inverse} 
we address the inverse problem for regular soliton solutions to the KP equation.
Finally, in Section \ref{sec:triangulation}, we completely classify the 
soliton graphs coming from solutions $u_A$  for $A\in (Gr_{2,n})_{>0}$,
and 
construct them all using triangulations of an $n$-gon.  

The present paper provides proofs of the results announced in 
\cite{KW11}.
In the sequel to this work \cite{KW12} 
we have extended many of the results of the present paper 
from the non-negative part of the Grassmannian to the real Grassmannian.  
In a future paper we plan to make a detailed study of the relationship
between {\it cluster transformations} and the evolution of soliton graphs.

\textsc{Acknowledgements:}
The authors are grateful for
the hospitality of the math departments at UC Berkeley and Ohio State,
where some of this work was carried out.  They are also grateful
to Sara Billey, and to an anonymous referee, 
whose comments helped them to greatly 
improve the exposition.

\section{Total positivity for the Grassmannian}\label{sec:TP}

In this section we review the Grassmannian $Gr_{k,n}$
and Postnikov's decomposition of its non-negative part $(Gr_{k,n})_{\geq 0}$
into positroid cells
\cite{Postnikov}.
Note that our conventions slightly differ from those of \cite{Postnikov}.

The real Grassmannian $Gr_{k,n}$ is the space of all
$k$-dimensional subspaces of $\R^n$.  An element of
$Gr_{k,n}$ can be viewed as a full-rank $k\times n$ matrix modulo left
multiplication by nonsingular $k\times k$ matrices.  In other words, two
$k\times n$ matrices represent the same point in $Gr_{k,n}$ if and only if they
can be obtained from each other by row operations.
Let $\binom{[n]}{k}$ be the set of all $k$-element subsets of $[n]:=\{1,\dots,n\}$.
For $I\in \binom{[n]}{k}$, let $\Delta_I(A)$
denote the maximal minor of a $k\times n$ matrix $A$ located in the column set $I$.
The map $A\mapsto (\Delta_I(A))$, where $I$ ranges over $\binom{[n]}{k}$,
induces the {\it Pl\"ucker embedding\/} $Gr_{k,n}\hookrightarrow \mathbb{RP}^{\binom{n}{k}-1}$.

For $\M\subseteq \binom{[n]}{k}$, the \emph{matroid stratum}
$S_{\mathcal{M}}$ is the set of elements of $Gr_{k,n}$
represented by all $k \times n$ matrices $A$ with 
$\Delta_I(A) \neq 0$ for $I\in \mathcal{M}$
and $\Delta_J(A) = 0$ for $J \notin \mathcal{M}$.  
The decomposition
of $Gr_{k,n}$ into the strata $S_{\mathcal M}$ is called the 
the \emph{matroid stratification}. 

\begin{definition} 
The \emph{totally non-negative Grassmannian} $(Gr_{k,n})_{\geq 0}$
(respectively, \emph{totally positive Grassmannian} $\Grkntop$)
is the subset of  $Gr_{k,n}$
that can be represented by $k\times n$ matrices $A$
with all 
$\Delta_I(A)$ non-negative (respectively, positive).
\end{definition}

Postnikov \cite{Postnikov} studied the
decomposition of $(Gr_{k,n})_{\geq 0}$ induced by 
the matroid stratification.  
More specifically,
for $\M\subseteq \binom{[n]}{k}$,
he defined the {\it positroid cell\/} $S_\mathcal{M}^{tnn}$ as
the set of elements of $(Gr_{k,n})_{\geq 0}$ represented by all $k\times n$ matrices $A$ with
$\Delta_I(A)>0$
for $I\in \mathcal{M}$ and 
$\Delta_J(A)=0$ for $J\not\in \mathcal{M}$.  It turns out that 
each nonempty $S_{\mathcal{M}}^{tnn}$ is actually a cell \cite{Postnikov},
and that this decomposition of 
$(Gr_{k,n})_{\geq 0}$ is 
a CW complex \cite{PSW}. Note that
$\Grkntop$ is a positroid cell; it is the unique 
positroid cell in $\Grkn$ of top dimension $k(n-k)$.
Postnikov showed that the cells of $(Gr_{k,n})_{\geq 0}$ are naturally labeled 
by (and in bijection
with) the following combinatorial objects \cite{Postnikov}:
\begin{itemize}
\item Grassmann necklaces $\mathcal I$ of type $(k,n)$
\item decorated permutations $\pi^{:}$ on $n$ letters with $k$ weak excedances 
\item equivalence classes of {\it reduced plabic graphs} of type $(k,n)$
\item $\Le$-diagrams of type $(k,n)$.
\end{itemize}

For the purpose of studying solitons, we are  interested only 
in the {\it irreducible}  positroid cells.

\begin{definition}
We say that a positroid cell 
$S_{\mathcal M}^{tnn}$ is \emph{irreducible} if 
the reduced-row echelon matrix $A$ of any point in the cell
has the following properties:
\begin{itemize}
\item Each column of $A$ contains at least one nonzero element.
\item Each row of $A$ contains at least one nonzero element in 
addition to the pivot.
\end{itemize}
\end{definition}

The irreducible positroid cells 
are indexed by:
\begin{itemize}
\item irreducible Grassmann necklaces $\mathcal I$ of type $(k,n)$
\item derangements $\pi$ on $n$ letters with $k$  excedances 
\item equivalence classes of {\it irreducible 
reduced plabic graphs} of type $(k,n)$
\item irreducible $\Le$-diagrams of type $(k,n)$.
\end{itemize}

We now review the definitions of  these objects and 
some bijections among them.

\begin{definition}
An  \emph{irreducible Grassmann necklace of type 
$(k,n)$} is a sequence $\mathcal I = (I_1,\dots,I_n)$ of 
subsets $I_r$ of $[n]$ of size $k$  such that, for $i\in [n]$, 
$I_{i+1}=(I_i \setminus \{i\}) \cup \{j\}$ for some $j\neq i$.
(Here indices $i$
are taken modulo $n$.)  
\end{definition}

\begin{example}\label{ex1}
$\mathcal I = (1257, 2357, 3457, 4567, 5678, 6789, 1789, 1289, 1259)$ is
an example of a Grassmann necklace of type $(4,9)$. 
\end{example}

\begin{definition}
A \emph{derangement} $\pi=(\pi(1),\dots,\pi(n))$
is a permutation $\pi \in S_n$ which has no fixed points.
An {\it excedance} of $\pi$ 
is a pair $(i,\pi(i))$ such that 
$\pi(i)>i$.  We call $i$ the {\it excedance position} and 
$\pi(i)$ the {\it excedance value}.  Similarly,
a {\it nonexcedance} is a pair $(i,\pi(i))$ such that $\pi(i)<i$.
\end{definition}

\begin{remark}
A \emph{decorated permutation} is a permutation in which 
fixed points are colored with one of two colors. 
Under the bijection between positroid cells and decorated permutations,
the  irreducible positroid cells correspond to derangements,
i.e. those decorated permutations which have no fixed points.
\end{remark}

\begin{example}\label{ex2}
The derangement $\pi=(6,7,1,2,8,3,9,4,5)\in S_9$ has 
excedances in positions $1, 2, 5, 7$.
\end{example}
\begin{definition}\label{def:plabic}
A \emph{plabic graph}
is a planar undirected graph $G$ drawn inside a disk
with $n$ \emph{boundary vertices\/} $1,\dots,n$ placed in counterclockwise
order
around the boundary of the disk, such that each boundary vertex $i$ 
is incident
to a single edge.\footnote{The convention of \cite{Postnikov} 
was to place the boundary vertices in clockwise order.}  Each internal vertex
is colored black or white.  See Figure \ref{plabic} for an example.
\end{definition}

\begin{figure}[h]
\centering
\includegraphics[height=1.25in]{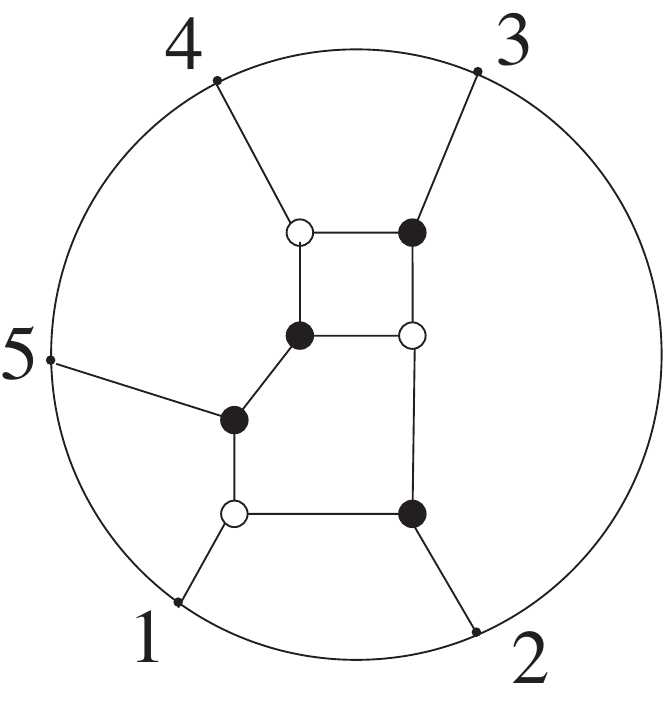}
\hskip 1cm
\includegraphics[height=1.25in]{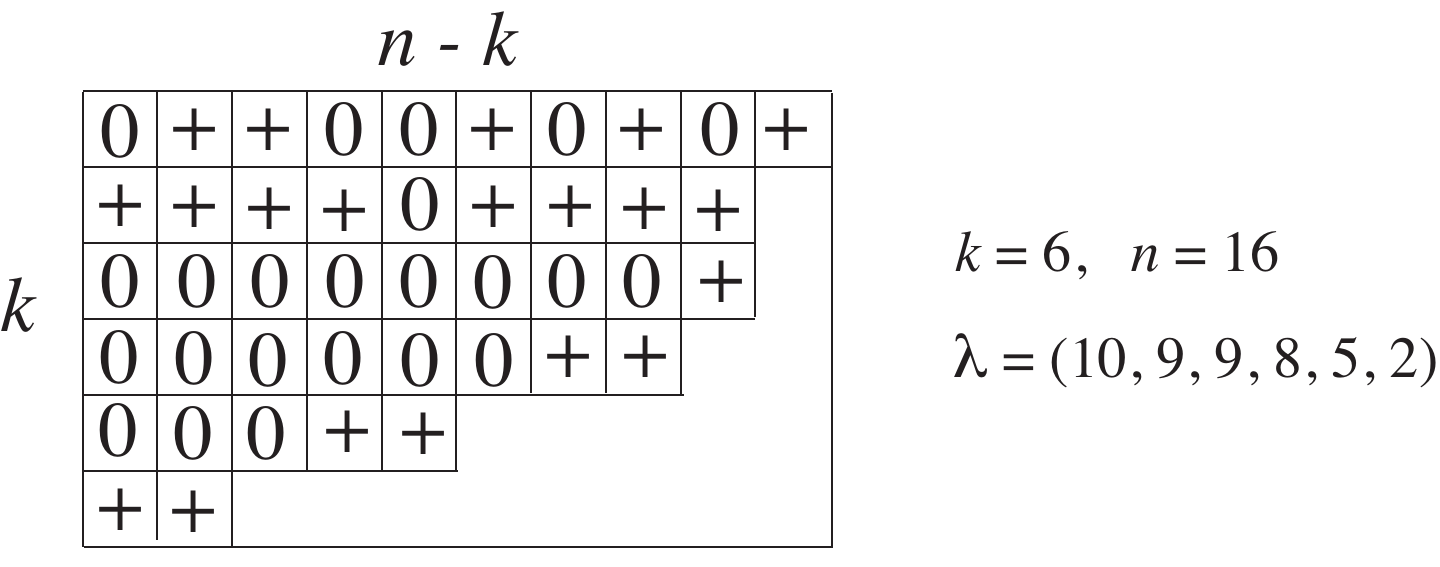}
\caption{A plabic graph, and an irreducible Le-diagram $L=(\lambda,D)_{k,n}$.}
\label{plabic}
\end{figure}

\begin{definition}
Fix $k$ and $n$. Let 
$Y_{\lambda}$ denote the Young diagram of the partition $\lambda$.  A {\it $\Le$-diagram}
(or Le-diagram)
$L=(\lambda, D)_{k,n}$ of type $(k,n)$ 
is a Young diagram $Y_{\lambda}$ contained in a $k \times (n-k)$ rectangle
together with a filling $D: Y_{\lambda} \to \{0,+\}$ which has the
{\it $\Le$-property}:
there is no $0$ which has a $+$ above it in the same column and a $+$ to its
left in the same row.
A $\Le$-diagram is \emph{irreducible} if each row and each column 
contains at least one $+$.
See the right of  Figure \ref{plabic} for
an example of an irreducible  $\Le$-diagram.
\end{definition}

\begin{theorem}\cite[Theorem 17.2]{Postnikov}\label{necklace}
Let $S_{\mathcal M}^{tnn}$ be a positroid
cell in $(Gr_{k,n})_{\geq 0}$. 
For $1 \leq r \leq n$, let $I_r$ be the  element
of  $\mathcal M$ which is lexicographically
minimal with respect to the order
$r < r+1 <  \dots < n < 1 < 2 < \dots r-1$.
Then $\mathcal I(\mathcal M):=(I_1,\dots,I_n)$ is a Grassmann necklace
of type $(k,n)$.
\end{theorem}

\begin{lemma}\cite[Lemma 16.2]{Postnikov}\label{Postnikov-permutation}
Given an irreducible Grassmann necklace $\mathcal I$, define
a derangement  $\pi=\pi(\mathcal I)$ by requiring that: 
if $I_{i+1} = (I_i \setminus \{i\}) \cup \{j\}$
for $j \neq i$, then 
$\pi(j)=i$.\footnote{Postnikov's convention was to set $\pi(i)=j$ above,
so the permutation we are associating is the inverse one to his.}
Indices are taken modulo $n$.
Then $\mathcal I \to \pi(\mathcal I)$
is a bijection from irreducible
Grassmann necklaces $\mathcal I=(I_1,\dots,I_n)$ of type $(k,n)$
to derangements $\pi(\mathcal I) \in S_n$ with $k$ excedances.
The excedances of 
$\pi(\mathcal I)$ are in positions $I_1$.
\end{lemma}

\begin{example}\label{ex3}
If $\mathcal I$ and $\pi$ are defined as in Examples \ref{ex1}
and \ref{ex2}, then $\pi(\mathcal I) = \pi$.
\end{example}



\begin{definition}\label{Le2permutation}
Given a $\Le$-diagram $L$ contained in a $k \times (n-k)$ rectangle,
label its southeast border with the numbers
$1,2,\dots,n$, starting at the northeast corner.
Replace each $+$ with an ``elbow" and each $0$ with a ``cross'';
see Figure \ref{pipedream}.  Now travel along each ``pipe" from
southeast to northwest, and label the end of a pipe with the same
number that labeled its origin.  Finally, we define a permutation
$\pi = \pi(L)$ as follows.  If $i$ is the label of a vertical
edge on the southeast border of $L$, then set $\pi(i)$ equal to 
the label of the vertical edge on the other side of that row.
If $i$ is the label of a horizontal edge on the southeast
border of $L$, then set $\pi(i)$ equal to the label of the horizontal
edge on the opposite side of that column.  
See Figure \ref{pipedream}.
\end{definition}
\begin{figure}[h]
\centering
\includegraphics[height=1.3in]{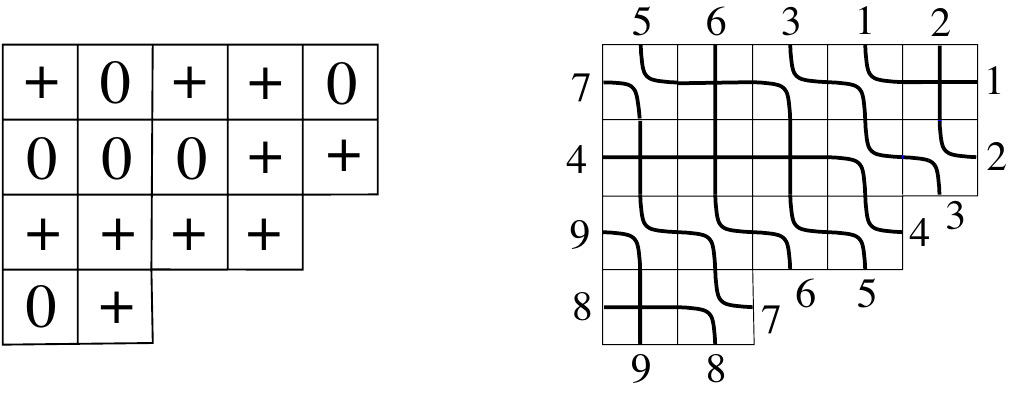}
\caption{The Le-diagram $L$ together with the computation of 
$\pi(L) = (7,4,2,9,1,3,8,6,5)$.}
\label{pipedream}
\end{figure}
\begin{proposition}
The map defined above gives a bijection from irreducible $\Le$-diagrams 
contained
in a $k \times (n-k)$ rectangle to 
derangements on $n$ letters with $k$ excedances.
\end{proposition}
\begin{proof}
This map can be shown to coincide with that from 
\cite[Section 2]{SW}, and up to a convention change,
coincides with 
the map in \cite[Corollary 20.1]{Postnikov}.
\end{proof}

\begin{remark}
Consider a positroid cell $S_{\mathcal M}^{tnn}$, and suppose that 
the Grassmann
necklace $\mathcal I$, the derangement $\pi$, and 
the $\Le$-diagram $L$, 
satisfy $\mathcal I = \mathcal I(\mathcal M)$, and 
$\pi = \pi(\mathcal I) = \pi(L)$.  
Then we also refer to this cell as
$S_{\mathcal M(\mathcal I)}^{tnn}$, $S_{\mathcal M(L)}^{tnn}$, 
$S_{\mathcal I}^{tnn}, S_{\pi}^{tnn}$, $S_{L}^{tnn}$, etc.  
\end{remark}


\section{Soliton solutions to the KP equation}\label{soliton-background}

In this section we explain how to construct a $\tau$-function
$\tau_A(x,y,\t)$ from a 
point of $Gr_{k,n}$, and then how to 
obtain a soliton solution 
to the KP equation from that $\tau$-function.

\subsection{From a point of the Grassmannian to a $\tau$-function.}

We first give a realization of $Gr_{k,n}$ with a 
specific basis of $\mathbb{R}^n$.  
The purpose of making this non-standard choice of basis is to 
identify the Pl\"ucker embedding of a point $A$ of the 
Grassmannian with a particular $\tau$-function, in \eqref{tau}
below.

Choose real parameters $\kappa_i$ such that 
$\kappa_1~<~\kappa_2~\cdots~<\kappa_n$.  In this paper
we will assume that the $\kappa_i$'s are \emph{generic}, 
meaning that:
\begin{itemize}
\item the sums $\sum_{j=1} ^p\kappa_{i_j}$ are all distinct for any $p$ with
$1<p<n$. 
\end{itemize}

We define a set of vectors $\{\mathsf{E}_j^{\bf 0}: j=1,\ldots,n\}$ by
\begin{equation}\label{eq:E}
\mathsf{E}_j^{\bf 0}:=\begin{pmatrix}
1\\ \kappa_j \\ \vdots \\ \kappa_j^n
\end{pmatrix} \,\in\,\mathbb{R}^n.
\end{equation}
Since all $\kappa_j$'s are distinct, the set $\{\mathsf{E}_j^{\bf 0}:j=1,\ldots,n\}$ forms a basis of $\mathbb{R}^n$.
Now define an $n\times n$ matrix $E^{\bf 0}=(\mathsf{E}_1^{\rm 0},\ldots,\mathsf{E}_n^{\bf 0})$ which is
the Vandermonde matrix in the $\kappa_j$'s, and let $A$ be a full-rank  $k\times n$ matrix  
representing a point of $Gr_{k,n}$. 
Then 
the vectors $\{\mathsf{F}_i^{\bf 0}\in\mathbb{R}^n :i=1,\ldots,k\}$
span a $k$-dimensional subspace in $\mathbb{R}^n$, 
where $\mathsf{F}_i^{\bf 0}$ is defined by
\[
\mathsf{F}_i^{\bf 0}:=\sum_{j=1}^na_{i,j}\,\mathsf{E}_j^{\bf 0},\qquad {\rm or}\qquad
(\mathsf{F}_1^{\bf 0},\ldots, \mathsf{F}_k^{\bf 0})={E}^{\bf 0}A^T,
\]
where $A^T$ is the transpose of $A$.
For $I=\{i_1,\ldots,i_k\}$, define the vector
$\mathsf{E}_{I}^{\bf 0}=\mathsf{E}_{i_1}^{\bf 0}\wedge\cdots\wedge \mathsf{E}_{i_k}^{\bf 0}
$.
Then we have a realization of the  Pl\"ucker embedding:
\[
\mathsf{F}_1^{\bf 0}\wedge\cdots\wedge \mathsf{F}_k^{\bf 0} =\sum_{I\in\binom{[n]}{k}}\Delta_I(A)\mathsf{E}_I^{\bf 0}.
\]

In \cite{Sato}, Sato showed that each solution of the KP equation is given by a $\text{GL}_{\infty}$-orbit on the universal Grassmannian. To construct such an orbit
for a finite dimensional Grassmannian, 
we first define multi-time variables $\hat{\bf t}:=(t_1,\ldots, t_m)$ for $m\ge 3$
 which
give a parameterization of the orbit.  Then
we consider a deformation $\mathsf{E}_j^{\hat{\bf t}}$ of the vector $\mathsf{E}_j^{\bf 0}$ for $j=1,\ldots,n$, given by
\[
\mathsf{E}_j^{\hat{\bf t}}:=\mathsf{E}_j^{\bf 0}\exp\theta_j(\hat{\bf t})\qquad\text{with}\qquad \theta_j(\hat{\bf t}):=\sum_{i=1}^m\kappa_j^it_i.
\]
For the KP equation, we identify $t_1=x, t_2=y$ and $t_3=t$, and denote
those flow parameters by
\[
\hat{\mathbf{t}}=(x,y,\mathbf{t}),\qquad\text{with}\qquad \mathbf{t}:=(t_3,\ldots,t_m).
\]

We now define an orbit generated by the 
matrix $E^{\hat{\bf t}}=(\mathsf{E}_1^{\hat{\bf t}},\ldots,\mathsf{E}_n^{\hat{\bf t}})$ on elements of $\GL_n$, 
\[
g^{\hat{\bf t}}:=E^{\hat{\bf t}}g \qquad {\rm for~each}\quad g\in \GL_n. 
\]
Here we identify the $n\times k$ matrix consisting of the first $k$ columns of $g$
with the matrix $A^T$, the transpose of the matrix $A$ parametrizing a point
$Gr_{k,n}$.  Then the Pl\"ucker embedding gives
\[
g\cdot e_1\wedge\cdots\wedge e_k=\sum_{1\le i_1<\cdots<i_k\le n}\Delta_I(A) \,e_{i_1}\wedge\cdots\wedge  e_{i_k}.
\]
where $e_j$ is the $j$-th standard vector in $\R^n$.
Then $g^{\bf t}\cdot e_{1}\wedge\cdots\wedge e_{k}$  defines
a flow (orbit) of the highest weight vector on the 
corresponding fundamental representation 
of ${\GL}_n$.

We now define the \emph{$\tau$-function} as 
\begin{align}\label{tau1}
\tau(x,y,\t):=&\langle e_1\wedge\cdots\wedge e_k,~ \mathsf{F}_1^{\hat{\bf t}}\wedge\cdots\wedge\mathsf{F}^{\hat{\bf t}}_k\rangle\\
=&\langle e_1\wedge\cdots\wedge e_k,~
g^{\hat{\bf t}}\cdot e_{1}\wedge \cdots\wedge e_{k}\rangle,\nonumber
\end{align}
where $\mathsf{F}^{\hat{\bf t}}_j:=g^{\hat{\bf t}}\cdot e_{j}$, and
the bracket $\langle \cdot,\cdot\rangle$ is the usual inner product on the wedge product space $\bigwedge^k\R^n$.
Given 
$I=\{i_1,\ldots,i_k\} \subset [n]$,  we let 
$E_I(\hat{\bf t})$ denote the scalar function
\begin{equation}\label{Efunction}
\begin{array}{lll}
E_I(x,y,\t)&=\langle e_1\wedge\cdots\wedge e_k,\,
\mathsf{E}^{\hat{\bf t}}_{i_1}\wedge \cdots\wedge\mathsf{E}^{\hat{\bf t}}_{i_k}\rangle=
\langle e_1\wedge\cdots\wedge e_k,\,
\mathsf{E}^{\bf 0}_{i_1}\wedge \cdots\wedge\mathsf{E}^{\bf 0}_{i_k}\rangle\,e^{\theta_{i_1} + \cdots + \theta_{i_k}}\\[1.0ex]
&=K_I\,\exp\Theta_I(x,y,\t),
\end{array}
\end{equation}
where $K_I:=\prod_{l<m}(\kappa_{i_m}-\kappa_{i_l})$, and $\Theta_I(x,y,\t)$ is defined by
\begin{equation}\label{KT}
\Theta_I(x,y,\t):=\sum_{j=1}^k(\kappa_{i_j}x+\kappa_{i_j}^2y)+\sum_{j=1}^k\theta_{i_j}(\t).
\end{equation}

Then the $\tau$-function can be written as a sum of exponential terms,
 \begin{equation}\label{tau}
\tau(x,y,\t)=\tau_A(x,y,\t) = \sum_{I\in\binom{[n]}{k}}\Delta_I(A)\,E_I(x,y,\t).
\end{equation}

\begin{remark}
We think of the right-hand side of \eqref{tau} as giving 
the Pl\"ucker
embedding of the Grassmannian $Gr_{k,n}$ into a wedge product space
whose basis is given by the $E_I(x,y,\t)$'s.
\end{remark}

It follows that if  $A\in (Gr_{k,n})_{\ge 0}$, then 
$\tau_A>0$ for all $(x,y,\t)\in\mathbb{R}^m$.  

Note that the $\tau$-function defined in \eqref{tau} can be also
written in the Wronskian form 
\begin{equation}\label{tauA}
\tau_A(x,y,\t)={\rm Wr}(f_1,f_2,\ldots, f_k),
\end{equation}
with the scalar functions $\{f_j:j=1,\ldots,k\}$ given by
\[
(f_1,f_2,\ldots, f_k)^T=A\cdot (E_1,E_2,\ldots,E_n)^T,
\]
where $E_j$ is the exponential function defined by $E_j:=\exp\theta_j(x,y,\t)$.

\subsection{From the $\tau$-function to solutions of the KP equation}

It is  well known
(see \cite{H04, CK1, CK2, CK3})   
that if we set $x=t_1, y=t_2, t=t_3$ (treating the other
$t_i$'s as constants), the $\tau$-function defined in \eqref{tauA} 
gives rise to a soliton solution of the KP equation \eqref{eq:KP},
namely
\begin{equation}\label{KPsolution}
u_A(x,y,\t)=2\frac{\partial^2}{\partial x^2}\ln\tau_A(x,y,\t).
\end{equation}
The other $t_j$'s correspond to the flow parameters of the higher symmetries of the KP equation, and
the set of the symmetries is called the \emph{KP hierarchy} (see e.g. \cite{MJD00}).

It is easy to show that if $A \in \Grkn$, then 
such a solution $u_A(x,y,\mathbf{t})$ is regular for all $t_j\in\mathbb{R}$.
In the sequel to this paper \cite{KW12}, we  show that 
if $u_A(x,y,\t)$ is regular for all $x,y$ and $t=t_3$ (with the 
other $t_i$'s fixed constants), then $A\in \Grkn$.  
(In \cite[Proposition 4.1]{K10}, a weaker statement was proved:
if $u_A(x,y,\mathbf{t})$ is regular for all $t_j\in \mathbb{R}$, 
then $A\in \Grkn$.) 
For this reason we are mainly interested in solutions $u_A(x,y,\ttt)$
of the KP equation which come from points $A$ of 
$\Grkn$.


\section{From soliton solutions to soliton graphs}\label{sec:solgraph}

In this section 
we define certain tropical curves associated with soliton solutions:
\emph{contour plots} and \emph{asymptotic contour plots}.  We also 
define the notion of \emph{soliton graph}.

\subsection{Contour plots}
One can visualize a solution $u_A(x,y,\t)$ in the $xy$-plane 
by drawing level sets of the solution when the coordinates $\t=(t_3,\ldots,t_m)$ are fixed.  
For each $r\in \mathbb{R}$, we denote the corresponding level set by 
\[
C_r(\t):=\{(x,y)\in\mathbb{R}^2: u_A(x,y,\t)=r\}. 
\]
Figure \ref{fig:1soliton} depicts both a three-dimensional image of a solution $u_A(x,y,\t)$ for fixed $\t$,
as well as multiple level sets $C_r$. 
These levels sets are lines parallel to the line of the wave peak.

\begin{example}\label{ex:1soliton}
We compute the soliton solution $u_A(x,y,\t)$ associated to 
the $1\times 2$ matrix $A=(1~a)$ with $a>0$, considered as an element of 
 $(Gr_{1,2})_{> 0}$.
Write $E_1=e^{\theta_1}$ and $aE_2=e^{\theta_2+\ln a}=e^{\tilde{\theta}_2}$.  
Then the  $\tau$-function $\tau_A$ and the soliton solution $u_A$ are given by
\begin{align*}
\tau_A(x,y,\t)&=e^{\theta_1}+e^{\tilde\theta_2} 
=2e^{\frac{1}{2}(\theta_1+\tilde\theta_2)}\cosh\frac{1}{2}(\theta_1-\tilde\theta_2), \text{ and} \\
u_A(x,y,\t)&=\frac{1}{2}(\kappa_1-\kappa_2)^2{\rm sech}^2 \frac{1}{2}(\theta_1-\tilde\theta_2).
\end{align*}
This is a {\it line-soliton solution}, and the peak of the solution (wave crest)
is given by the equation $\theta_1=\tilde\theta_2$, i.e.
\[
x+(\kappa_1+\kappa_2)y+\sum_{j=3}^mh_{j-1}(\kappa_1,\kappa_2)t_j=-\frac{1}{\kappa_2-\kappa_1}\ln a.
\]
where we have used $\kappa_2^{j+1}-\kappa_1^{j+1}=(\kappa_2-\kappa_1)h_j(\kappa_1,\kappa_2)$ with $h_j(\kappa_1,\kappa_2)=\sum_{p+q=j}\kappa_1^p\kappa_2^q$.
\begin{figure}[h]
\begin{center}
\includegraphics[height=1.6in]{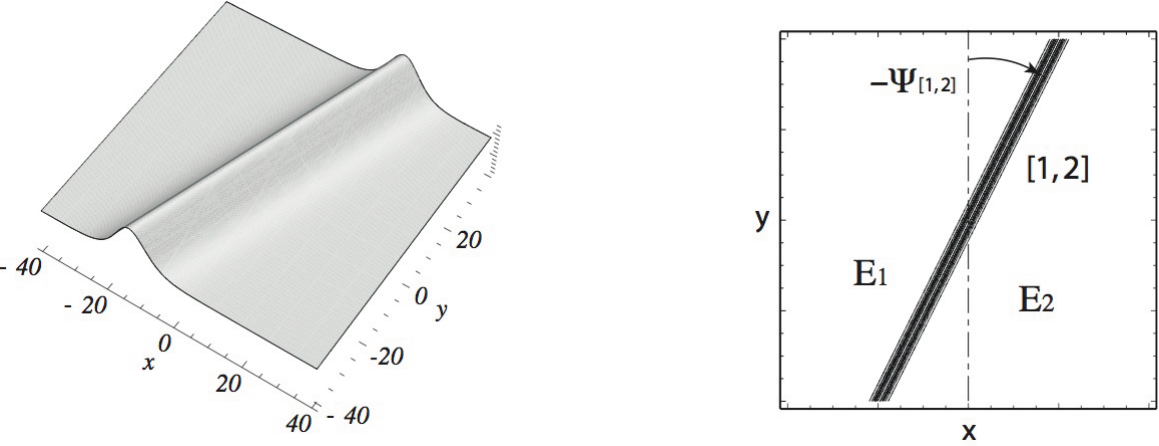}
\par
\end{center}
\caption{A line-soliton solution $u_A(x,y,\t)$ where $A=(1,1) \in (Gr_{1,2})_{\geq 0}$, depicted via
the 3-dimensional profile 
$u_A(x,y,{\bf 0})$, and the level sets of $u_A(x,y,0)$.
$E_i$ represents the dominant exponential in each region.
\label{fig:1soliton}}
\end{figure}

For each fixed $\t$, $\theta_1 = \tilde\theta_2$ gives a line which divides the $xy$-plane into two regions. 
The exponential $E_1$ dominates in the region including $x\ll 0$, and $E_2$ dominates the 
other region where  $x\gg 0$.
We label each region by its dominant  exponential.
Figure \ref{fig:1soliton} depicts 
$u_A(x,y,\t)$,
where  $\t=(0,\ldots,0)$, $a=1$, and 
$(\kappa_1,\kappa_2)=(-1,2)$. 
\end{example}

To study the behavior of $u_A(x,y,\ttt)$ for $A\in S_\mathcal{M}$,
we consider the dominant exponentials at each point $(x,y,\ttt)$,
and we define 
\begin{align*}
\hat{f}_A(x,y,\ttt)&= 
\underset{J\in\mathcal{M}}\max\{\Delta_J(A) E_J(x,y,\ttt)\}=\underset{J\in\mathcal{M}}\max\{\Delta_J(A) K_J
E_{j_1}\dots E_{j_k}\} \\ 
              &=\underset{J\in\mathcal{M}}\max \biggl\{\exp \biggl(\ln\bigl(\Delta_J(A) K_J \bigr) +  \Theta_J(x,y,\t)\biggr)\biggr\}, 
\end{align*}
where $K_J$ and $\Theta_J(x,y,\t)$ are defined in \eqref{Efunction}
and \eqref{KT}.
From (\ref{tau}), we see that 
$\tau_A$ can be approximated by $\hat{f}_A$.

Let $f_A(x,y,\t)$ be the closely related function 
\begin{equation}
{f}_A(x,y,\t) =
              \underset{J\in\mathcal{M}}\max \left\{\ln\bigl(\Delta_J(A)K_J\bigr) +
                    \Theta_J(x,y,\t)\right\}.
\end{equation}
Note that at a given point $(x,y,\t)$, $f_A(x,y,\t)$ is equal to a given term
if and only if 
$\hat{f}_A(x,y,\t)$ is equal to the exponentiated version of that term.

\begin{definition}\label{contour}
Given a solution $u_A(x,y,\t)$ of the KP equation as in 
(\ref{KPsolution}), we define 
$\CC(u_A, \t_0)$ 
for fixed $\t=\t_0$
to be the locus in $\R^2$ where $f_A(x,y,\t=\t_0)$ is not 
linear, and we refer to this as a \emph{contour plot} of the solution $u_A(x,y,\t)$.
\end{definition}

\subsection{Asymptotic contour plots}
Some of this paper will be concerned with the contour plots for large scales of the 
variables 
$(x,y,t_3,\ldots,t_m)$.  In this case,
each of the constant terms
$\ln\bigl(\Delta_J(A)K_J\bigr)$ in $f_A(x,y,\t)$ is negligible. More precisely,
we use rescaled variables $\bar\t:=(\bar{x},\bar{y},\mathbf{a})$ defined by
\[
\hat{\t}=(x,y,\mathbf{t})=s\,(\bar{x},\bar{y},\mathbf{a}),\qquad\text{for some}\quad s\gg 0.
\]
We then approximate $f_A$ by the function
\begin{equation*}
f_{\mathcal{M}}(\bar{x},\bar{y},\mathbf{a}) := \underset{J\in\mathcal{M}}\max \left\{
                    \Theta_J(\bar{x},\bar{y},\mathbf{a}))\right\},
\end{equation*}
which is obtained by taking the limit of $\frac{1}{s}f_A(x,y,\t)$ as $s\to\infty$.

\begin{definition}\label{contourplot-infinity}
We define the \emph{asymptotic contour plot}
$\CC_{\a_0}(\mathcal{M})$ for fixed $\a=\a_0$ to be the locus in $\R^2$ where
$f_{\mathcal{M}}(\bar{x},\bar{y},\a=\a_0)$ is not linear.  
Most of this paper will be concerned with the asymptotic contour plots
$\CC_{\a_0}(\mathcal{M})$ for 
$\a_0 = \pm (1,0,0,\dots)$, which we denote by 
$\mathcal{C}_{\pm}(\mathcal{M})$.  These asymptotic contour plots
are the limits of the finite contour plots
$\CC(u_A,\t_0)$ for $A \in S_{\mathcal{M}}$ and 
$\hat{\t}=s\bar{\t}$ in the limit $s\to\infty$.
\end{definition}


Note that 
each region of the complement of 
$\CC(u_A, \t_0)$ 
is a domain of linearity for $f_A(x,y,\t_0)$, and hence each region is  
naturally associated to a {\it dominant exponential}  
$\Delta_J(A) E_J(x,y,\t_0)$ from the $\tau$-function \eqref{tau}.
We label this region by $E_J$ or $\Theta_J$.
We label regions of the complement of each asymptotic contour plot
in the same way.

A \emph{line-soliton} is a finite or unbounded line segment
in a contour plot (or asymptotic contour plot) which 
represents a balance between two dominant exponentials in 
the $\tau$-function.  Lemma \ref{separating} and 
\eqref{eq-soliton} provide the equation for a line-soliton. 

Each contour plot $\CC(u_A, \t_0)$ and each 
asymptotic contour plot $\CC_{\a_0}(\mathcal{M})$ consists
of line segments, some of 
which have finite length, while others are unbounded
and extend in the $y$ direction to $\pm \infty$.  
The unbounded lines
are all 
line-solitons, which we call {\it unbounded}  line-solitons.
The finite line segments in asymptotic contour plots are
all line-solitons, 
but some of the finite line segments in non-asymptotic contour plots
may represent
{\it phase shifts}, which
have lengths which are determined by
the $\kappa$-parameters (see \cite[page 35]{CK3} for details).

\begin{lemma}\label{separating}\cite[Proposition 5]{CK3}
Consider a line-soliton in a contour plot.
The index sets of the 
dominant exponentials of the $\tau$-function in adjacent regions
of the
contour plot in the $xy$-plane are of the form $\{i,m_2,\dots,m_k\}$ and
$\{j, m_2,\dots,m_k\}$.
\end{lemma}

According to Lemma \ref{separating}, 
those two exponential terms have $k-1$ common
phases, so we call the line separating them a {\it line-soliton of type $[i,j]$}, or
simply an \emph{$[i,j]$-soliton}.
Locally we have
\begin{align*}
\tau_A\approx \Delta_I(A)E_I+\Delta_J(A)E_J
&=\left(\Delta_I(A)K_I E_i+\Delta_J(A)K_J E_j\right)\prod_{l=2}^kE_{m_l}\\
&=\left(e^{\theta_i+\ln(\Delta_I(A)K_I)}+e^{\theta_j+\ln(\Delta_J(A)K_J}\right)\prod_{l=2}^{k}E_{m_l},
\end{align*}
so the equation for this line-soliton is
\begin{equation}\label{eq-soliton}
x+(\kappa_i+\kappa_j)y+\sum_{p=3}^mh_{p-1}(\kappa_i,\kappa_j)t_p=-\frac{1}{\kappa_j-\kappa_i}\ln\frac{\Delta_J(A)K_J}{\Delta_I(A)K_I}.
\end{equation}
See also Example \ref{ex:1soliton}.

The equation for a line-soliton
in an asymptotic contour plot 
is the same as in \eqref{eq-soliton}, except that 
the constant term on the right-hand side is $0$ (this is immediate  from
 the definition of asymptotic contour plot).


\begin{remark}\label{slope}
Consider a line-soliton given by (\ref{eq-soliton}) for fixed $\t=(t_3,\ldots,t_m)$.
Compute the angle $\Psi_{[i,j]}$ 
between the line-soliton of type $[i,j]$ and the positive $y$-axis,
measured in the counterclockwise direction, so that the negative $x$-axis
has an angle of $\frac{\pi}{2}$ and the positive $x$-axis has an 
angle of $-\frac{\pi}{2}$. Then $\tan \Psi_{[i,j]} = \kappa_i+\kappa_j$,
so we refer to $\kappa_i+\kappa_j$ as the \emph{slope} of the 
$[i,j]$ line-soliton (see Figure \ref{fig:1soliton}).
Also note that the location of the line depends on 
the ratio of the Pl\"ucker coordinates corresponding to the 
dominant exponentials on either side of the line-soliton.
\end{remark}

We will be interested in the combinatorial structure of asymptotic
contour plots,
that is, the pattern of how line-solitons interact with each other. 
Generically we expect a point at which several line-solitons
meet to have degree $3$; we regard such a point as a trivalent 
vertex.  Three line-solitons meeting at a trivalent vertex
exhibit a {\it resonant interaction} (this corresponds to the 
{\it balancing condition} for a tropical curve),
 see Section \ref{resonance}.  
One may also have two line-solitons which cross over
each other, forming an $X$-shape: we call this
an \emph{$X$-crossing}, but do not regard it as a vertex.
See Figure \ref{contour-soliton} for examples. 
We will give more details about X-crossings in Section \ref{X-crossing}.

\begin{definition}\label{def:generic}
A contour plot $\CC(u_A,\t)$ is called \emph{generic} if there exists
an $\epsilon >0$ such that 
$\CC(u_A,\t')$ has the same topology as 
$\CC(u_A,\t)$ for any $\t'$ satisfying $\|\t-\t'\|<\epsilon$.
Similarly, an asymptotic contour plot $\CC_\a(\M)$ is called \emph{generic} if there exists
an $\epsilon >0$ such that 
$\CC_{\a'}(\M)$ has the same topology as 
$\CC_\a(\M)$ for any $\a'$ satisfying $\|\a-\a'\|<\epsilon$.
Here the norm $\|\cdot\|$ is the usual Euclidian norm in $\R^{m-2}$.
\end{definition}


\subsection{Soliton graphs}

The following notion of {\it soliton graph}
forgets the metric
data of the asymptotic contour plot, but preserves
the data of how line-solitons interact and which exponentials dominate.

\begin{definition}\label{soliton-graph}
Let $\CC_{\a_0}(\mathcal{M})$ be an asymptotic contour plot
with $n$ unbounded
line-solitons.
Color a trivalent
vertex black (respectively, white)
if it has a unique edge extending downwards (respectively, upwards) from it.
Label a region by $E_I$ if the dominant exponential in that region is $\Delta_I E_I$.  Label
each edge (line-soliton) by the {\it type} $[i,j]$ of that
line-soliton.
Preserve the topology of the metric graph,
but forget the metric structure.
Embed the resulting graph with bicolored vertices 
into a disk with $n$ boundary vertices,
replacing each unbounded line-soliton with an edge
that ends at a boundary vertex.
We call this labeled graph
the  \emph{soliton graph} $G_{\a_0}(\mathcal{M})$.

Abusing notation, we will often refer to the edges of $G_{\a_0}(\mathcal{M})$
as {\it line-solitons},
and use the terminology {\it unbounded line-solitons} and {\it unbounded 
regions} 
to refer to the edges and regions incident to the boundary of the disk.
\end{definition}

\begin{figure}[h]
\begin{center}
\includegraphics[height=1.8in]{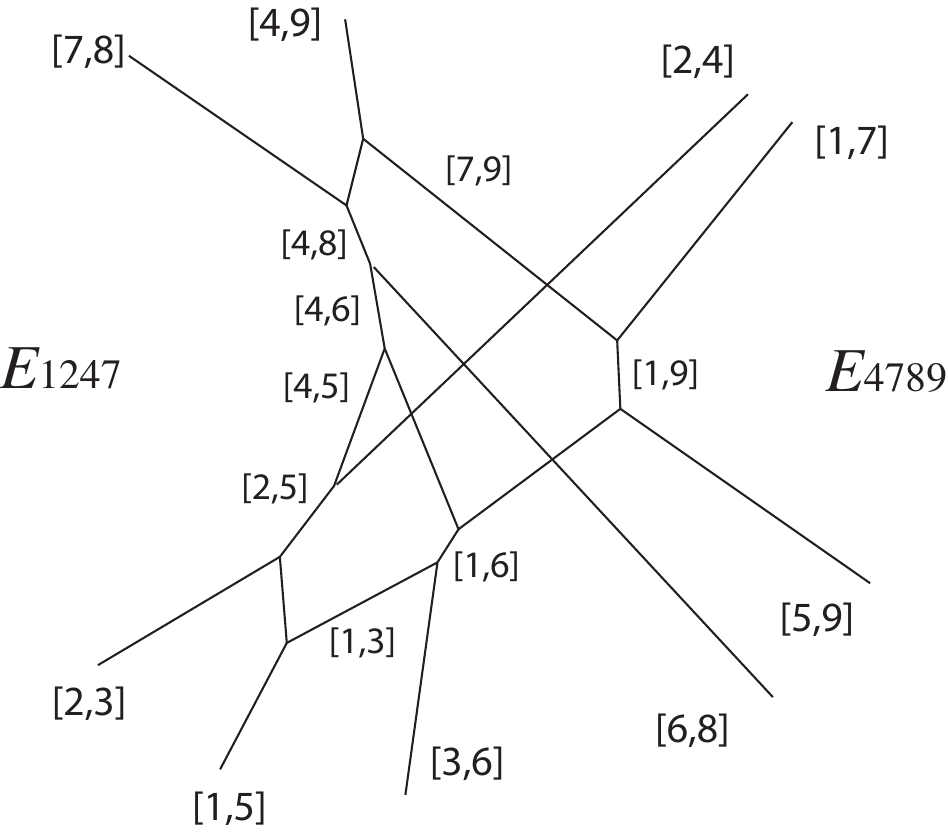}
\hskip 1cm
\includegraphics[height=1.8in]{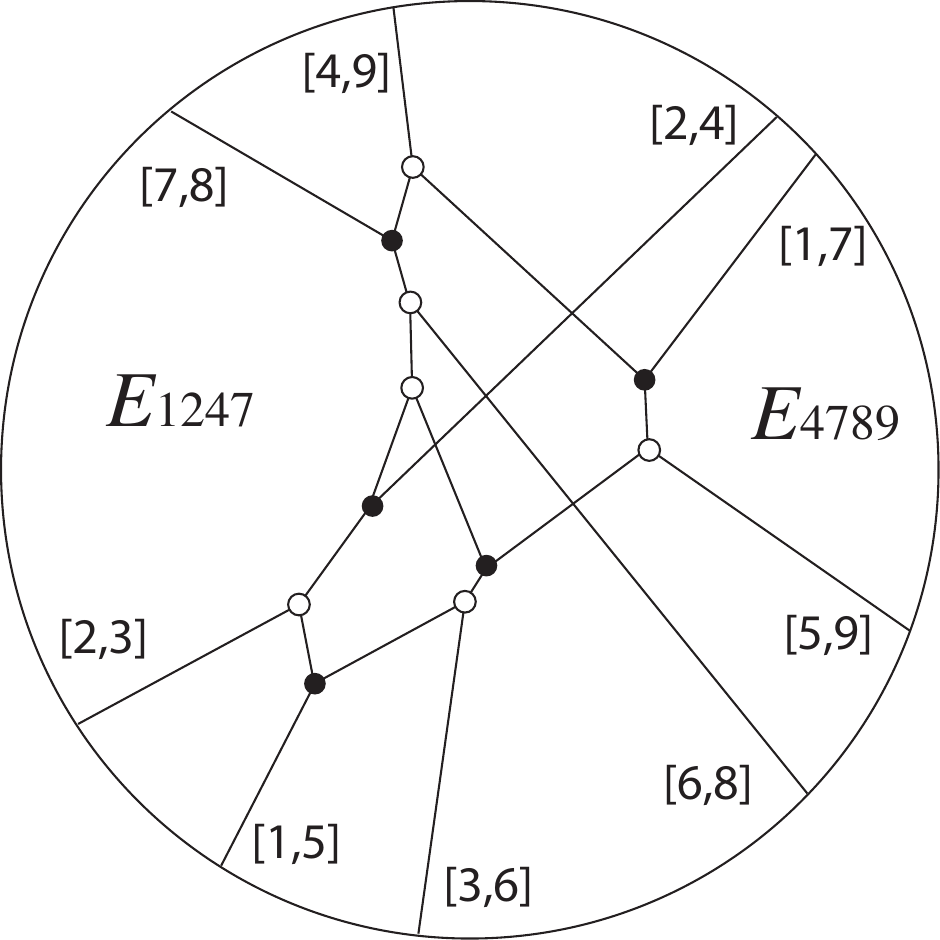}
\end{center}
\caption{
Example of an asymptotic contour plot and the soliton graph associated to
$\S_{\pi}^{tnn}$ with
$\pi=(7,4,2,9,1,3,8,6,5)$. \label{contour-soliton}}
\end{figure}
See Figure \ref{contour-soliton} for an example of
a soliton graph.
Although we have not labeled all regions or all edges,
the remaining labels can be determined using Lemma \ref{separating}.

\subsection{Resonance of line-solitons}\label{resonance}
In this section
 we explain the physical meaning of trivalent vertices in the contour plot.
It follows from Example \ref{ex:1soliton}
that a line-soliton of $[i,j]$-type has the form
\[
u=\frac{1}{2}(\kappa_i-\kappa_j)^2{\rm sech}^2\frac{1}{2}\Theta_{[i,j]}(x,y,\t),
\]
where the phase function $\Theta_{[i,j]}(x,y,\t)$ is given by
\[
\Theta_{[i,j]}(x,y,\t)=(\kappa_j-\kappa_i)x+(\kappa_j^2-\kappa_i^2)y+\sum_{s=3}^m\Omega^{(s)}_{[i,j]}(\kappa)t_s+\Theta_{[i,j]}^0(\kappa,A),
\]
with
\[
\Omega_{[i,j]}^{(s)}(\kappa):=\kappa_j^s-\kappa_i^s,\qquad\text{and}\qquad  \Theta_{[i,j]}^0(\kappa,A):=\ln\frac{\Delta_J(A)K_J}{\Delta_I(A)K_I}.
\]
In particular, the coefficients of $x,y$ and $t=t_3$ are called 
 the wavenumber-vector and the frequency,
and they are given by
\begin{equation}\label{WF}
{\bf K}_{[i,j]}:=(K_{[i,j]}^x,K_{[i,j]}^y)=(\kappa_j-\kappa_i,\kappa_j^2-\kappa_i^2),\qquad \Omega_{[i,j]}=
\kappa_j^3-\kappa_i^3.
\end{equation}

There is an algebraic relation, called the {\it dispersion relation} of the KP equation, among ${\bf K}_{[i,j]}$ and $\Omega_{[i,j]}$, which is given by
\begin{equation}\label{DispR}
D({\bf K}_{[i,j]},\Omega_{[i,j]}):=-4\Omega_{[i,j]}K^x_{[i,j]}+(K_{[i,j]}^x)^4+3(K^y_{[i,j]})^2=0.
\end{equation}
See \cite[Chapter 11.1]{Whitham} for more details.
This implies that if a plane wave of the form $\phi({\bf K}_j\cdot{\bf x}+\Omega_jt)$
is a solution of the KP equation, then ${\bf K}_j$ and $\Omega_j$ must satisfy the dispersion relation.  Note that the wavenumber-vectors and the frequency given in
  \eqref{WF} satisfy \eqref{DispR}, i.e. $D({\bf K}_j,\Omega_j)=0$.
 
 In  wave theory, if 
for two plane waves $\phi_i({\bf K}_i\cdot {\bf x}+\Omega_it)$ for $i=1$ and $2$
we have
 \[
 D({\bf K}_1+{\bf K}_2,\Omega_1+\Omega_2)=0,
 \]
 then as a result, a third wave can be generated.  Moreover, the new wave $\phi_3({\bf K}_3\cdot \mathbf{x}+\Omega_3t)$
 satisfies the so-called {\it resonant conditions},
 \[
 {\bf K}_3={\bf K}_1+{\bf K}_2,\qquad {\rm and}\qquad \Omega_3=\Omega_1+\Omega_2.
 \]
 In the KP dispersion relation, the line-solitons of 
 types $[i,j]$, $[j,\ell]$, and $[i,\ell]$ 
(here $i<j<\ell$) trivially satisfy the resonant conditions, i.e.
\begin{align} \label{balancing}
{\bf K}_{[i,\ell]}={\bf K}_{[i,j]}+{\bf K}_{[j,\ell]},\qquad {\rm and} \qquad \Omega_{[i,\ell]}=\Omega_{[i,j]}+\Omega_{[j,\ell]}.
\end{align}
The resonant relations \eqref{balancing} also hold for the
higher terms $\Omega_{[i,j]}^{(s)}$ for $s=4,\ldots,m$, i.e.
\[
\Omega_{[i,\ell]}^{(s)}=\Omega_{[i,j]}^{(s)}+\Omega_{[j,\ell]}^{(s)}\qquad\text{for}\quad s=4,\ldots,m.
\]
This means that resonant interactions arise quite naturally in the KP hierarchy, and each 3-wave resonant interaction
appears as a trivalent vertex in the contour plot.  At that trivalent vertex, 
since the slope of each soliton is given by
$\tan\Psi_{[i,j]}=\kappa_i+\kappa_j$, those three solitons appear as $[i,j],[j,l]$ and $[i,l]$
in counterclockwise order.  This condition
led  us to discover a new characterization of reduced plabic graphs,
which we describe in Section \ref{Reduced-Cluster}.
Note also that in the contour plot,
one may interpret 
equation (\ref{balancing}) as the {\it balancing condition}
for a tropical curve.

\section{Permutations and soliton asymptotics}\label{permutations}
Given a contour plot $\CC(u_A, \t_0)$ where $A$ belongs
to an irreducible positroid cell, we show that the
labels of the unbounded solitons
allow us to determine which positroid cell $A$ belongs to.
Conversely, given $A$ in the irreducible
positroid cell $S_{\pi}^{tnn}$,
we can predict the asymptotic behavior of the unbounded solitons
in $C(u_A, \t_0)$.

\begin{theorem}\label{perm-asymp}
Suppose $A$ is an element of an irreducible positroid cell in
$(Gr_{k,n})_{\geq 0}$.
Consider the contour plot $\CC(u_A, \t_0)$ for any time $\t_0$.
Then there are $k$ unbounded
line-solitons at $y\gg0$ which are labeled by
pairs $[e_r,j_r]$ with $e_r<j_r$, and there are $n-k$
unbounded line-solitons at $y\ll0$ which are labeled by pairs
$[i_r,g_r]$ with $i_r<g_r$.  We obtain a derangement in $S_n$
with $k$ excedances by setting $\pi(e_r) = j_r$ and
$\pi(g_r) = i_r$.  Moreover, $A$ must belong to
the cell $S_{\pi}^{tnn}$.
\end{theorem}

The first part of this theorem follows from work of
Biondini and Chakravarty 
\cite[Lemma 3.4 and Theorem 3.6]{BC06} (see Proposition \ref{rank-conditions} below) and  
Chakravarty and Kodama
\cite[Prop. 2.6 and 2.9]{CK1}, \cite[Theorem 5]{CK3} (see 
Theorem \ref{soliton-permutation}
below).  In particular, Chakravarty and Kodama
had already associated a derangement $\pi$ to $A$, but it was not 
clear how this $\pi$ was related to the derangement indexing the 
cell containing $A$.
Our contribution is a proof that the derangement $\pi$ is precisely
the derangement labeling the cell $S_{\pi}^{tnn}$ that
$A$ belongs to (see Proposition \ref{perm-prop})\footnote{S. Chakravarty informed us that he also
proved an equivalent proposition.}.
This fact is the first step towards establishing that various other
combinatorial objects in bijection with positroid cells
(Grassmann necklaces, plabic graphs) carry useful information about the
corresponding soliton solutions.

Given a  matrix $A$ with $n$ columns,
let $A(k,\dots,\ell)$ be the submatrix of $A$ obtained
from columns $k, k+1,\dots, \ell-1, \ell$, where the columns
are listed in the circular order
$k, k+1,\dots,n-1, n, 1, 2, \dots,k-1$.
\begin{proposition}\cite[Lemma 3.4]{BC06}\label{rank-conditions}
Let $A$ be a $k \times n$ matrix representing an element in 
an irreducible positroid cell in $(Gr_{k,n})_{\geq 0}$, and consider the contour
plot $\CC(u_A, \t_0)$ for any time $\t_0$.
Then there are $n-k$ unbounded line-solitons at $y\ll0$
and $k$ unbounded line-solitons at $y\gg0$:

There is an unbounded line-soliton of $\CC(u_A,\t_0)$ at $y\ll0$
labeled $[i,g]$ with $i<g$
if and only if
\begin{equation}\label{eq:rank}
\rank A(i,\dots,g-1) = \rank A(i+1,\dots,g) = \rank A(i,\dots,g) = 
\rank A(i+1,\dots,g-1) + 1.
\end{equation}
Moreoever, $g$ is a non-pivot column of $A$.

And there is an unbounded line-soliton of $\CC(u_A,\t_0)$ at $y\gg0$
labeled $[e,j]$ with $e<j$
if and only if
\begin{equation}\label{eq:rank2}
\rank A(j,\dots,e-1) = \rank A(j+1,\dots,e) = \rank A(j,\dots,e) = 
\rank A(j+1,\dots,e-1) + 1.
\end{equation}
Moreoever, $e$ is a pivot column of $A$.
\end{proposition}

\begin{theorem}\cite[Prop. 2.6 and 2.9]{CK1}\cite[Theorem 5]{CK3}\label{soliton-permutation}
Consider an irreducible positroid cell $S_{\mathcal M}^{tnn}$
in $(Gr_{k,n})_{\geq 0}$, and 
let $A$ be a full rank matrix representing a point in that cell.
Use the notation of Proposition \ref{rank-conditions}.
Define $\pi^{!}:=\pi^{!}(\mathcal M)$ by setting
$\pi^{!}(e):=j$ and $\pi^{!}(g):=(i)$ for each pivot $e$ and
non-pivot $g$.  Then
$\pi^{!}$ is a derangement on $n$ letters with $k$ weak excedances.
\end{theorem}

\begin{proposition}\label{perm-prop}
Consider an irreducible 
positroid cell $S_{\mathcal M}^{tnn}=S_{\pi}^{tnn}$,
where $\pi = 
\pi({\mathcal I}(\mathcal M))$.
Then 
$\pi^{!}({\mathcal M})=\pi.$ 
\end{proposition}

\begin{proof}
Consider a $k \times n$ matrix $A$ representing an element in 
$S_{\mathcal M}^{tnn}$.  Then all maximal minors of $A$ 
are non-negative, and the column indices of the non-zero minors
are the subsets in $\mathcal M$.
Let us first consider the derangement $\pi=\pi({\mathcal I}(\mathcal M))$.
Let $I_i=\{i=x_1, x_2,\dots,x_{k}\}$ be the lexicographically 
minimal minor in $\mathcal M$
with respect to the total order $i < i+1 < \dots < n < 1 < \dots < i-1$.
Then $I_{i+1}=(I_i\setminus \{i\}) \cup \{j\}$ is obtained from 
$I_i\setminus \{i\}$ by considering the column indices
in the order $i+1, i+2, \dots, n, 1, 2, \dots, i$ and greedily choosing
the earliest index $h$ such that the columns of $A$
indexed by the set $\{x_2,\dots,x_{k}\} \cup \{h\}$
are linearly independent.  
Then $\pi(h)$ is defined to be $i$.

Now consider the ranks of various submatrices of $A$ obtained
by selecting certain columns.

{\it Claim 0.}  $\rank A(i+1,\dots, h-1,h)=
  1+\rank A(i+1,\dots, h-1)$.  This claim follows from 
the way in which we chose $h$ above.

{\it Claim 1.}  $\rank A(i,i+1,\dots,h) = \rank A(i,i+1,\dots,h-1)$.
To prove this claim, we consider two cases.  Either 
$x_1 <_i h <_i x_k$ or $x_1 <_i x_k <_i h$, where $<_i$
is the total order $i<i+1<\dots<n<1<\dots<i-1$.
In the first case, the claim follows, because 
$h$ is not contained in the set $I_i$ but {\it is} contained
in $I_{i+1}$.  In the second case,
$\rank A(i,i+1,i+2,\dots,x_k) = k$, and the index set
$\{i,i+1,\dots,x_k\}$ is a strict subset of 
$\{i,i+1,\dots,h\}$, so 
$\rank A(i,\dots,h) = \rank A(i,\dots,h-1)=k$.

Now let $R = \rank A(i+1,i+2,\dots,h-1)$.  
By Claim 0, $\rank A(i+1,\dots,h) = R+1$.
Therefore we have
$\rank A(i,\dots,h) \geq \rank A(i+1,\dots,h) = R+1$.
By Claim 1, $\rank A(i,\dots,h) = \rank A(i,\dots,h-1)$,
but $\rank A(i,\dots,h-1) \leq R+1$, so $\rank A(i,\dots,h) \leq R+1$.
We now have $\rank A(i,\dots,h) = R+1$.
But also $\rank A(i,\dots,h-1)=\rank A(i,\dots,h)=R+1$.

We have just shown that 
$\rank A(i,i+1,\dots,h-1)= \rank A(i+1,\dots,h-1,h)=\rank A(i,\dots,h)=
  \rank A(i+1,\dots,h-1) +1$.
Comparing these rank conditions to 
either part of 
Proposition \ref{rank-conditions}, and using
Theorem \ref{soliton-permutation}, 
we see that $\pi^{!}(h)=i$.
This shows that $\pi^{!}$ and $\pi$ coincide.
\end{proof}

\begin{remark}
Proposition \ref{perm-prop} is closely related to results on cyclic rank
matrices from \cite{KLS}.
\end{remark}

We now give a concrete algorithm for writing down the asymptotics 
of the soliton solutions of the KP equation.

\begin{figure}
\begin{center}
\includegraphics[height=4.5cm]{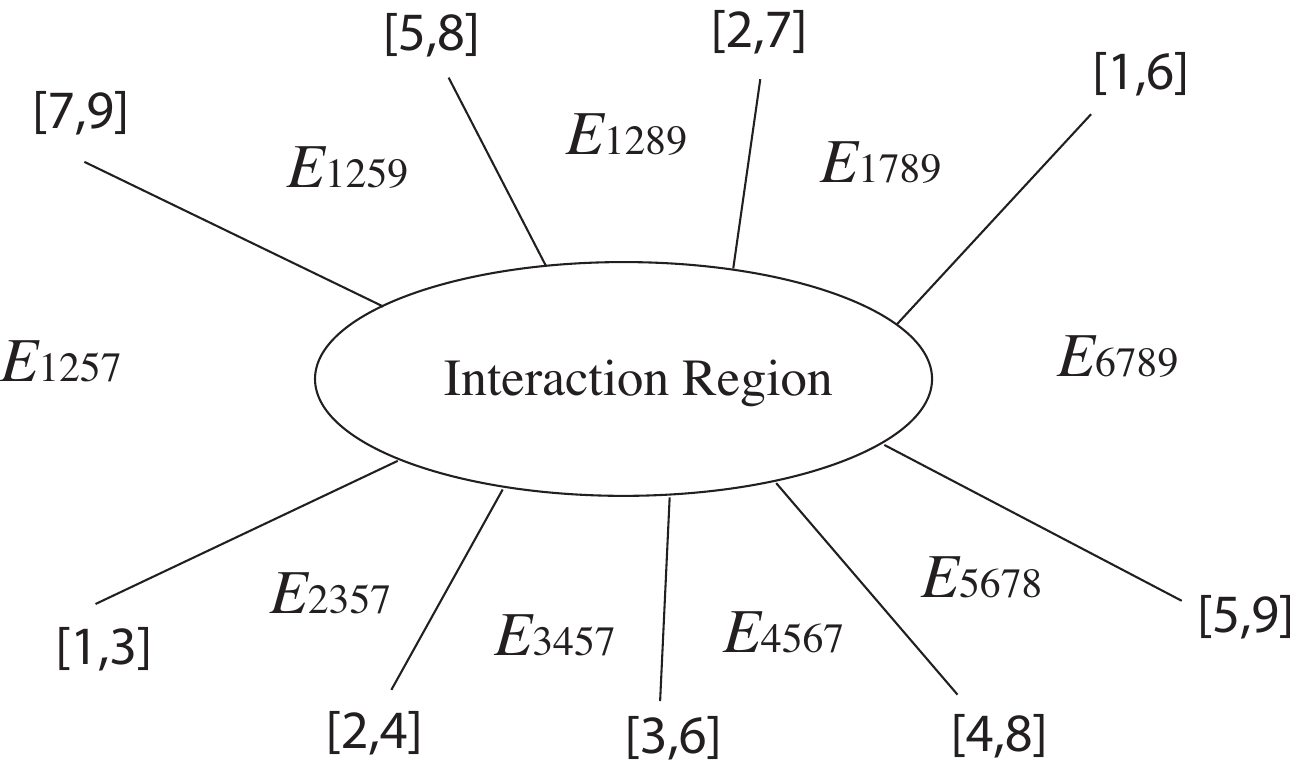}
\end{center}
\caption{Asymptotic line-solitons for $\pi=(6,7,1,2,8,3,9,4,5)$.
Each $E_{ijkl}$ shows the dominant exponential in this region.
\label{671283945}}
\end{figure}

\begin{theorem}\label{algo}
Fix real generic parameters $\kappa_1<\dots < \kappa_n$.  
Let $A$ be a point in an irreducible 
positroid cell
$\SSS_{\pi}^{tnn}$ in $(Gr_{k,n})_{\geq 0}$.
(So $\pi$ has $k$ excedances.)  
For any $\t_0$, the asymptotic behavior of the contour plot
 $\mathcal{C}(u_A,\t_0)$
--  its unbounded line-solitons 
and the dominant exponentials in its unbounded regions -- can  be
read off from $\pi$ as follows.
\begin{itemize}
\item For $y \gg0$, there is an unbounded line-soliton which 
we label $[i,\pi(i)]$ for each
excedance $\pi(i)>i$.  
From left to right,
list these solitons in decreasing order 
of the quantity
$\kappa_i + \kappa_{\pi(i)}$.
\item For $y\ll 0$, there is an unbounded line-soliton  which we label
$[\pi(j),j]$ for each nonexcedance $\pi(j)<j$.  From left to right,
list these solitons in increasing order 
$\kappa_j + \kappa_{\pi(j)}$.
\item Label the unbounded region for $x\ll 0$ with 
the exponential $E_{i_1,\dots,i_k}$,
where $i_1,\dots,i_k$ are the excedance positions of $\pi$.
\item Use Lemma \ref{separating} to label the remaining
unbounded regions of the contour plot.
\end{itemize}
\end{theorem}

\begin{proof}
The fact that the set of unbounded line-solitons are specified by 
the derangement $\pi$ 
comes from \cite[Theorem 5, page 125]{CK3} and 
\cite[Corollary 1, page 124]{CK3}.  It then follows 
from Remark \ref{slope} that for sufficiently
large $y$ (respectively, sufficiently small $y$), these solitons are ordered 
from left to right by decreasing (respectively, increasing)
order of their slopes $\kappa_i+\kappa_j$.  
\end{proof}

\begin{example}\label{ex4}
Consider the positroid cell corresponding to 
$(6,7,1,2,8,3,9,4,5)\in S_9$.  The algorithm of Theorem 
\ref{algo} gives rise to the picture in Figure \ref{671283945}.
Note that if one reads the dominant exponentials in counterclockwise order,
starting from the region at the left, then one recovers exactly
the Grassmann necklace from Examples \ref{ex1} and \ref{ex3}.
This correspondence will be generalized in 
Theorem \ref{necklace-soliton}.
\end{example}


\section{Grassmann necklaces and soliton asymptotics}\label{sec:necklace}

One particularly nice class of positroid cells is the 
{\it TP} or 
{\it totally positive Schubert cells}. A TP Schubert cell is 
a positroid cell $S_L^{tnn}$ which comes from a $\Le$-diagram $L$
such that all boxes of $L$ contain a $+$.  Note that 
the intersection of a usual Schubert cell with $\Grkn$ is a union
of positroid cells, of which the one with greatest dimension
is the TP Schubert cell.
When $A$ lies in a TP Schubert cell
$S_{\pi}^{tnn}$, 
we can make another link between the 
soliton solution $u_A(x,y,\t)$  and the combinatorics of 
$(Gr_{k,n})_{\geq 0}$.  Namely, 
the dominant exponentials labeling the unbounded regions of 
the contour plot $\CC(u_A,\ttt)$ form 
the Grassmann necklace associated to $S_{\pi}^{tnn}$.

It is easy to verify the following lemma.
\begin{lemma}\label{Schubert-derangement}
A positroid cell $S_{\pi}^{tnn} = S_{L}^{tnn}$ of $\Grkn$
is a TP Schubert cell if and only if  
the following
condition holds:\\
If 
$i_1<i_2<\dots<i_k$ 
and $j_1<j_2<\dots<j_{n-k}$ 
are the positions of the excedances
and nonexcedances, respectively, of $\pi$,
then $\pi(i_1)=n-k+1$, $\pi(i_2) = n-k+2$, \dots, $\pi(i_k) = n$ and 
$\pi(j_1)=1$, $\pi(j_2)=2$, \dots, $\pi(j_{n-k})=n-k$.
\end{lemma}

We have the following result.
\begin{theorem}\label{necklace-soliton}
Let $A$ be an element of a TP Schubert cell $\S_{\pi}^{tnn}$, and 
consider the contour plot $\CC(u_A,\t_0)$ for an arbitrary time
$\t_0$.  Let the index sets of the dominant exponentials of 
the unbounded regions of $\CC(u_A,\t_0)$ be denoted $R_1, \dots, R_n$,
where $R_1$ labels the region at $x\ll 0$, and $R_2,\dots,R_n$
label the regions in the counterclockwise direction from $R_1$.
Then $(R_1,\dots, R_n)$ is a Grassmann necklace $\mathcal I$,
and $\pi(\mathcal I)= \pi$.
\end{theorem}

Theorem \ref{necklace-soliton} is illustrated in Example \ref{ex4}.  See also Figure \ref{671283945}.

\begin{remark}
Theorem \ref{necklace-soliton} does not hold if we replace ``TP
Schubert cell" by ``positroid cell."  For example, 
the Grassmann necklace
associated to the derangement $\pi=(4,3,1,2)$ is $(12, 23, 34, 24)$.
However, if $\kappa_1=0, \kappa_2=1, \kappa_3=1.5, \kappa_4=1.75$, then the 
corresponding sequence of dominant exponentials labeling the unbounded
regions of any contour plot coming from the cell $\S_{\pi}^{tnn}$ is 
$(12, 23, 34, 13)$.
\end{remark}

\begin{remark}\label{invert}
To recover a Grassmann necklace $\mathcal I=(I_1,\dots,I_n)$ from a derangement
$\pi\in S_n$
(inverting the procedure of Lemma \ref{Postnikov-permutation}), 
we do the following:
\begin{itemize}
\item Set $I_1 = \{i_1,\dots,i_k\}$, the positions of the excedances of $\pi$.
\item For each $r\geq 1$, set $I_{r+1} = (I_r \setminus \{r\}) \cup \{\pi^{-1}(r)\}$.
\end{itemize}
\end{remark}

We now prove  Theorem \ref{necklace-soliton}.

\begin{proof}
Let $i_1<\dots<i_k$ be the positions of the excedances of $\pi$,
and let $j_1 < \dots < j_{n-k}$ be the positions of the nonexcedances.
By Lemma \ref{Schubert-derangement},
we have that 
$\pi(i_1) < \pi(i_2) < \dots < \pi(i_k)$, and $\pi(j_1) < \pi(j_2) < \dots < \pi(j_{n-k})$.
Define the partial order $\prec$ on pairs $(i,j)$ of integers in 
$\{1,2,\dots,n\}$ by setting 
$(i,j) \prec (i',j')$ if and only if $\kappa_i + \kappa_j <
\kappa_{i'} + \kappa_{j'}$.
Then the 
condition that $\kappa_1<\kappa_2<\dots < \kappa_n$ implies that
$(i_1,\pi(i_1)) \prec (i_2,\pi(i_2)) \prec \dots \prec (i_k,\pi(i_k))$ and 
$(j_1,\pi(j_1)) \prec \dots \prec (j_{n-k},\pi(j_{n-k}))$.
Using Theorem \ref{algo}, the asymptotic directions of the contour graph of the soliton solution
are as in Figure \ref{contour1}. 
\begin{figure}
\begin{center}
\includegraphics[height=4.5cm]{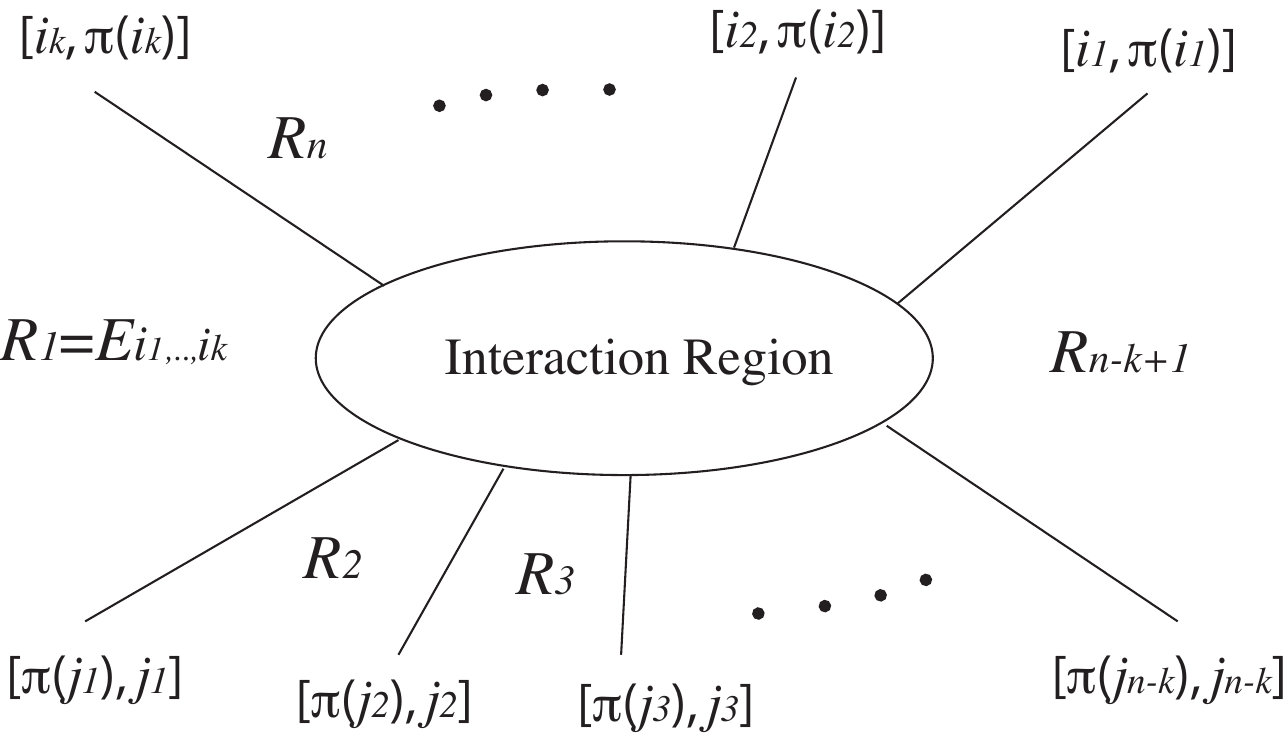}
\end{center}
\caption{Grassmann necklace $(R_1,R_2,\ldots,R_n)$ in a contour plot associated to the TP Schubert cell $\S_{\pi}^{tnn}$.}
\label{contour1}
\end{figure}

From the conditions on our permutation, we must have
$\pi(j_1)=1, \pi(j_2)=2,\dots,\pi(j_{n-k})=n-k$, and also
$\pi(i_1)=n-k+1,\pi(i_2)=n-k+2,\dots,\pi(i_k)=n$.
Therefore the unbounded line-solitons  of the contour plot of the soliton solution
are labeled as $[i_l,n-k+l]$ for $l=1,\ldots,k$ and $[m,j_m]$ for $m=1,\ldots,n-k$.

{\it Claim 1}. We claim that if 
$R_1,\dots,R_n$ are the index sets of the dominant exponentials in the unbounded
regions as in Figure \ref{contour1}, then 
$R_{\ell}$ contains $\ell$.  Therefore by Lemma 
\ref{separating}, $R_{\ell+1}$ is obtained from
$R_{\ell}$ by removing $\ell$ and adding one more index not already in $R_{\ell}$. 
By Remark \ref{invert}, Claim 1 implies 
$(R_1,\dots,R_n)$ is the Grassmann necklace associated to $\pi$,
and therefore implies Theorem \ref{necklace-soliton}.

We first prove Claim 1 for $\ell \leq n-k+1$.  Clearly $1\in R_1$,
since $1$ is always the position of an excedance of a derangement.  Suppose
by induction that the claim is true up through $\ell-1$.
Then 
$$R_{\ell} = (((((\{i_1,\dots,i_k\} \cup \{j_1\})\setminus \{1\})
     \cup \{j_2\}) \setminus \{2\}) \dots \cup \{j_{\ell_1}\}) \setminus \{\ell_1\}.$$
Suppose that $\ell \notin R_{\ell}$.  In steps $1$ through $\ell-1$, we have only
removed the numbers $\{1,2,\dots,\ell-1\}$, and so 
$\ell \notin \{i_1,\dots,i_k\}$.  And we have only added the numbers
$\{j_1,\dots,j_{\ell-1}\}$, and so $\ell \notin \{j_1,\dots,j_{\ell-1}\}$.
Since $\ell\notin \{i_1,\dots,i_k\}$, we have $\pi(\ell) < \ell$,
and so $\ell \in \{j_1,\dots,j_{n-k}\}$.  Since $\ell\notin \{j_1,\dots,j_{\ell-1}\}$,
we have $\ell \in \{j_{\ell},j_{\ell+1},\dots,j_{n-k}\}$.  But
$1 < j_1< j_2< \dots<j_{n-k}$ and so each element in 
$\{j_{\ell},j_{\ell+1},\dots,j_{n-k}\}$ is greater than $\ell$.
This is a contradiction.

{\it Claim 2}. $R_{n-k+1} = \{\pi(i_1),\dots,\pi(i_k)\}$.
Note that since Claim 1 is true for $\ell \leq n-k+1$, $R_{n-k+1}$  contains an index for
each excedance position $i_r$ such that $\pi^{-1}(i_r)<i_r<\pi(i_r)$.
(These are the elements of $R_1$ that remain in each $R_2,R_3,\dots,R_{n-k+1}$.)
$R_{n-k+1}$ also contains any nonexcedance position $j_r$ as long as it is not
the case that $j_r=\pi^{-1}(j_s)$ for some $s$.  That is,
$R_{n-k+1}$ contains any $j_r$ such that $\pi^{-1}(j_r) < j_r$.
Therefore we see that $R_{n-k+1}$ is equal to the set of values that $\pi$
takes at the excedance positions of $\pi$.  This proves Claim 2.

We now prove Claim 1 for $\ell > n-k+1$.  Again we use induction on $\ell$.
The claim is true for $n-k+1$.  Suppose that $\ell\notin R_{\ell}$ but 
Claim 1 is true for smaller $\ell$.  Certainly $\ell \in R_{n-k+1}$.
So $\ell\notin R_{\ell}$ means that $\ell$ must have been removed at some
earlier step -- say step $r$, for $n-k+1 \leq r < \ell$.  But the numbers
removed at these steps were precisely the numbers $n-k+1, n-k+2,\dots,\ell-1$.
This is a contradiction.
This finishes the proof of Claim 1 and hence of Theorem \ref{necklace-soliton}.
\end{proof}

\section{Soliton graphs are generalized plabic graphs}\label{sec:sol=plabic}

In this section we will show that we can think of soliton 
graphs as {\it generalized plabic graphs}.  
More precisely, we will
associate a  generalized plabic graph
$Pl(C)$ to each soliton graph $C$.  
We then show that from $Pl(C)$ -- whose only labels are on the boundary
vertices -- we can recover the labels of the line-solitons and 
dominant exponentials of $C$.
The upshot is that all edge and region labels
of a soliton graph $C$ may be reconstructed from a labeling of each
boundary vertex of $C$ by an integer.

\begin{definition}
A \emph{generalized plabic graph\/}
is an undirected graph $G$ drawn inside a disk
with $n$ \emph{boundary vertices\/} labeled $1,\dots,n$ placed 
in {\it any} order
around the boundary of the disk, such that each boundary vertex $i$ 
is incident
to a single edge. Each internal vertex is colored black or white,
and edges are allowed to cross each 
other in an $X$-crossing  (which is not considered to be a vertex). 
\end{definition}

\begin{definition}\label{soliton2plabic}
Fix an irreducible cell $\S_{\pi}^{tnn}$ of $(Gr_{k,n})_{\geq 0}$.
To each soliton graph $C$ coming from that cell we associate 
a generalized plabic graph
$Pl(C)$ by:
\begin{itemize}
\item labeling the boundary vertex 
incident to the edge $\{i,\pi_i\}$ 
by $\pi_i=\pi(i)$;
\item forgetting the labels of all edges
and regions.
\end{itemize}
\end{definition}
See Figure \ref{contour-soliton2} for a soliton graph $C$
(the same one from Figure \ref{contour-soliton}) together with the 
corresponding generalized plabic graph $Pl(C)$.

\begin{remark}\label{Schubert-labeling}
When $\pi$ indexes a TP Schubert cell 
in $(Gr_{k,n})_{\geq 0}$, the boundary vertices will be
labeled by $1,2,\dots,n$ in counterclockwise order, with 
$1,2,\dots,n-k$ labeling the boundary vertices
corresponding to the $y\ll 0$ part of the soliton graph.
\end{remark}

We now generalize the notion 
of trip from \cite[Section 13]{Postnikov}.

\begin{definition}\label{gen:trip}
Given a generalized plabic graph $G$,
the \emph{trip} $T_i$ is the directed path which starts at the boundary vertex 
$i$, and follows the ``rules of the road": it turns right at a 
black vertex,  left at a white vertex, and goes straight through the 
$X$-crossings. Note that $T_i$  will also 
end at a boundary vertex.  The \emph{trip permutation} $\pi_G$
is the permutation such that $\pi_G(i)=j$ whenever $T_i$
ends at $j$. 
\end{definition}

We use the trips to label the  edges and regions
of each generalized plabic graph.

\begin{definition}\label{labels}
Given a generalized plabic graph $G$ with $n$ boundary vertices, 
start at each boundary
vertex $i$ and label every edge along trip $T_i$ with $i$.
Such a trip divides the disk containing $G$ into two parts: 
the part to the left of $T_i$, and the part to the right.
Place an $i$ in every region which is to the left of $T_i$.
After repeating this procedure for each boundary vertex,
each edge will be labeled by up to two numbers (between $1$ and $n$),
and each region will be labeled by a collection of numbers.
Two regions separated by an edge labeled by both $i$ and $j$ will have 
region labels $S$ and $(S\setminus \{i\}) \cup \{j\}$.
When an edge is labeled by two numbers $i<j$, we write $[i,j]$
on that edge, or $\{i,j\}$ or $\{j,i\}$ if we do not wish to specify
the order of $i$ and $j$.  
\end{definition}

\begin{theorem}\label{soliton-plabic}
Consider a soliton graph $C$ coming from 
an irreducible positroid cell 
$\S_{\pi}^{tnn}$.
Then the trip permutation associated to $Pl(C)$ is $\pi$,
and by labeling edges and regions of $Pl(C)$ according
to Definition \ref{labels}, we will recover the original
labels in $C$.
\end{theorem}

\begin{figure}[t]
\begin{center}
\includegraphics[height=1.7in]{Figure4bP}
\hskip 1.5cm
\includegraphics[height=1.8in]{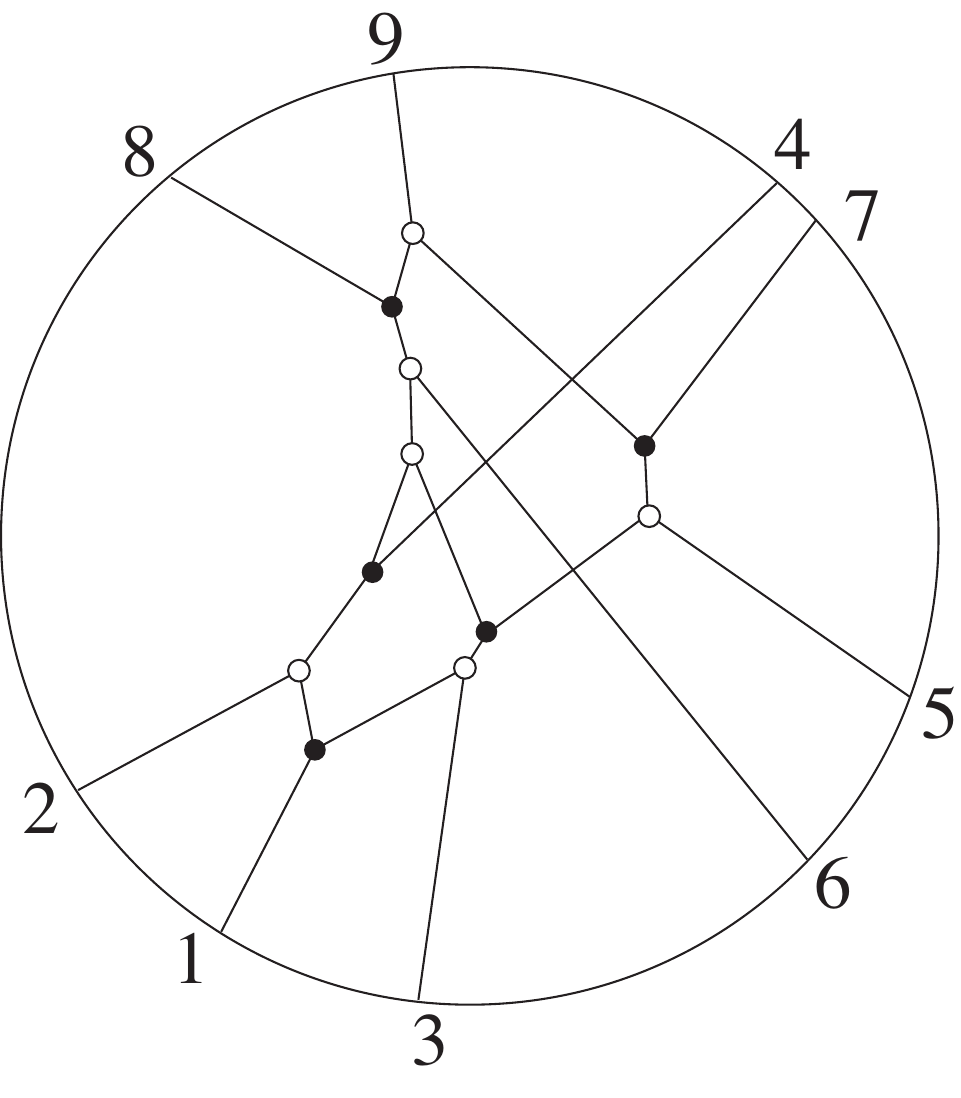}
\end{center}
\caption{A soliton graph $C$ and  generalized plabic graph $Pl(C)$
for $\pi=(7,4,2,9,1,3,8,6,5)$.}
\label{contour-soliton2}
\end{figure}

We invite the reader to verify Theorem \ref{soliton-plabic} for 
the graphs in Figure \ref{contour-soliton2}.

\begin{remark} By Theorem \ref{soliton-plabic}, we can identify 
each soliton graph $C$ with its generalized plabic graph $Pl(C)$.
From now on, we will often ignore the labels of edges and regions
of a soliton graph, and simply record the labels on boundary vertices.
\end{remark}

In the  proof below, 
we will sometimes refer to the contour plot
from which the soliton graph came; 
it is useful to think about whether edges are directed
up or down.


\begin{proof}
We begin by analyzing the edge labels around a trivalent vertex
in a soliton graph.  They must have edge labels 
$[i,j]$, $[i,m]$, and $[j,m]$ in some order, where without loss
of generality $i<j<m$.  
Recall that the slope of a line-soliton labeled $[i,j]$
is $\kappa_i+\kappa_j$. Also recall that we fixed 
$\kappa_1<\kappa_2<\dots<\kappa_n$.  Therefore
we know that the slopes of these three line-solitons are ordered by 
$\kappa_i+\kappa_j < \kappa_i+\kappa_m < \kappa_j+\kappa_m$.  
It follows that a trivalent vertex in the contour plot
with a unique edge directed down (respectively, up) from the vertex must have
line-solitons labeled as in the left (respectively, right) of 
Figure \ref{blackwhite}.  
\begin{figure}[h]
\centering
\includegraphics[height=.8in]{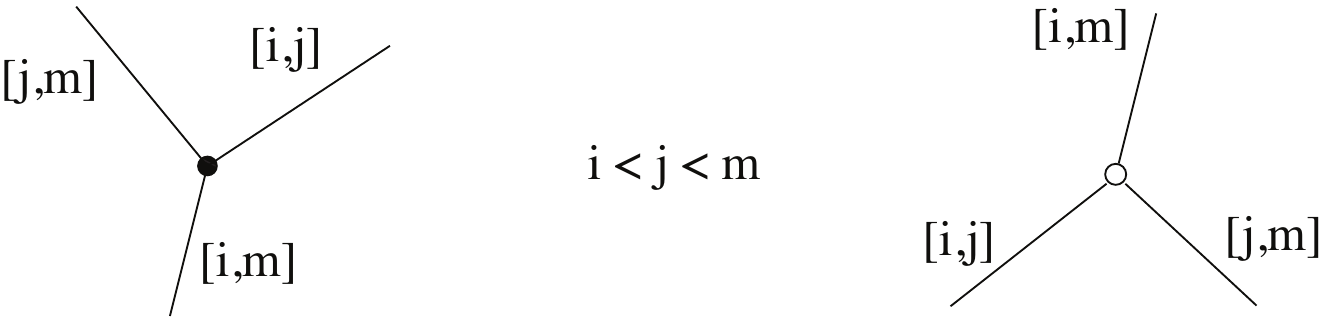}
\caption{Contour plots of resonant interactions of three line-solitons}
\label{blackwhite}
\end{figure}

We now fix $r$ between $1$ and $n$,
and analyze the set of all edges in the soliton graph $C$
whose label contains an $r$.  We aim to show that 
this set of edges is a {\it trip}. 

If $r$ is an excedance value of $\pi$, then we know from 
Theorem \ref{algo} 
that there is an edge incident to the boundary of $C$
which is labeled $[\ell,r]$, where $\ell<r$.
This is an unbounded edge going to $y\to \infty$ in the contour plot.
And if $r$ is a nonexcedance value, there is 
an edge incident to the boundary of $C$
which is labeled $[r,\ell]$ where $r<\ell$.
This is an unbounded edge going to $y \to -\infty$ in the contour plot.

Considering Figure \ref{blackwhite} and 
Definition \ref{soliton2plabic}, it is clear
that the set of all edges containing an $r$ in $C$
will be a path between boundary vertices
$r$ and $\pi_r$ in $Pl(C)$.
We call this the {\it soliton path}.

We now claim that if we start at vertex $r$
and follow the soliton path to vertex $\pi_r$,
then the path will have the following property:
the path travels {\it down} along an edge
with labels $q$ and $r$ if and only if $q<r$, 
and the path travels {\it up} along an edge
with labels $q$ and $r$ if and only if $q>r$.

This claim is clearly true for the first edge of 
each soliton path. 
Now we just need
to check that the claim remains true as we pass 
through black and white vertices.

Suppose that we are traveling down along an edge
with labels $i$ and $r$ where $i<r$, and we 
get to a white vertex.  Then, looking at the 
right side of Figure \ref{blackwhite},
we must have $m=r$, so the next edge that 
we traverse must be the edge $[j,m]$ in the figure
(that is, $[j,r]$).  Note that we will continue
to go down along an edge with labels $j$ and $r$,
with $j<r$.

Suppose that we are traveling down along an edge
with labels $q$ and $r$ where $q<r$, and we 
get to a black vertex.  Then, looking at the 
left side of Figure \ref{blackwhite}, there are two 
possibilities.  Either we are traveling down 
along the left edge (labeled $[j,m]$ in the figure,
so that $j=q$ and $m=r$), or we are traveling
down along the right edge (labeled $[i,j]$ in the figure,
so that $i=q$ and $j=r$).  In the first case,
the next edge we traverse will be the edge labeled
$[i,m]$ in the figure, i.e. $[i,r]$, so we will continue
to go down along an edge with labels $i$ and $r$,
with $i<r$.  In the second case, the next edge we
traverse will be the edge labeled $[j,m]$ in the figure,
i.e. $[r,m]$.  So in this case, our next edge in the path
will go {\it up} along an edge with labels $[r,m]$,
where $m>r$. In all cases, the claim continues to hold. 

There are also three cases to analyze if we go {\it up} 
along an edge.  These three cases are completely analogous.
Therefore the claim is true by induction.

Finally we note that in all of the above cases,
every sequence of edges in the soliton path obeys the 
``rules of the road".
This shows that the soliton paths agree with the trips,
completing the proof of Theorem \ref{soliton-plabic}.
\end{proof}

\section{A construction for asymptotic contour plots}
\label{plabic-soliton}

In this section we will explicitly compute the asymptotic
contour plots $\CC_{\pm}(\mathcal{M})$. That is, 
we have the scaled coordinates $(\bar{x},\bar{y},\a)$ with
$\a=(\pm1,0,\ldots,0)$.
In particular, we will 
provide an algorithm that constructs the associated soliton graphs,
and we will give coordinates for all the trivalent vertices in the $\bar{x}\bar{y}$-plane, 
which then allows one to completely describe the asymptotic contour plot.
Most of this section will be devoted to the case when $a_3=-1$ (i.e.
$t=t_3\to -\infty)$, and then we will explain how the same
ideas can be applied to the case $a_3=1$ (i.e. $t=t_3\to \infty$).

Since we consider the coordinates $(\bar{x},\bar{y},\a)=(\bar{x},\bar{y},-1,0,\ldots,0)$, we first define
 \[
\phi_i(\bar{x}, \bar{y}) :=\theta_i(x,y,-1,0,\ldots,0)= \kappa_i \bar{x} + 
\kappa_i^2 \bar{y} - \kappa_i^3.
\] 
Then
from Definition \ref{contourplot-infinity}, 
the asymptotic contour plot
$\CC_{-}(\mathcal{M})$ is defined
to be the locus in $\R^2$ where
\begin{equation*}
\underset{J\in\mathcal{M}}\max \left\{
                     \sum_{i=1}^k \phi_{j_i}(\bar{x},\bar{y}) \right\}
\end{equation*}
is not linear.  

To compute $\mathcal{C}_+(\mathcal{M})$, we need to work with the functions
$\phi'_i(\bar{x}, \bar{y}) :=\theta_i(x,y,+1,0,\ldots,0)= 
\kappa_i \bar{x} + 
\kappa_i^2 \bar{y} + \kappa_i^3$ 
instead of $\phi_i$.  Note that 
$\kappa_i \bar{x} + 
\kappa_i^2 \bar{y} + \kappa_i^3$ is maximized if and only if
$\kappa_i (-\bar{x}) + 
\kappa_i^2 (-\bar{y}) - \kappa_i^3$ is minimized.
Therefore
$\mathcal{C}_+(\mathcal{M})$ can be computed
as the $180^{\circ}$ rotation of the locus where 
\[
\underset{J\in\mathcal{M}}\min \left\{
                     \sum_{i=1}^k \phi_{j_i}(\bar{x},\bar{y}) \right\}
\]
is not linear.

\begin{definition}\label{def:v}
For $1 \leq i < j \leq n$, let $L_{ij}$ be the line in the $\bar{x}\bar{y}$-plane
where $\phi_i(\bar{x},\bar{y})=\phi_j(\bar{x},\bar{y})$.  
And let  $v_{i,\ell,m}$  be the point  where 
$\phi_i(\bar{x},\bar{y}) = 
\phi_{\ell}(\bar{x},\bar{y}) = 
\phi_m(\bar{x},\bar{y}).$
\end{definition}

The following lemma is easy to check.
\begin{lemma}\label{lem:v}
$L_{ij}$ has the equation
$$\bar{x} + (\kappa_i + \kappa_j) \bar{y} - (\kappa_i^2+\kappa_i \kappa_j + \kappa_j^2) = 0,$$ 
and the points $v_{i,\ell,m}$ have coordinates 
$$v_{i,\ell,m} = (-(\kappa_i \kappa_{\ell} + \kappa_i \kappa_m +
\kappa_{\ell} \kappa_m),\, \kappa_i+\kappa_{\ell}+\kappa_m) \in \R^2.$$
\end{lemma}

Some of the points $v_{i,\ell,m}\in \mathbb{R}^2$ will be trivalent vertices in the contour plots we construct; such a point corresponds to the resonant interaction of three line-solitons of types $[i,\ell]$, $[\ell,m]$ and $[i,m]$
(see Theorem \ref{t<<0} below).

\subsection{Main results on 
 $\mathcal{C}_{\pm}(\mathcal{M})$ and their soliton graphs}

Consider a positroid cell $S_{\mathcal{M}}^{tnn} = 
S_L^{tnn}$ where $L$ is the $\Le$-diagram indexing the cell.
We will explain how to use $L$ to construct a generalized
plabic graph $G_-(L)$.

\begin{algorithm}  \label{LeToPlabic} From a $\Le$-diagram $L$ to 
the graph $G_-(L)$:
\begin{enumerate}
\item Start with a $\Le$-diagram $L$ contained in a $k\times (n-k)$
rectangle, and use the construction of Definition \ref{Le2permutation}
to replace $0$'s and $+$'s by crosses and elbows,
and to label its border.
\item Add an edge, and one white and one black vertex  
to each elbow, as shown
in the upper right of Figure \ref{LePlabic}.  Forget the labels
of the southeast border.  If there is an endpoint of a pipe on the east or south border whose pipe
starts by going straight, then erase the straight portion preceding the first elbow.
\item Forget any degree $2$ vertices, and forget
any edges of the graph which end 
at the southeast border of the diagram.
Denote the resulting
graph  $G_-(L)$.  
\item After embedding the graph in a disk 
with $n$ boundary vertices
(this is just a cosmetic change which we sometimes omit),
we obtain a generalized plabic graph,
which we also denote $G_-(L)$.
If desired, stretch and rotate $G_-(L)$ so that the boundary vertices
at the west side of the diagram are at the north instead.
\end{enumerate}
\end{algorithm}

Figure \ref{LePlabic} illustrates the steps of Algorithm \ref{LeToPlabic}, 
starting
from the $\Le$-diagram of the positroid 
cell  $\S_{\pi}^{tnn}$ where $\pi=(7,4,2,9,1,3,8,6,5)$.  
\begin{figure}[h]
\centering
\includegraphics[height=3in]{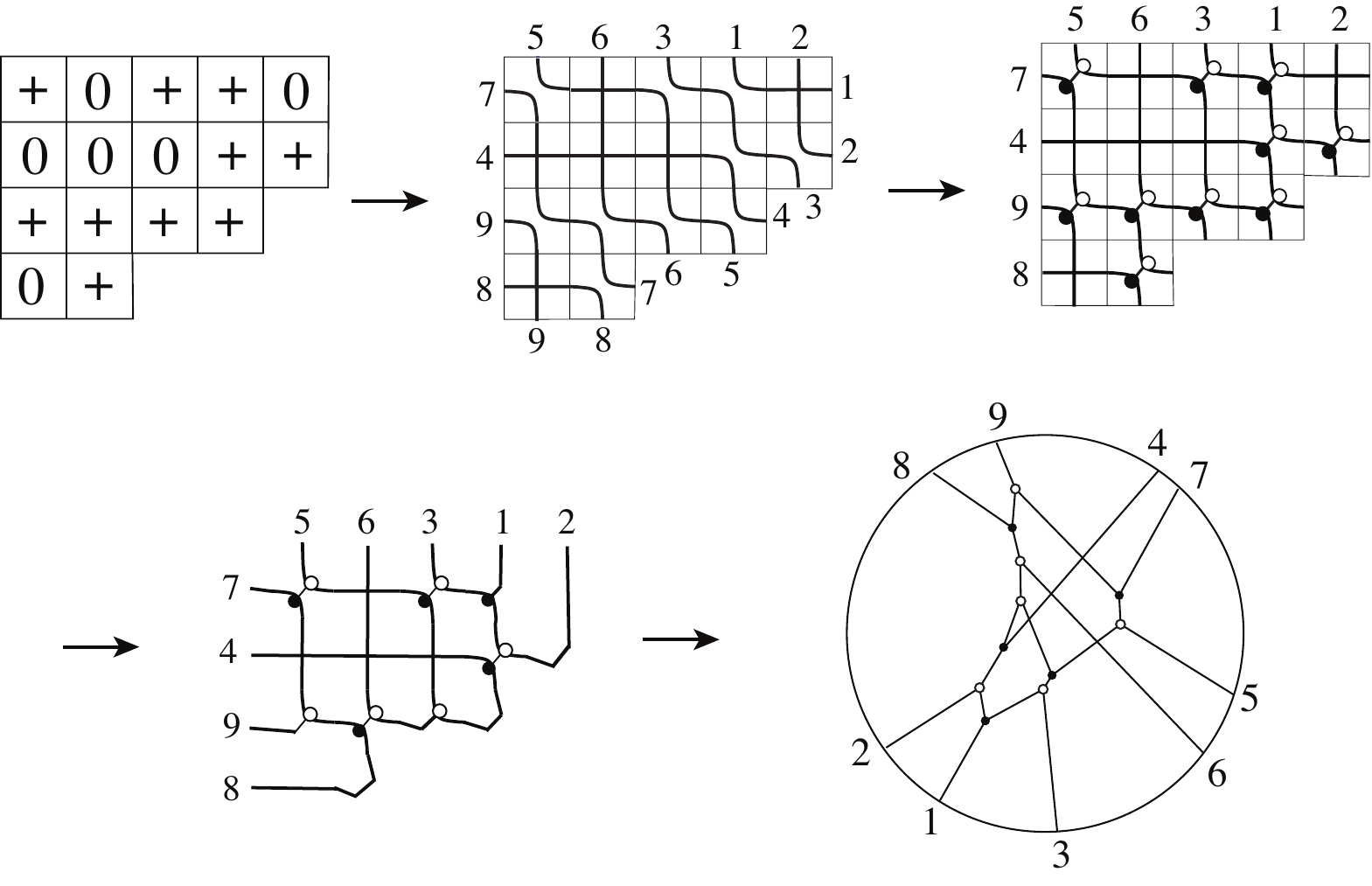}
\caption{Construction of $G_-(L)$ where $\pi(L) = 
(7,4,2,9,1,3,8,6,5)$. The top left figure is $L$.
\label{LePlabic}}
\end{figure}
After labeling the edges according to the rules of the road,
we will produce the graph from Figure 
\ref{contour-soliton}.
\begin{remark}
If every box of $L$ contains a $+$ (that is,
$\S_{L}^{tnn}$ is a TP Schubert cell),
then $G_-(L)$ will not contain
any $X$-crossings. 
\end{remark}

The following is the main result of this section.  The proof will be given in the next subsection.
\begin{theorem}\label{t<<0}
Choose a positroid cell $S_{\mathcal{M}}^{tnn} = 
S_{L}^{tnn} = S_{\pi}^{tnn}.$
Use Algorithm \ref{LeToPlabic} 
to obtain $G_-(L)$.   
Then $G_-(L)$ has trip permutation $\pi$, and we can use it to explicitly
construct $\CC_{-}(\mathcal{M})$ as follows.
Label the  edges of $G_-(L)$ 
according to the rules of the road.  Label each trivalent
vertex incident to solitons $[i,\ell]$, $[i,m]$, and $[\ell,m]$
by $x_{i,\ell,m}$ and give that point the coordinates
$(\bar{x},\bar{y}) = (-(\kappa_i \kappa_{\ell}+\kappa_i \kappa_m + \kappa_{\ell} \kappa_m),\,
\kappa_i+\kappa_{\ell}+\kappa_m)$.
Place each unbounded line-soliton of type $[i,j]$ so that it has
slope $\kappa_i + \kappa_j$.
(Each bounded line-soliton of type $[i,j]$ will automatically have slope
$\kappa_i + \kappa_j$.)  
\end{theorem}

\begin{remark}\label{rem:slide}
Although Theorem \ref{t<<0} dictates which collections of line-solitons
meet at a trivalent vertex, it does not determine which pairs of line-solitons
form an $X$-crossing.  Which line-solitons form an $X$-crossing 
is determined by the parameters $(\kappa_1,\dots,\kappa_n)$.  See
Figure \ref{3realizations} for three contour plots based on 
three different choices of $(\kappa_1,\dots,\kappa_n)$.  All of them
can be constructed using the graph $G_-(L)$ from Figure \ref{LePlabic},
together with Theorem \ref{t<<0}.
\end{remark}
\begin{figure}[h]
\centering
\includegraphics[height=1.6in]{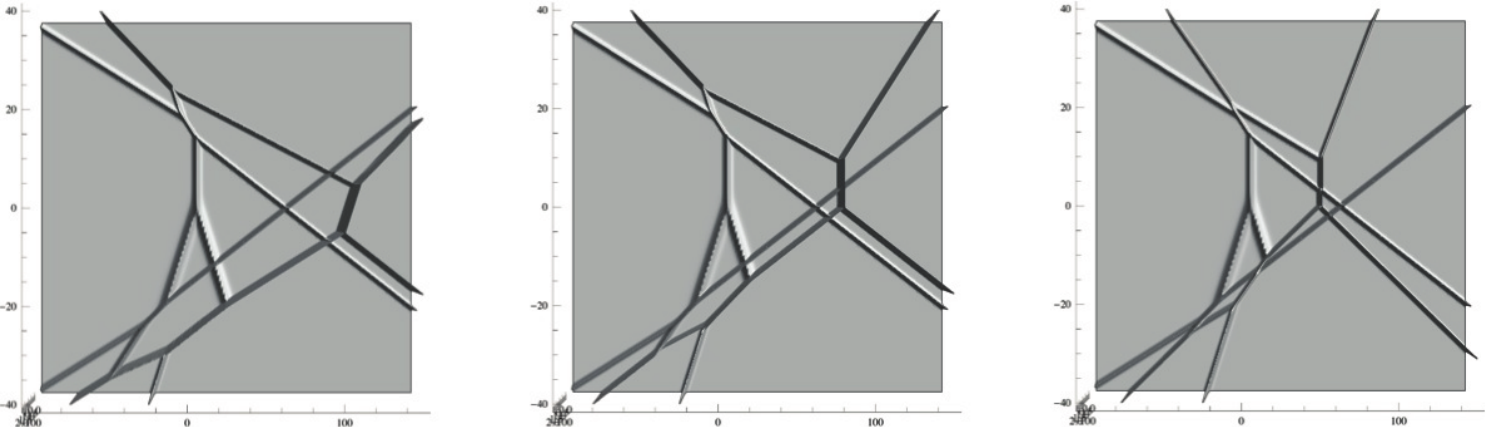}
\caption{Three different asymptotic contour plots 
$\CC_{-\infty}(\M)$ where 
$\pi(\M) = 
(7,4,2,9,1,3,8,6,5)$, based on different choices for 
$(\kappa_1,\dots,\kappa_9)$. The left panel corresponds to the contour plot with
$(\kappa_1,\ldots,\kappa_9)=(-4,-3,-2,-1,0,1,2,3,4)$;  note that this is the same as 
 $G_-(L)$ (cf. Figure \ref{LePlabic}).}
\label{3realizations}
\end{figure}

We can use a very similar algorithm to construct $G_+(L)$ from the ``dual'' $\Le$-diagram of $L$.
\begin{definition}
Given $\mathcal{M} \subset {[n] \choose k}$, we define its \emph{dual} 
$\mathcal{M^*}$ to be the collection 
$$\mathcal{M^*} = \{ \{n+1-j_1, n+1-j_2,\dots,n+1-j_k\} \ \vert \ \{j_1,\dots,j_k\} \in \mathcal{M} \}.$$

Given $\pi \in S_n$, we define its \emph{dual} 
to be the permutation
$\pi^* = \iota \circ \pi^{-1}$, where $\iota$ is the involution in $S_n$ such that 
$\iota(j) = n+1-j$.  

Given a $\Le$-diagram $L$, we define its \emph{dual}
to be the $\Le$-diagram $L^*$ such that 
$\pi(L^*) = \pi(L)^*$.
\end{definition}

\begin{remark}
Note that which positroid cell 
$S_{\mathcal{M}}^{tnn} = S_{\pi}^{tnn} = S_{L}^{tnn}$ a fixed element of 
$(Gr_{k,n})_{\geq 0}$ lies in depends on 
a choice of ordered basis $(e_1,e_2,\dots,e_n)$ for $\R^n$.  If we relabel each 
basis element $e_i$ by $e_{n+1-i}$, then $\mathcal{M}$, $\pi$, and $L$ are replaced by 
their duals $\mathcal{M^*}$, $\pi^*$, and $L^*$. 
\end{remark}


\begin{figure}[h]
\centering
\includegraphics[height=3in]{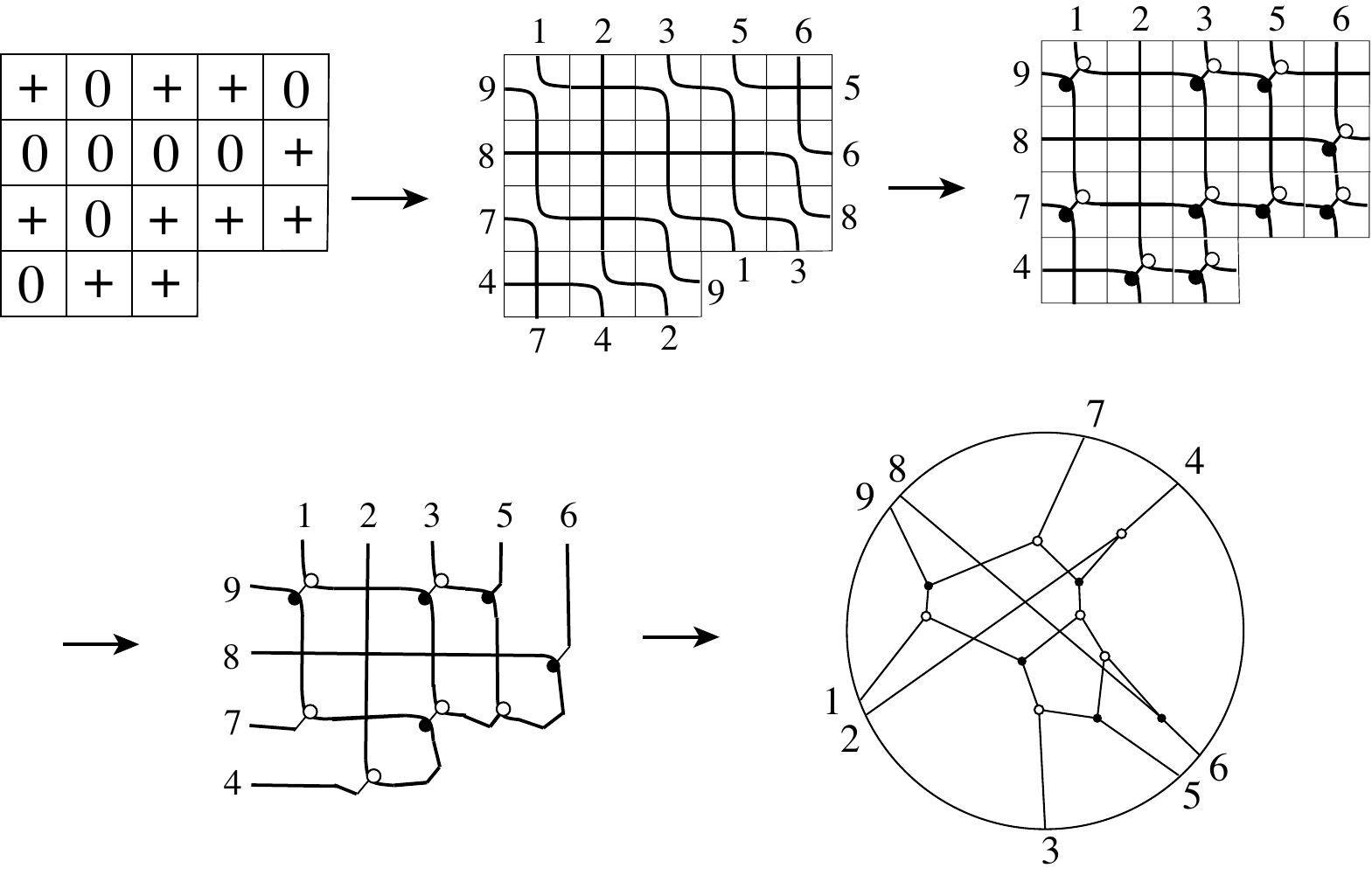}
\caption{Construction of $G_+(L)$ where $\pi(L) = 
(7,4,2,9,1,3,8,6,5)$.  The top left figure shows $L^*$.
If we compare the unbounded solitons in $G_+(L)$ and $G_-(L)$, they appear to be in a different order.
However, their order will be the same 
at $|y|\gg0$.
\label{LePlabic+}}
\end{figure}

\begin{theorem}\label{t>>0}
Choose a positroid cell $S_{\mathcal{M}}^{tnn} = 
S_{L}^{tnn} = S_{\pi}^{tnn}$ where $\pi \in S_n$.
Apply Algorithm \ref{LeToPlabic} to $L^*$, but replace every label $j$ around the 
boundary of the $\Le$-diagram and plabic graph with $\pi(n+1-j)$.  This produces  
a graph we call  $G_+(L)$.   
Then we can explicitly construct $\CC_{+}(\mathcal{M})$ from 
$G_+(L)$, just as in Theorem \ref{t<<0}.
See Figure \ref{LePlabic+}.
\end{theorem}


\subsection{The proof of Theorem \ref{t<<0}}

In this section we present the proof of Theorem \ref{t<<0}.  The main
strategy is to use induction on the number of rows in the 
$\Le$-diagram $L$.  More specifically,
let $L'$ denote the $\Le$-diagram $L$ with its
top row removed.  
In Lemma \ref{lem:algo} we will explain that $G_-(L')$ can be seen 
as a labeled subgraph of $G_-(L)$.  In Theorem \ref{induction}, we will explain
that if $\mathcal{M'} = \mathcal{M}(L')$, then there is a polyhedral
subset of $\CC_{-}(\mathcal{M})$ which coincides with 
$\CC_{-}(\mathcal{M'})$.  And moreover, every vertex of 
$\CC_{-}(\mathcal{M'})$ appears as a vertex of 
$\CC_{-}(\mathcal{M})$.  By induction we 
can assume that Theorem \ref{t<<0} correctly computes 
$\CC_{-}(\mathcal{M'})$, which in turn provides us with 
a description of 
``most" of $\CC_{-}(\mathcal{M})$, including all 
line-solitons and vertices whose indices do not include $1$.  
On the other hand, Theorem \ref{algo} gives
a complete description of the unbounded solitons of both 
$\CC_{-}(\mathcal{M'})$ and
$\CC_{-}(\mathcal{M})$ in terms of $\pi(L')$ and $\pi(L)$.  
In particular, 
$\CC_{-}(\mathcal{M})$ contains one more unbounded soliton
at $\bar{y}\gg 0$ than does 
$\CC_{-}(\mathcal{M})$,
and $\CC_{-}(\mathcal{M})$ contains $\ell$ more unbounded solitons at $\bar{y}\ll 0$
where $\ell$ is the difference in length of the first two rows.
This information together with 
the resonance property allows us to complete the description of 
$\CC_{-}(\mathcal{M})$ and match it up with the combinatorics
of $G_-(L)$.

\begin{lemma}\label{lem:asymptotics}
Let $\pi = \pi(L)$ be the derangement associated to $L$.  Then 
Algorithm \ref{LeToPlabic} produces a generalized plabic graph $G_-(L)$ 
whose trip permutation is $\pi$.  
\end{lemma}
\begin{proof}
It is clear from the construction that 
$G_-(L)$ is a generalized plabic graph.
Note that if we follow the rules of the road starting from
a boundary vertex of $G_-(L)$, we will first follow a
``pipe" northwest (see the top right picture in Figure \ref{LePlabic}),
and then travel straight across the row or column where that pipe
ended.  This has the same effect as the bijection
of Definition \ref{Le2permutation}.
\end{proof}

We now present a lemma which 
explains the relationship between $G_-(L)$ and $G_-(L')$,
where $L'$ is the $\Le$-diagram $L$ with the top row removed.

\begin{lemma}\label{lem:algo}
Let $L$ be a $\Le$-diagram with $k$ rows and $n-k$ columns, 
and let $G$ denote the generalized
plabic graph associated to $L$
via Algorithm \ref{LeToPlabic}.
Recall that Algorithm \ref{LeToPlabic} uses
Definition \ref{Le2permutation} to label the boundary vertices 
of $G$; we then use the rules of the road to label edges of $G$
by pairs of integers.
Form a new $\Le$-diagram $L'$ from  $L$ by removing the top row of $L$; 
suppose that $\ell$ is the sum of the 
number of rows and columns in $L'$.
Let $G'$ denote the edge-labeled plabic graph associated to $L'$,
but instead of using the labels $\{1,2,\dots,\ell\}$, use the 
labels $\{n-\ell+1, n-\ell+2,\dots,n\}$.  
Let $h$ denote the label of the top row of $L$.  Then 
$G'$ is obtained from $G$ by removing the trip $T_h$ starting 
at $h$, together with any edges to the right of the trip
which have a trivalent vertex on $T_h$.
\end{lemma}

We omit
the proof of Lemma \ref{lem:algo}; it should be clear
after the following  example.  

\begin{example} 
Figure \ref{building} illustrates Lemma \ref{lem:algo}
 with the example of $\pi=(7,4,2,9,1,3,8,6,5)$
as in 
Figure \ref{LeToPlabic}.  It
illustrates the result of Algorithm \ref{LeToPlabic}, applied
to the chain of $\Le$-diagrams obtained by successively adding rows from the bottom
of the diagram. 
We suggest that the reader use the 
rules of the road to fill in all edge labels on these (generalized)
plabic graphs.  The middle part of Figure \ref{building} gives the 
permutation associated to the corresponding $\Le$-diagram.
  Notice
the relationship between the excedances in these permutations
and the labeled line-solitons on the right side of the figure, e.g. the excedances
$(1,2,4,7)$ and the soliton index  $[1,7], [2,4],[4,9],[7,8]$ in the top figure.
It follows immediately from  the rules of the road that the 
sequence of (edge-labeled) plabic graphs on the right side of the figure
are nested within each other.
\begin{figure}[h]
\centering
\includegraphics[height=3.6in]{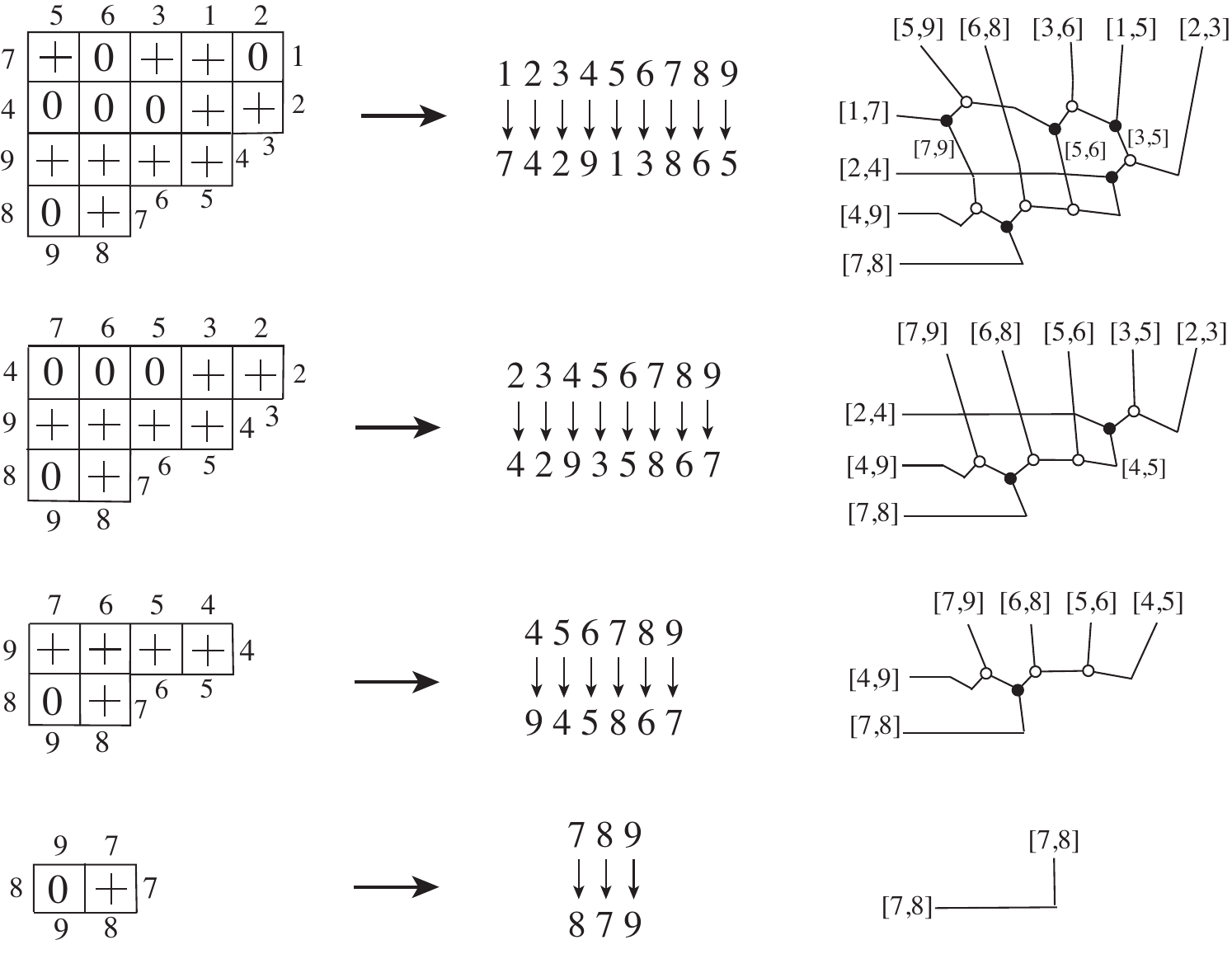}
\caption{Inductive construction of the edge-labeled
generalized plabic graph $G_-(L)$ of the case $\pi=(7,4,2,9,1,3,8,6,5)$.
\label{building}}
\end{figure}
\end{example}

Let $\{i_1,\dots,i_k\}$ denote the lexicographically 
minimal element of $\mathcal{M}$.  (This corresponds to the collection
of pivots for any $A\in S_{\mathcal{M}}^{tnn}$.)  
To simplify
the notation, we will {\it assume without loss of generality that $i_1 = 1$.}  
Now set $\mathcal{M'} = \mathcal{M}(L')$.  We can also describe
$\mathcal{M'} = \{ J\setminus \{1\} \ \vert \ 1\in J \text{ and } J\in \mathcal{M} \}.$
Our next goal is to explain in Theorem \ref{induction} the relationship between 
$\CC_{-}(\mathcal{M})$ and 
$\CC_{-}(\mathcal{M'})$.  However, we first prove a useful 
lemma.

\begin{lemma}\label{dominant}
Consider the point $v_{a,b,c}$ where $1 \notin \{a,b,c\}$.  Then at this 
point, we have that $\phi_1 > \phi_a = \phi_b = \phi_c$.
It follows that every region $R$ in $\CC_{-}(\mathcal{M})$
incident to the point $v_{a,b,c}$ is labeled by a 
dominant exponential $E_J$ such that $1\in J$.
\end{lemma}

\begin{proof}
Recall that $\phi_i(\bar{x},\bar{y}) = \kappa_i
\bar{x} + \kappa_i^2 \bar{y} - \kappa_i^3.$
A calculation shows that 
$\phi_a(v_{a,b,c}) = 
\phi_b(v_{a,b,c}) = 
\phi_c(v_{a,b,c}) = 
-\kappa_a \kappa_b \kappa_c$, while
$\phi_1(v_{a,b,c}) =- \kappa_1(\kappa_a \kappa_b + \kappa_a \kappa_c +
\kappa_b \kappa_c) + \kappa_1^2 (\kappa_a+\kappa_b + \kappa_c) - \kappa_1^3.$

Without loss of generality suppose $a<b<c$, so 
then  $\kappa_1 < \kappa_a < \kappa_b < \kappa_c$.
It follows that $(\kappa_b - \kappa_1) (\kappa_c-\kappa_1) > 0$,
which implies that $\kappa_1 \kappa_b + \kappa_1 \kappa_c -\kappa_1^2 <
\kappa_b \kappa_c$.
Multiplying both sides by $(\kappa_a-\kappa_1)$, which is positive,
we get 
$$\kappa_1 \kappa_a \kappa_b + \kappa_1 \kappa_a \kappa_c - \kappa_1^2 \kappa_1
-\kappa_1^2 \kappa_b - \kappa_1^2 \kappa_c + \kappa_1^3 < \kappa_a \kappa_b \kappa_c - \kappa_1 \kappa_b \kappa_c.$$
Therefore 
$$\kappa_1(\kappa_a \kappa_b + \kappa_a \kappa_c + \kappa_b \kappa_c) -
\kappa_1^2(\kappa_a + \kappa_b + \kappa_c) + \kappa_1^3 < \kappa_a \kappa_b \kappa_c,$$ which implies that 
$\phi_1(v_{a,b,c}) > \phi_a(v_{a,b,c}) = \phi_b(v_{a,b,c}) = \phi_c(v_{a,b,c}).$
\end{proof}

\begin{theorem}\label{induction}
There is an unbounded polyhedral subset
$\mathcal{R}$ of $\R^2$   whose boundary
is formed by line-solitons of $\CC_{-}(\mathcal{M})$, such that 
every region in $\mathcal{R}$ is labeled by 
a dominant exponential 
$E_J$ for some $J$ containing $1$.  In $\mathcal{R}$,
$\CC_{-}(\mathcal{M})$ coincides with  
$\CC_{-}(\mathcal{M'})$. 
Moreover, every region of 
$\CC_{-}(\mathcal{M'})$ which is incident to a trivalent
vertex and labeled by $E_{J'}$ corresponds to a region of 
$\CC_{-}(\mathcal{M})$ which is labeled by 
$E_{J' \cup \{1\}}$.
\end{theorem} 

\begin{proof}
The proof of the first part of the theorem is straightforward.  Note that
for any value of $\bar{y}$, there is an $\bar{x}$
sufficiently large such that 
$$\phi_1(\bar{x},\bar{y}) \gg
\phi_2(\bar{x},\bar{y}) \gg  \dots \gg
\phi_n(\bar{x},\bar{y}).$$
This proves the existence of the subset $\mathcal{R}$,
where every dominant exponential $E_J$ has the property
that $1\in J$.  Therefore the asymptotic contour plot within $\mathcal{R}$
depends only on the information of $\mathcal{M'}$, and hence 
coincides with 
$\CC_{-}(\mathcal{M'})$.  (More specifically, the positions of points
and line-solitons are identical, and each region label is identical
to the one from 
$\CC_{-}(\mathcal{M'})$ except that a $1$ is added to 
the index set.)

We have now shown that $\mathcal{R}$ exists, but do not yet have
any information about how large it is.  What we'll show next is that
$\mathcal{R}$ contains ``most" of $\CC_{-}(\mathcal{M'})$.  More 
specifically,
every region of 
$\CC_{-}(\mathcal{M'})$ which is incident to 
at least one trivalent vertex also corresponds to a region of 
$\CC_{-}(\mathcal{M})$.\footnote{In theory 
$\CC_{-}(\mathcal{M'})$ could e.g. have 
an unbounded region incident to an $X$-crossing but not incident
to any
trivalent vertices, which does not correspond to a region in 
$\CC_{-}(\mathcal{M})$.} 
For this we need Lemma \ref{dominant}.

By definition, all points $v_{a,b,c}$ that appear in 
$\CC_{-}(\mathcal{M'})$ have the property that
$1 \notin \{a,b,c\}$.  The three regions 
$R_1$, $R_2$, $R_3$ incident to $v_{a,b,c}$ in 
$\CC_{-}(\mathcal{M'})$ are labeled by $E(J_1)$, $E(J_2)$, and $E(J_3)$.
In particular, this means that at region $R_1$, $J_1$
is the subset $\{j_1,\dots,j_{k-1}\}$ of $\mathcal{M'}$
which maximizes the value
$\phi_{j_1} + \dots + \phi_{j_{k-1}}$.
Without loss of generality we can assume that $a\in J_1$, $b\in J_2$,
and $c\in J_3$.  By Lemma \ref{dominant}, there is a neighborhood $N$
of $v_{a,b,c}$ where $\phi_1 > \phi_a$.
It follows that in $N \cap R_1$, $J_1 \cup \{j_k=1\}$ is 
the subset of $\mathcal{M}$ that maximizes the value  
$\phi_{j_1} + \dots + \phi_{j_{k}}$.
Therefore the region $R_1$ of $\CC_{-}(\mathcal{M'})$ 
which is labeled by $E_{J_1}$ 
corresponds to a region of $\CC_{-}(\mathcal{M})$ which is 
labeled by $E_{J_1 \cup \{1\}}$.  Similarly for $R_2$ and $R_3$.
This completes the proof of the theorem.
\end{proof}
\begin{figure}
\centering
\includegraphics[height=6cm]{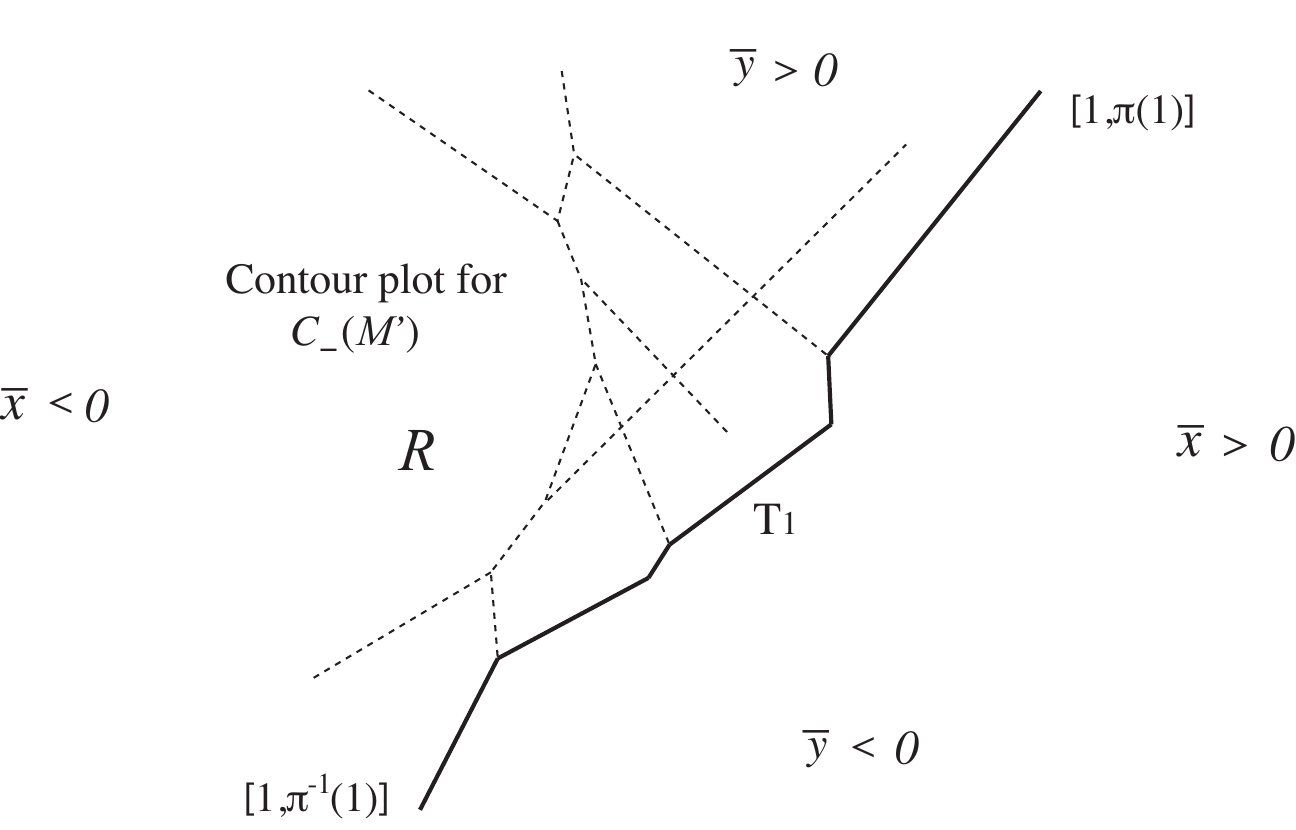}
\caption{The asymptotic contour plot $\CC_{-}(\mathcal{M'})$ within
the asymptotic contour plot $\CC_{-}(\mathcal{M}).$} \label{fig:induction}
\end{figure}
Figure \ref{fig:induction} illustrates how 
the asymptotic contour plot $\CC_{-}(\mathcal{M'})$ sits inside
the asymptotic contour plot $\CC_{-}(\mathcal{M}).$  
Recall that $T_1$ represents the {\it trip} consisting of 
all line-solitons labeled $[1,j]$ for any $j$ (cf. Figure \ref{contour-soliton}).

Theorem \ref{induction} immediately implies the following.
\begin{corollary}\label{cor:trip}
The set of trivalent vertices in $\CC_{-}(\mathcal{M})$ is equal to 
the set of trivalent vertices in $\CC_{-}(\mathcal{M'})$  together
with some vertices of the form $v_{1,b,c}$.  These vertices are the 
vertices along the trip $T_1$.  In particular, every line soliton
in $\CC_{-}(\mathcal{M})$ which was not present in 
$\CC_{-}(\mathcal{M'})$ and is not along the trip $T_1$ 
must be unbounded.  And every new bounded line-soliton in
$\CC_{-}(\mathcal{M})$ that did not come from a line-soliton 
in $\CC_{-}(\mathcal{M'})$ is of type $[1,j]$ for some $j$.
\end{corollary}


We can now complete the proof of Theorem \ref{t<<0}.
This proof will repeatedly
use the characterization of unbounded line-solitons
given by Theorem \ref{algo}.

\begin{proof}
Recall that $\mathcal{M} = \mathcal{M}(L)$ and 
$\mathcal{M'} = \mathcal{M}(L')$, where $L'$ is $L$ with 
the top row removed.
By Theorem \ref{induction}, we can construct the asymptotic contour plot
$\CC_{-}(\mathcal{M})$ inductively from the $\Le$-diagram 
$L$:
we start by drawing the asymptotic contour plot
associated with its bottom row, and then consider what happens when
we add back one row at a time.  On the other hand, by Lemma 
\ref{lem:algo},
the construction of Algorithm \ref{LeToPlabic} 
can also be viewed as an inductive procedure which involves 
adding one row at a time to the $\Le$-diagram.
Using Lemma \ref{lem:asymptotics} 
and Theorem \ref{algo}, we see that Algorithm \ref{LeToPlabic} produces
a (generalized) plabic graph whose labels on unbounded edges agree with 
the labels of the unbounded line-solitons for the soliton graph
of any $A\in S_{L}^{tnn}$.  The same is true for $A' \in S_{L'}^{tnn}$.

Let us now characterize the new vertices and line-solitons
which $\CC_{-}(\mathcal{M})$ contains, but which 
$\CC_{-}(\mathcal{M'})$ did not.  In particular,
we will show that the set of new vertices is precisely 
the set of $v_{1,b,c}$
(where $1<b<c$), such that 
either $c\to b$ is a nonexcedance of
$\pi = \pi(\mathcal{M})$, or $c\to b$ is a nonexcedance of 
$\pi' = \pi(\mathcal{M'})$, but not both.
Moreover, if $c\to b$ is a nonexcedance of $\pi$, then 
$v_{1,b,c}$ is white, while if $c\to b$ is a nonexcedance of 
$\pi'$, then $v_{1,b,c}$ is black.

By Corollary \ref{cor:trip}, 
all new vertices have the form $v_{1,b,c}$ and lie on the trip 
$T_1$.  Additionally, all new line-solitons 
which begin at some point $v_{1,b,c}$ and 
which are not on the trip $T_1$ must be unbounded.  
Since the points $v_{1,b,c}$ are trivalent, each one is incident
to either an unbounded line-soliton in 
$\CC_{-}(\mathcal{M})$,
which lies in $\mathcal{R}^c$,
or is incident to a bounded soliton of type $[i,j]$ which lies
in $\mathcal{R}$.
(Possibly both  are true when $i=1$).

If $v_{1,b,c}$ is incident to a bounded line-soliton $[i,j]$
which lies in $\mathcal{R}$, that soliton must have been unbounded in  
$\CC_{-}(\mathcal{M'})$, and hence came from 
a nonexcedance $j\to i$ in $\pi'$.
(All excedances of $\pi'$ are also excedances in $\pi$.)
In particular, $i\neq 1$, so we can conclude that $v_{1,b,c} = v_{1,i,j}$.
Conversely, if $j\to i$ is a nonexcedance of $\pi'$ which
is not a nonexcedance of $\pi$, then the corresponding 
unbounded line-soliton $[i,j]$ from 
$\CC_{-}(\mathcal{M'})$
becomes a bounded line-soliton $[i,j]$ in 
$\CC_{-}(\mathcal{M})$ which is incident to $v_{1,i,j}$.
This characterizes the new points $v_{1,b,c}$ which are incident to 
a bounded line-soliton $[i,j]$ contained in $\mathcal{R}$.

Each other new point $v_{1,b,c}$ will be incident to either:
\begin{itemize}
\item one unbounded
  line-soliton $[i,j]$ of 
$\CC_{-}(\mathcal{M})$ which lies in  
$\mathcal{R}^c$
(plus two bounded line-solitons of $T_1$), or 
\item two unbounded line-solitons of 
$\CC_{-}(\mathcal{M})$ which lie in $\mathcal{R}^c$ 
(plus one bounded line-soliton of $T_1$).
\end{itemize}
Either way, it follows that $v_{1,b,c}$ is incident to an 
unbounded line-soliton $[i,j]$ where $i \neq 1$, such that
$j\to i$ is a nonexcedance of 
$\pi$ but not a nonexcedance
of $\pi'$.  Therefore  $v_{1,b,c} = v_{1,i,j}$.

Conversely, each nonexcedance $j\to i$ of $\pi$ (respectively, $\pi'$) 
such that $1<i<j$,
and such that $j\to i$ is not a nonexcedance of $\pi'$ (respectively, $\pi$), 
gives
rise to a point $v_{1,i,j}$ of 
$\CC_{-}(\mathcal{M})$.  This is simply because 
these line-solitons must have an endpoint 
in $\CC_{-}(\mathcal{M})$ which did not appear 
in $\CC_{-}(\mathcal{M'})$.

Also note that if $v_{1,b,c}$ is a new vertex such that 
$c\to b$ is a nonexcedance of $\pi$, then 
the line-soliton $[b,c]$ must go down (towards $\bar{y}<0$)
from $v_{1,b,c}$.  However, remembering the 
resonant condition (see Figure \ref{blackwhite}),
and using the fact that $1<b<c$, we see that $[b,c]$ cannot 
be the only line-soliton going down from $v_{1,b,c}$.
Therefore $v_{1,b,c}$ must have two line-solitons going down 
from it and one line-soliton going up from it, so it is a 
white vertex.

Similarly, if $v_{1,b,c}$ is a new vertex such that 
$c \to b$ is a nonexcedance of $\pi'$, then
the line-soliton $[b,c]$ must go up (towards $\bar{y} >0$)
from $v_{1,b,c}$.  By the resonant condition,
we see that $[b,c]$ cannot 
be the only line-soliton going up from $v_{1,b,c}$.
Therefore $v_{1,b,c}$ must have two line-solitons going up
from it and one line-soliton going down from it, so it is a 
black vertex.

Using the bijection from Definition \ref{Le2permutation},
it is straightforward to verify that the above description
also characterizes the set of new vertices 
which Algorithm \ref{LeToPlabic} associates to the top row of the 
$\Le$-diagram $L$.

Finally, let us discuss the order in which the vertices 
$v_{1,b,c}$ occur along the trip $T_1$ in the asymptotic contour plot.  
First note that 
the trip $T_1$ starts at $\bar{y}<0$ and along each 
line-soliton it always heads up (towards $\bar{y}> 0$).  
This follows from the resonance condition -- see Figure 
\ref{blackwhite} and take $i=1$.  Therefore the 
order in which we encounter the vertices $v_{1,b,c}$ along the trip
is given by the total order on the $\bar{y}$-coordinates of the vertices,
namely $\kappa_1 + \kappa_b + \kappa_c$.

We now claim that this total order is identical 
to the total order on the positive integers $1<b<c$,
that is, it does not depend on the 
choice of $\kappa_i$'s, as long as 
$\kappa_1 < \dots < \kappa_n$.  If we can show this,
then we will be done, because this is precisely the order
in which the new vertices occur along the trip 
$T_1$ in the graph $G_-(L)$.

 \begin{figure}[h]
\centering
\includegraphics[height=4cm]{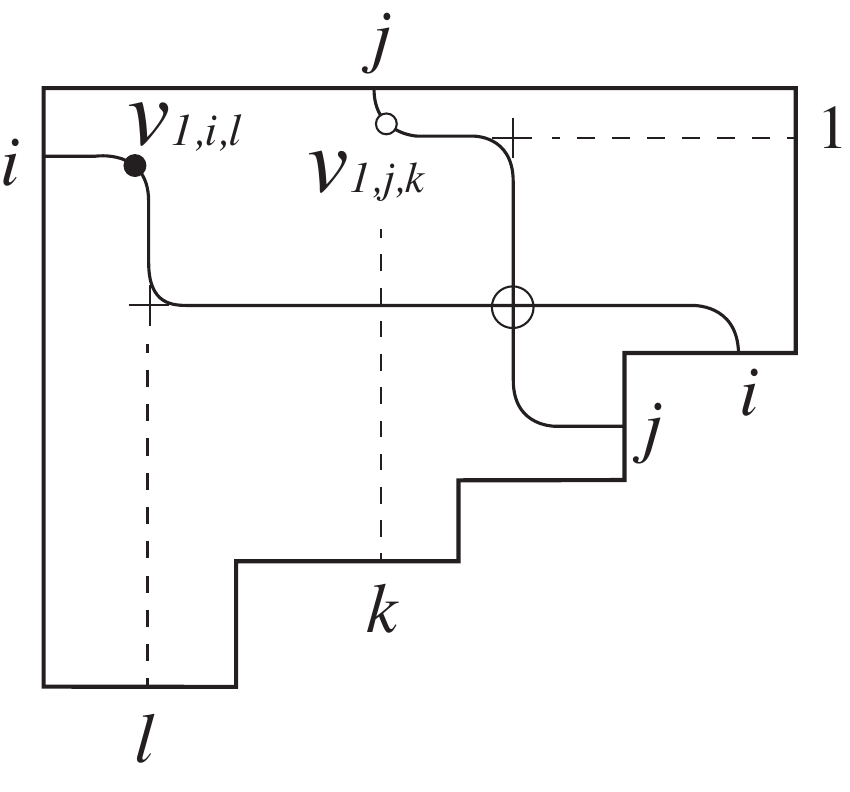}
\caption{}
\label{fig:pipecrossing}
\end{figure}

To prove the claim,
it is enough to show that among the set of new vertices
$v_{1,b,c}$, there are not two of the form 
$v_{1,i,\ell}$ and $v_{1,j,k}$ where 
$i<j<k<\ell$.  To see this, note that the 
indices $b$ and $c$ of the new vertices $v_{1,b,c}$ 
can be easily read off from the algorithm in 
Definition \ref{Le2permutation}: $c$ will come from
the bottom label of the corresponding column, while
$b$ will come from the northwest endpoint of the pipe
that $v_{1,b,c}$ lies on.  Therefore, 
if there are two new vertices $v_{1,i,\ell}$
and $v_{1,j,k}$, then they must come from 
a pair of crossing pipes, as in Figure 
\ref{fig:pipecrossing}.  Note that the crossing
of the pipes must have come from a $0$ in the 
$\Le$-diagram.  From the figure it is clear that
the pipe heading north from the crossing must 
turn west at some point, while the pipe heading 
west from the crossing must turn north at some point.
Both of these turning points must have come from a $+$
in the $\Le$-diagram, but now we see that the 
$\Le$-diagram violates the $\Le$-property.
This is a contradiction, and completes the proof.
\end{proof}

\subsection{The proof of Theorem \ref{t>>0}}

Theorem \ref{t>>0} can be seen as a corollary of Theorem \ref{t<<0}.
\begin{proof}
Recall that $\kappa_1 < \dots < \kappa_n$.
We define $\lambda_i = -\kappa_{n+1-i}$.  Then 
$\lambda_1 < \dots < \lambda_n$.  Set $\bar{y}' = -\bar{y}$.
Then \begin{align*}
\underset{J\in\mathcal{M}}\max \left\{
                     \sum_{i=1}^k \phi_{j_i}(\bar{x},\bar{y}) \right\}  
&= \underset{J\in\mathcal{M}}\max \left\{
                     \sum_{i=1}^k \kappa_{j_i} \bar{x} + \kappa_{j_i}^2 \bar{y} - \kappa_{j_i}^3 \right\}  \\
&= \underset{J\in\mathcal{M}}\min \left\{
                     \sum_{i=1}^k -\kappa_{j_i} \bar{x} - \kappa_{j_i}^2 \bar{y} + \kappa_{j_i}^3 \right\}  \\
&= \underset{J\in\mathcal{M}}\min \left\{
                     \sum_{i=1}^k \lambda_{n+1-j_i} \bar{x} - \lambda_{n+1-j_i}^2 \bar{y} -\lambda_{n+1-j_i}^3 \right\}  \\
&= \underset{J\in\mathcal{M^*}}\min \left\{
                     \sum_{i=1}^k \lambda_{j_i} \bar{x} - \lambda_{j_i}^2 \bar{y} -\lambda_{j_i}^3 \right\}  \\
&= \underset{J\in\mathcal{M^*}}\min \left\{
                     \sum_{i=1}^k \lambda_{j_i} \bar{x} + \lambda_{j_i}^2 \bar{y}' -
\lambda_{j_i}^3 \right\}.  \\
\end{align*}
Therefore 
$\CC_{+}(\mathcal{M})$ is the locus of $\R^2$ where 
the last equation above is not linear.  Comparing this with the definition of 
$\CC_{-}(\mathcal{M})$, we see that 
$\CC_{+}(\mathcal{M})$  can be constructed from 
$\CC_{-}(\mathcal{M^*})$, with each label $j$ replaced by $n+1-j$, and with 
an involution replacing $\bar{y}$ by $-\bar{y}$.
The effect of the involution is to switch the colors of the black and white vertices
in the plabic graph, or equivalently, to replace every boundary vertex $i$ of the plabic graph
by $\pi(i)$.  This completes the proof of the theorem.
\end{proof}

\begin{example}\label{exercise}
 We invite readers to reconstruct the asymptotic contour plots in Figure \ref{contour-example}.
The plots correspond to the TP Schubert cell $S_{\pi}^{tnn}$ with $\pi=(4,5,1,2,6,3)$.
Take the $\kappa$-parameters as $(\kappa_1,\ldots,\kappa_6)=(-1,-\frac{1}{2},0,\frac{1}{2},1,\frac{3}{2})$.
Calculate the trivalent vertices $v_{i,j,k}=(\mp(\kappa_i\kappa_j+\kappa_j\kappa_k+\kappa_i\kappa_k),\pm(\kappa_i+\kappa_j+\kappa_k))$ obtained from the $\Le$-diagram and its dual.
There are 8 trivalent vertices for both $t\gg 0$ and $t\ll0$ as shown in Figure \ref{fig:exercise}.
Then following Theorem \ref{t<<0}, one obtains the asymptotic contour plots for $\CC_{\pm}(\mathcal M)$
which approximate the plots in Figure \ref{contour-example}.
 \begin{figure}[h]
\centering
\includegraphics[height=3cm]{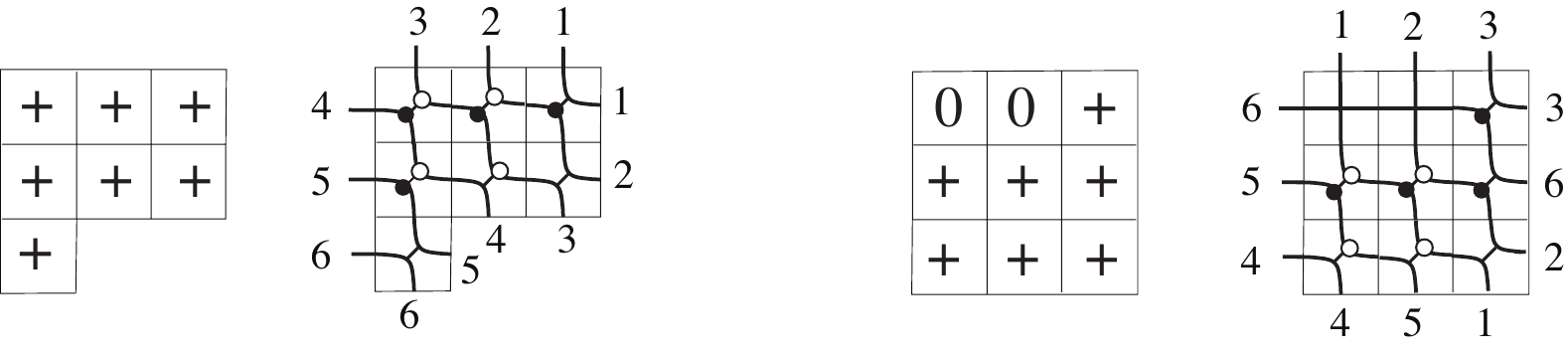}
\caption{The Le-diagrams $L$, $L^*$, and the plabic graphs for $\pi=(4,5,1,2,6,3)$.
Note that there are two X-crossings corresponding to the 0's in $L^*$ (cf. Figure \ref{contour-example}).
} \label{fig:exercise}
\end{figure}
\end{example}


\section{X-crossings and vanishing Pl\"ucker coordinates} \label{X-crossing}


In this section we show that for an arbitrary matroid stratum
$S_{\mathcal{M}}$, and for an arbitrary vector $\a$,
each $X$-crossing 
in the asymptotic contour plot 
$\CC_{\a}(\mathcal{M})$ corresponds to a vanishing 
Pl\"ucker coordinate, i.e. an index $J$ such that $J \notin \mathcal{M}$.
This implies that for any $A \in S_{\mathcal{M}}$,
the four Pl\"ucker coordinates corresponding to the dominant exponentials
incident to that $X$-crossing satisfy a 
``two-term" Pl\"ucker relation.  Note that in this 
section we are working over the real Grassmannian, as opposed to 
restricting to
$(Gr_{k,n})_{\geq 0}$.  One consequence of our main result
(Theorem \ref{th:X-crossing}) 
is that the asymptotic contour plots (and hence the soliton graphs)
coming from the totally positive Grassmannian $(Gr_{k,n})_{>0}$
have no $X$-crossings.


Before stating Theorem \ref{th:X-crossing}, 
we need some notation.
Let $h_j(x_1,\ldots,x_r)$ be the \emph{complete homogeneous symmetric 
polynomial} of degree $j$ defined by
\[
h_j(x_1,\ldots,x_r)=\sum_{n_1+\cdots+n_r=j}x_1^{n_1}x_2^{n_2}\cdots x_r^{n_r}.
\]
Then for each $\a=(a_3,\ldots,a_m)$, we define
\begin{equation}\label{gamma}
\gamma_{\a}(x_1,\ldots,x_r):=\sum_{j=1}^{m-2}h_{j-1}(x_1,\ldots,x_r)a_{j+2},
\end{equation}

\begin{theorem}\label{th:X-crossing}
Let $S_{\mathcal{M}}$ be a matroid stratum in $Gr_{k,n}$, and consider the 
corresponding asymptotic contour plot 
$\CC_{\a}(\mathcal{M})$
for fixed $\a$. 
Choose $1 \leq h < i < j < \ell \leq n$, and set $\gamma_{\a}(\kappa):=\gamma_{\a}(\kappa_h,\kappa_i,\kappa_j,\kappa_{\ell})$.  In the statements below,
$S$ is a $(k-2)$-element subset of $\{1,2,\dots,n\}$ which is 
disjoint from $\{h,i,j,\ell\}$.
\begin{enumerate}
\item Suppose there is an $X$-crossing 
involving 
line-solitons $[h,\ell]$ and $[i,j]$. 
  \begin{enumerate}
  \item 
   Then if $\kappa_h+\kappa_{\ell} > \kappa_i + \kappa_j$, 
   the dominant exponentials around the $X$-crossing 
in $\CC_{\a}(\mathcal{M})$ 
    are as in Figure \ref{fig:X-crossing} (a). 
   If $\gamma_{\a}(\kappa) <0$ then
  $S \cup \{h,\ell\} \notin \mathcal{M}$, and 
   if $\gamma_{\a}(\kappa)>0$ then
 $S \cup \{i,j\} \notin \mathcal{M}$.
   \item  
   Then if $\kappa_i+\kappa_{j} > \kappa_h + \kappa_{\ell}$, 
   the dominant exponentials around the $X$-crossing 
in $\CC_{\a}(\mathcal{M})$ 
    are as in Figure \ref{fig:X-crossing} (b).  
   If $\gamma_{\a}(\kappa)<0$
 then $S \cup \{i,j\} \notin \mathcal{M}$, and 
   if $\gamma_{\a}(\kappa)>0$ 
  then $S \cup \{h,\ell\} \notin \mathcal{M}$.
  \end{enumerate}
\item  Suppose there is an $X$-crossing involving line-solitons
$[h,i]$ and $[j,\ell]$.  
   Then the dominant exponentials around the $X$-crossing 
in $\CC_{\a}(\mathcal{M})$ 
    are as in Figure \ref{fig:X-crossing} (c).  
   If $\gamma_{\a}(\kappa)<0$
 then $S \cup \{j,\ell\} \notin \mathcal{M}$, and 
 if $\gamma_{\a}(\kappa)>0$ then
 $S \cup \{h,i\} \notin \mathcal{M}$.
\item  Suppose there is an $X$-crossing involving line-solitons
$[h,j]$ and $[i,\ell]$.  
   Then the dominant exponentials around the $X$-crossing 
in $\CC_{\a}(\mathcal{M})$ 
    are as in Figure \ref{fig:X-crossing} (d).  
   If $\gamma_{\a}(\kappa)<0$ then
  $S \cup \{h,j\} \notin \mathcal{M}$, and 
 if $\gamma_{\a}(\kappa)>0$ then
  $S \cup \{i,\ell\} \notin \mathcal{M}$.
\end{enumerate}
It follows that  in each of the above cases, we get a 
``two-term" Pl\"ucker relation for any $A \in S_{\mathcal{M}}$: 
\begin{itemize}
\item In Case (1), we have 
$\Delta_{i\ell S}(A) \Delta_{hj S}(A) = \Delta_{hiS}(A) \Delta_{j\ell S}(A)$.
\item In Case (2), we have 
$\Delta_{h\ell S}(A) \Delta_{ijS}(A) = \Delta_{hjS}(A) \Delta_{i\ell S}(A).$
\item In Case (3), we have
$\Delta_{h\ell S}(A) \Delta_{ij S}(A) = - \Delta_{hiS}(A) \Delta_{j\ell S}(A).$
\end{itemize}
where $\Delta_{hiS}$ is shorthand for $\Delta_{\{h,i\}\cup S}$, etc.
\end{theorem}
\begin{figure}[h]
\centering
\includegraphics[height=1.5in]{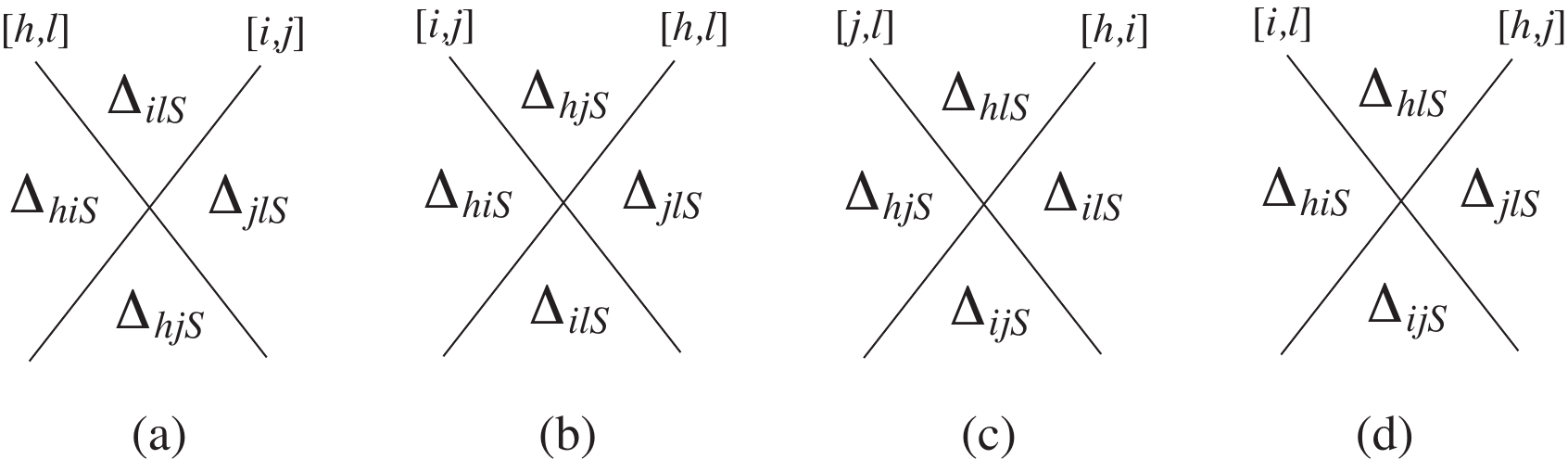}
\caption{Different types of $X$-crossings.}
\label{fig:X-crossing}
\end{figure}

Corollary \ref{no-X} follows immediately from Theorem 
\ref{th:X-crossing}.
\begin{corollary}\label{no-X}
Let $\mathcal{M} = {[n] \choose k}$  
and consider the corresponding \emph{uniform matroid stratum}
$S_{\mathcal{M}} \subset Gr_{k,n}$. 
Then for any vector $\a$,
there are no $X$-crossings
in the
asymptotic contour plot 
$\CC_{\a}(\mathcal{M})$.  In particular,
the asymptotic contour plots coming from the totally positive
Grassmannian $(Gr_{k,n})_{>0}$ have no $X$-crossings.
\end{corollary}

\begin{remark}
Note that Case (3) (i.e. Figure \ref{fig:X-crossing} (d)) is impossible for 
$A \in (Gr_{kn})_{\geq 0}$, since 
the relation 
$\Delta_{h\ell S}(A) \Delta_{ij S}(A) = - \Delta_{hiS}(A) \Delta_{j\ell S}(A)$
implies that one of these four Pl\"ucker coordinates must have a sign
which is different from the other three.
Therefore asymptotic contour plots associated to positroid cells
cannot contain $X$-crossings involving line-solitons $[h,j]$
and $[i,\ell]$ for $h<i<j<\ell$.
\end{remark}

\subsection{Proof of Theorem \ref{th:X-crossing}}

First note that if we are considering the neighborhood of 
an $X$-crossing formed by two line-solitons on indices
${h,i,j,\ell}$, then we may as well assume that 
$\{h,i,j,\ell\}=\{1,2,3,4\}$ and $S = \emptyset$.

Recall that $\kappa_1 < \dots < \kappa_n$.  Also recall
that 
\[
\theta_i(\bar{x}, \bar{y},\a) = \kappa_i \bar{x} + 
\kappa_i^2 \bar{y} + \sum_{p=3}^m\kappa_i^pa_p
\] 
and the asymptotic contour plot 
$\CC_{\a}(\mathcal{M})$ is defined to be the locus in $\R^2$ where
\begin{equation*}
f_{\mathcal{M}}(\bar{x},\bar{y},\a)=\underset{J\in\mathcal{M}}\max \left\{
                     \sum_{i=1}^k \theta_{j_i}(\bar{x},\bar{y},\a) \right\}
\end{equation*}
is not linear.  

Recall from \eqref{eq-soliton} that a line-soliton of type $[i,j]$ in 
 $\CC_{\a}(\mathcal{M})$ lies on the line $L_{ij}$ whose equation is
\begin{equation}\label{Lij}
\bar{x}+(\kappa_i+\kappa_j)\bar{y}+\sum_{p=1}^{m-2}h_{p+1}(\kappa_i,\kappa_j)a_{p+2}=0.
\end{equation}
\begin{lemma}\label{lem:v1}
Let $\kappa_a<\kappa_b<\kappa_c$.
Then the $\bar y$-coordinate of the trivalent vertex $v_{abc}=(v^{\bar x}_{abc},v^{\bar y}_{abc})$ where the lines $L_{a,b}$, $L_{b,c}$, and $L_{a,c}$
mutually intersect
is given by
\[
v_{abc}^{\bar y}=-\sum_{p=1}^{m-2}h_p(\kappa_a,\kappa_b,\kappa_c)a_{p+2}.
\]
\end{lemma}

\begin{proof}
 From the intersection between $L_{a,b}$ and $L_{b,c}$, we have
 \[
 (\kappa_c-\kappa_a)\bar{y}+\sum_{p=1}^{m-2}[h_{p+1}(\kappa_b,\kappa_c)-h_{p+1}(\kappa_a,\kappa_b)]a_{p+2}=0.
 \]
 So we need to show that
 \[
 h_{p+1}(\kappa_b,\kappa_c)-h_{p+1}(\kappa_a,\kappa_b)=(\kappa_c-\kappa_a)h_p(\kappa_a,\kappa_b,\kappa_c).
 \]
 This follows from Lemma \ref{lem:h-function} below.
\end{proof}

\begin{lemma}\label{lem:h-function}
For each $i\ge1$, we have
\[
h_i(x_1,\ldots,x_r,x_{\alpha})-h_i(x_1,\ldots,x_r,x_{\beta})=(x_{\alpha}-x_{\beta})h_{i-1}(x_1,\ldots,x_r,x_{\alpha},x_{\beta}).
\]
\end{lemma}
\begin{proof}
Recall that the generating function for the 
homogeneous symmetric polynomials  is given by
\[
\exp\left[\sum_{i=1}^{\infty}\frac{1}{i}\left(\sum_{j=1}^rx_j^i\right)\lambda^i\right]=\sum_{i=0}^{\infty}h_i(x_1,\ldots,x_r)\lambda^i.
\]
 Then we have
 \begin{align*}
 &\sum_{i=1}^{\infty}\left[h_i(x_1,\ldots,x_r,x_\alpha)-h_i(x_1,\ldots,x_r,x_\beta)\right] \lambda^i \\
= &\exp\left[\sum_{i=1}^{\infty}\frac{1}{i}\left(\sum_{j=1}^{r}x_j^i+x_{\alpha}^i\right)\lambda^i\right] -
  \exp\left[\sum_{i=1}^{\infty}\frac{1}{i}\left(\sum_{j=1}^{r}x_j^i+x_{\beta}^i\right)\lambda^i\right]  \\
= &\exp\left[\sum_{i=1}^{\infty}\frac{1}{i}\left(\sum_{j=1}^{r}x_j^i+x_{\alpha}^i+x_{\beta}^i\right)\lambda^i\right]
\left(e^{-\sum\frac{1}{i}x_{\beta}^i\lambda^i}-e^{-\sum\frac{1}{i}x_{\alpha}^i\lambda^i}\right)\\
=&\left[\sum_{i=0}^{\infty}h_i(x_1,\ldots,x_r,x_\alpha,x_\beta)\lambda^i\right](x_{\alpha}-x_{\beta})\lambda \\
=&(x_{\alpha}-x_{\beta})\sum_{i=1}^{\infty}h_{i-1}(x_1,\ldots,x_r,x_\alpha,x_\beta)\lambda^i.
\end{align*}
Here we have used the formula
\[
\sum_{i=1}^{\infty}\frac{1}{i}x^i=-\ln(1-x).
\]
\end{proof}

\begin{lemma}\label{lem:2}
Recall from \eqref{gamma} and Theorem \ref{th:X-crossing} the
definition of  $\gamma_{\a}(\kappa)=\gamma_{\a}(\kappa_h,\kappa_i,\kappa_j,\kappa_{\ell})$.  Then
using Lemma \ref{lem:v1}, we have
\begin{itemize}
\item[(i)] if $\gamma_{\a}(\kappa)<0$, then $v_{123}^{\bar y}<v_{124}^{\bar y}<v_{134}^{\bar y}<v_{234}^{\bar y}$, and
\item[(ii)] if $\gamma_{\a}(\kappa)>0$, then $
v_{123}^{\bar y}>v_{124}^{\bar y}>v_{134}^{\bar y}>v_{234}^{\bar y}.$
\end{itemize}
\end{lemma}
\begin{proof}
Using Lemma \ref{lem:h-function},
we compute
\begin{align*}
v_{abd}^{\bar y}-v_{abc}^{\bar y}&=-\sum_{p=1}^{m-2}\left[h_p(\kappa_a,\kappa_b,\kappa_d)-
h_p(\kappa_a,\kappa_b,\kappa_c)\right]\,a_{p+2}\\
&=-(\kappa_d-\kappa_c)\left[\sum_{p=1}^{m-2}h_{p-1}(\kappa_a,\kappa_b,\kappa_c,\kappa_d)a_{p+2}\right]\\
&=-(\kappa_d-\kappa_c)\gamma_{\a}(\kappa).
\end{align*}
Then it is straightforward to show the assertion.
\end{proof}

We have the following total
order on the slopes of the lines $L_{ij}$ for $1 \leq i < j \leq 4$:
\begin{itemize}
\item[(a)] If $\kappa_1+\kappa_4 > \kappa_2+\kappa_3$ then 
$$\kappa_1 + \kappa_2 < \kappa_1 + \kappa_3 < \kappa_2+\kappa_3 < \kappa_1 + \kappa_4 < \kappa_2 + \kappa_4 < \kappa_3 + \kappa_4.$$
\item[(b)] If
$\kappa_1+\kappa_4 < \kappa_2+\kappa_3$ then 
$$\kappa_1 + \kappa_2 < \kappa_1 + \kappa_3 < \kappa_1+\kappa_4 < \kappa_2 + \kappa_3 < \kappa_2 + \kappa_4 < \kappa_3 + \kappa_4.$$
\end{itemize}

\begin{proposition}
Suppose $\gamma_{\a}(\kappa)=\gamma_{\a}(\kappa_1,\kappa_2,\kappa_3,\kappa_4)<0$.
If $\kappa_1 + \kappa_4 > \kappa_2 + \kappa_3$ then the configuration 
of lines $L_{ij}$ for $1 \leq i < j \leq 4$ is as in the left
of Figure \ref{fig:6lines} (up to perturbing the $\kappa_i$'s, which
perturbs the slopes of lines while keeping the total order as 
shown above).  And if 
$\kappa_1 + \kappa_4 < \kappa_2 + \kappa_3$ then the configuration 
of lines is as in the right of Figure \ref{fig:6lines}. 

For the other cases with
$\gamma_{\a}(\kappa)>0$, the configurations of lines $L_{ij}$ can be  obtained by
a  $180^{\circ}$ rotation of those figures.
\end{proposition}
\begin{figure}[h]
\centering
\includegraphics[height=1.7in]{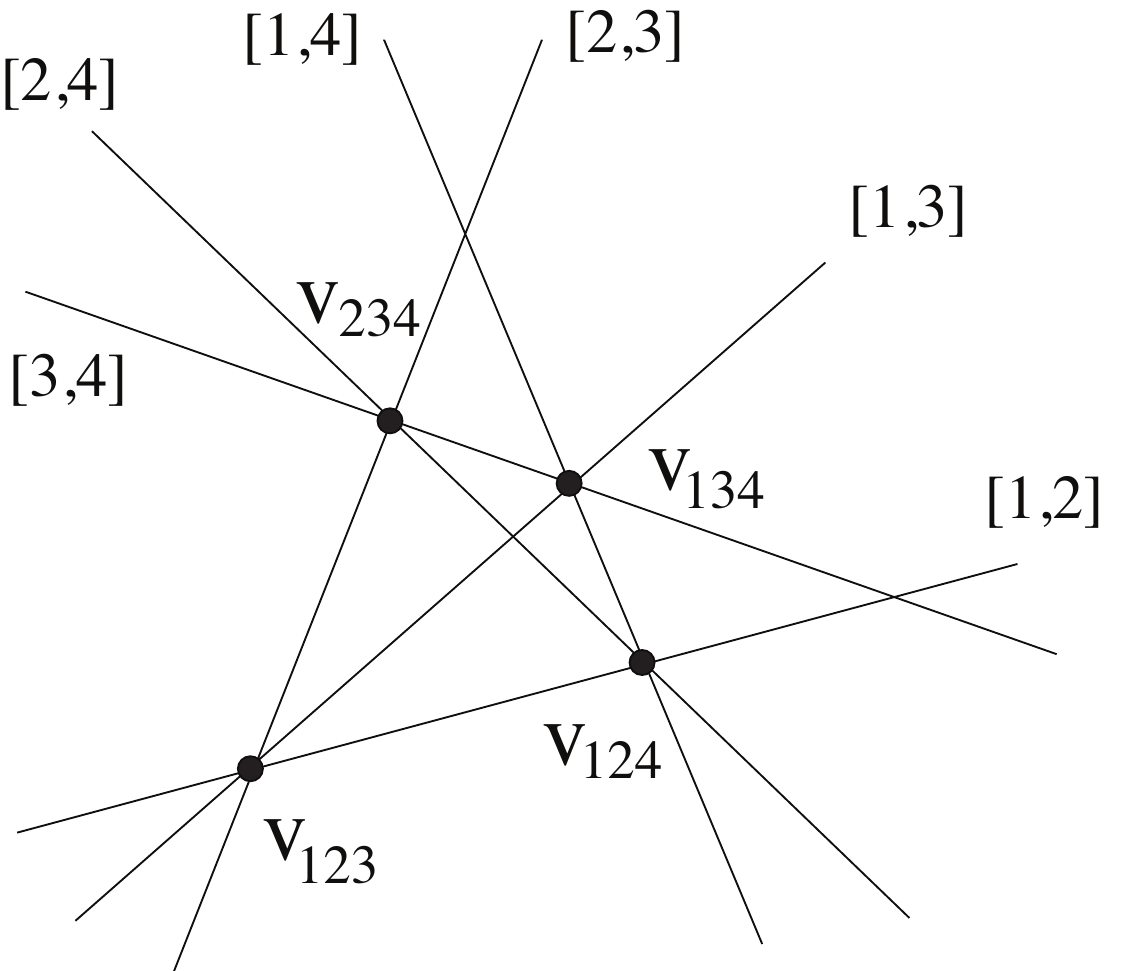}
\hskip 1cm
\includegraphics[height=1.7in]{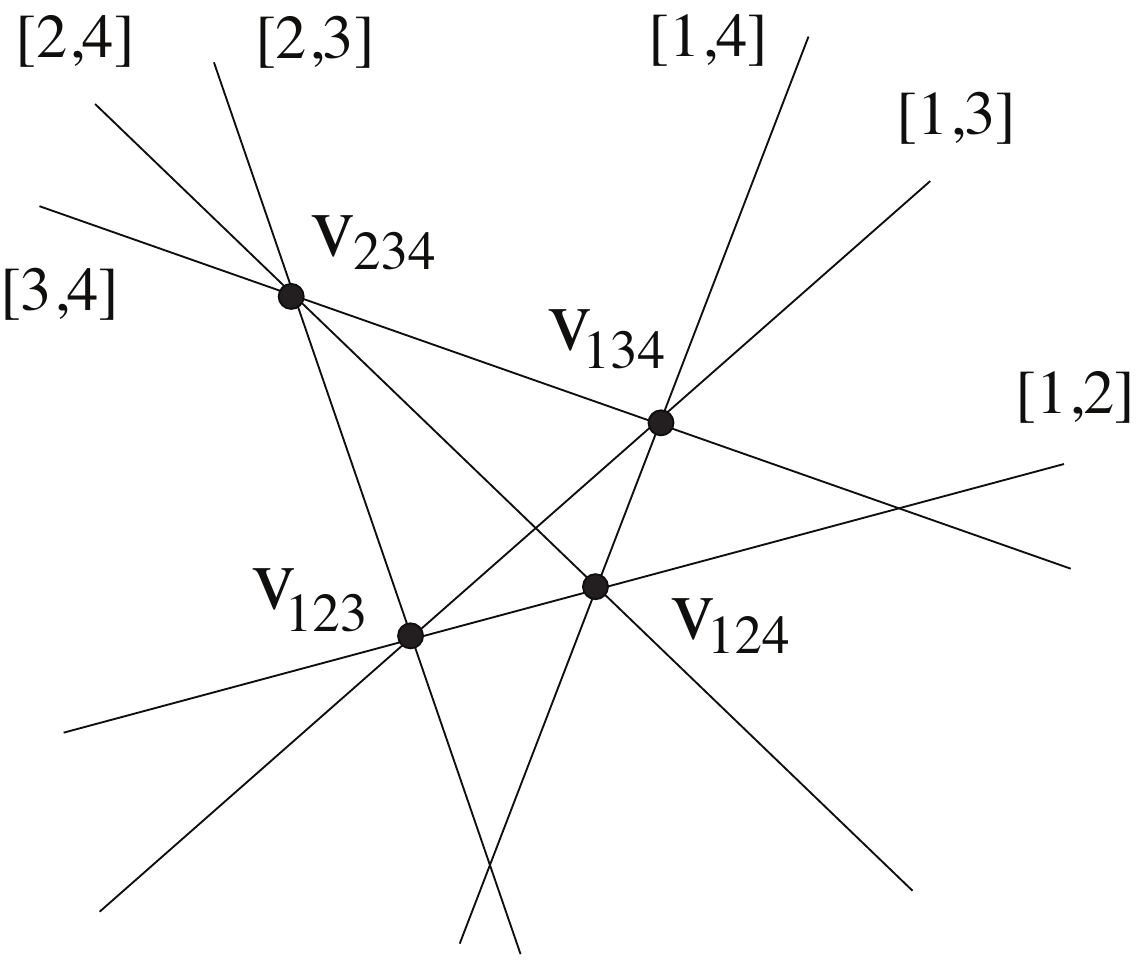}
\caption{The configuration of the lines $L_{ij}$  for $\gamma_{\a}(\kappa)<0$, based on whether 
$\kappa_1+\kappa_4>\kappa_2+\kappa_3$ or 
$\kappa_1+\kappa_4<\kappa_2+\kappa_3$.}
\label{fig:6lines}
\end{figure}

\begin{proof}
Let $x$ be the point where $L_{14}$ and $L_{23}$ meet.
If $\kappa_1+\kappa_4 > \kappa_2 + \kappa_3$, then 
$L_{24}$ must intersect both 
$L_{14}$ and $L_{23}$ \emph{below} $x$.  This follows from the fact
that $\kappa_2+\kappa_4 > \kappa_1 + \kappa_4$ and 
$v_{124}^{\bar y} < v_{234}^{\bar y}$ (from Lemma \ref{lem:2}).  
While if $\kappa_1+\kappa_4 < \kappa_2 + \kappa_3$, then 
$L_{24}$ must intersect both 
$L_{14}$ and $L_{23}$ \emph{above} $x$.  This follows from the fact
that $\kappa_2+\kappa_4 > \kappa_2 + \kappa_3$ and 
$v_{124}^{\bar y} < v_{234}^{\bar y}$.  
In either case, we can now draw $L_{24}$, and so have locations
for the points $v_{124}$ and $v_{234}$.

Now consider the placement of $L_{13}$.  If 
$\kappa_1+\kappa_4 > \kappa_2 + \kappa_3$ 
(respectively $\kappa_1+\kappa_4 < \kappa_2 + \kappa_3$)
then $\kappa_1 + \kappa_3< \kappa_2 + \kappa_3$
(respectively, 
$\kappa_1 + \kappa_3< \kappa_1 + \kappa_4$).
And $L_{13}$ intersects $L_{14}$ and $L_{23}$ in 
$v_{134}$ and $v_{123}$, which must satisfy 
$v_{234}^{\bar y} > v_{134}^{\bar y} > v_{124}^{\bar y} > v_{123}^{\bar y}$.
So $L_{13}$ must be as shown in Figure \ref{fig:6lines}.
We now have locations for all four points $v_{ijk}$,
so we can draw in all six lines $L_{ij}$.
\end{proof}

Now for each region in the two figures, we will compute the total
order on $\{\theta_1, \theta_2, \theta_3, \theta_4\}$.
If $\theta_a(\bar{x}, \bar{y},\a) > 
\theta_b(\bar{x}, \bar{y},\a) > 
\theta_c(\bar{x}, \bar{y},\a) > 
\theta_d(\bar{x}, \bar{y},\a)$ then we will write
$abcd$ as shorthand for this order.  Also note that if 
$\bar{y}$ is finite then for $\bar{x} \ll0$, we have
$\theta_1(\bar{x},\bar{y},\a) >  
\theta_2(\bar{x},\bar{y},\a) >  
\theta_3(\bar{x},\bar{y},\a) >  
\theta_4(\bar{x},\bar{y},\a).$  This allows us to compute the total orders
on the $\theta_i$'s, as shown in Figure \ref{fig:linesorder}. 
\begin{figure}[h]
\centering
\includegraphics[height=1.8in]{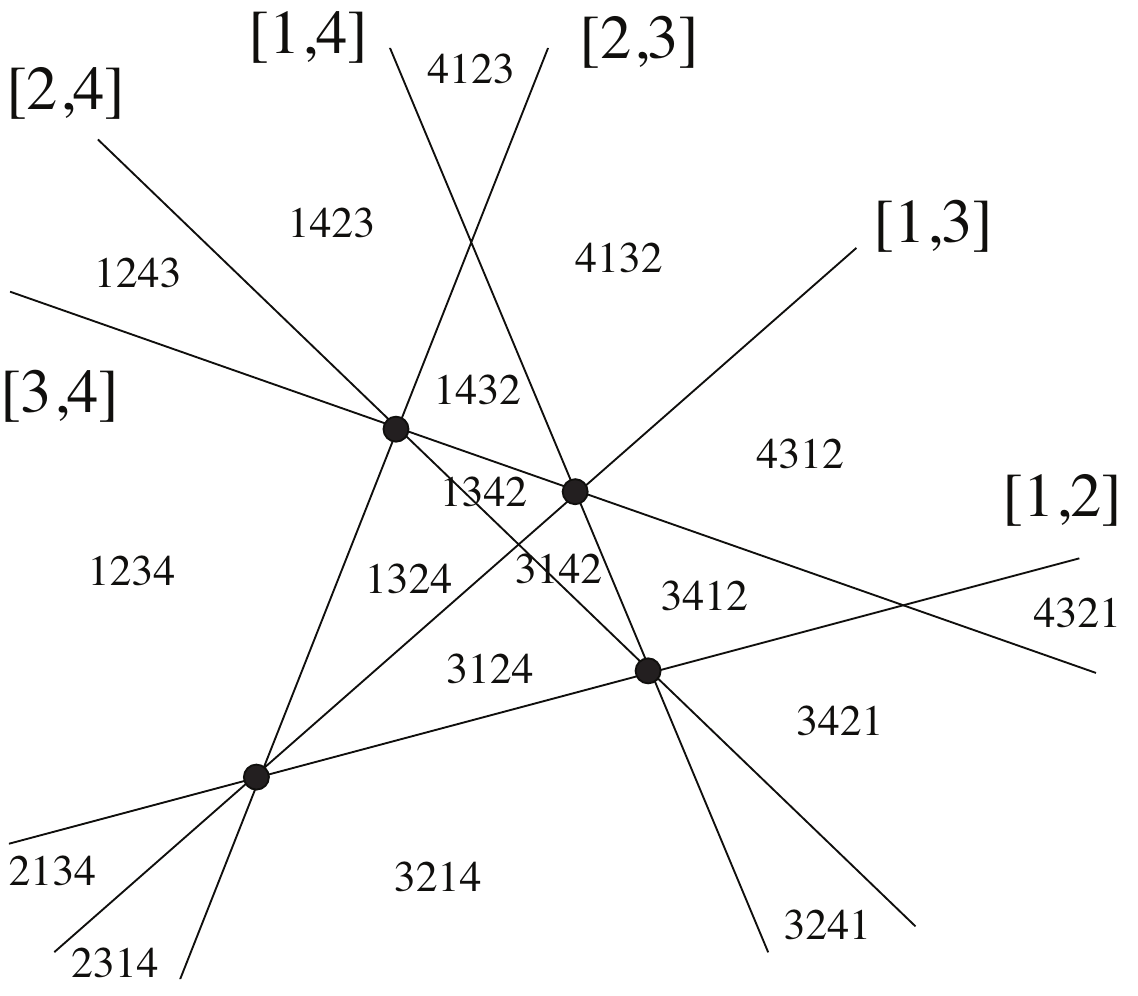}
\hskip 1cm
\includegraphics[height=1.8in]{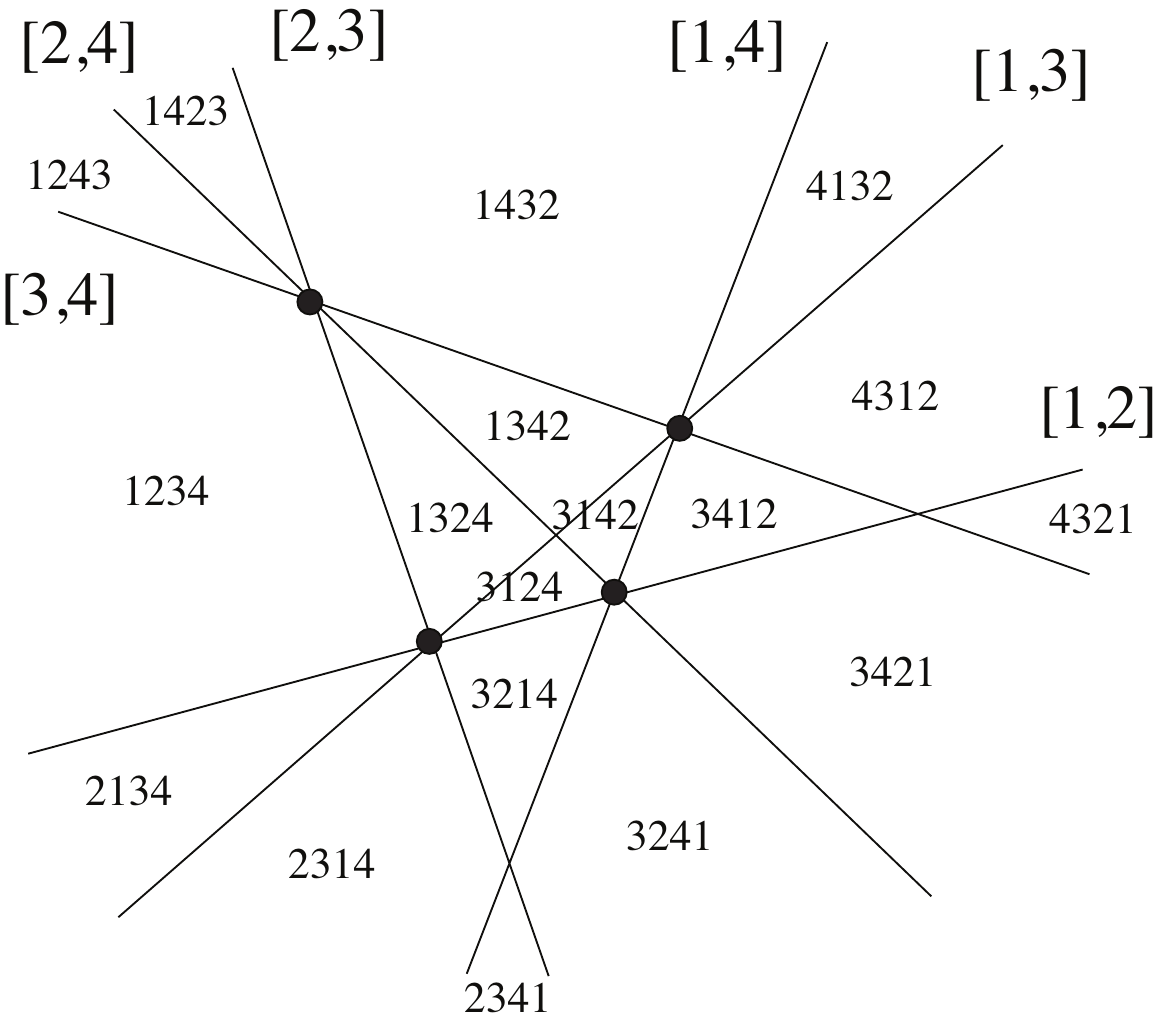}
\caption{The total order on the $\theta_i(\bar{x},\bar{y},\a)$'s  for $\gamma_{\a}(\kappa)<0$, 
based on whether 
$\kappa_1+\kappa_4>\kappa_2+\kappa_3$ or 
$\kappa_1+\kappa_4<\kappa_2+\kappa_3$.}
\label{fig:linesorder}
\end{figure}

Using Figure \ref{fig:linesorder} for $\gamma_{\a}(\kappa)<0$, we can 
compute the dominant exponentials.  To compute
the dominant exponentials  in a given region, we consider the 
region label $abcd$ and choose the leftmost two indices
 such that the corresponding
Pl\"ucker coordinate is nonzero.

We now prove Theorem \ref{th:X-crossing}.
\begin{proof}
Consider Part (1a) of the theorem.  Suppose that we see an 
$X$-crossing in the contour plot involving line-solitons of types
$[1,4]$ and $[2,3]$. Let us consider the local neighborhood
of this $X$-crossing, looking at the left of Figure \ref{fig:linesorder}.
Note that in all four regions immediately incident to the $X$-crossing,
we have that each of $\theta_1$ and $\theta_4$ is greater than each of 
$\theta_2$ and $\theta_3$.  So at $\gamma_{\a}(\kappa)<0$, if $\Delta_{14}(A) \neq 0$,
then this $X$-crossing would not appear in the contour plot
($E_{14}$ would be the dominant exponential in a neighborhood of 
the $X$-crossing).  Therefore we must have $\Delta_{14}(A) = 0$.

Similarly, at $\gamma_{\a}(\kappa)>0$, if $\Delta_{23}(A) \neq 0$, then this 
$X$-crossing would not appear in the contour plot
($E_{23}$ would be the dominant exponential in a neighborhood of the 
$X$-crossing.)  Therefore we must have $\Delta_{23}(A) = 0$.

Proving Part (1b) of the theorem is precisely analogous, but 
we look at the right of Figure \ref{fig:linesorder}.  
Proving Parts (2) and (3) are very similar, and we leave them to the 
reader.
\end{proof}


\section{TP Schubert cells, reduced plabic graphs,  and cluster algebras}\label{Reduced-Cluster}

The most important plabic graphs are those which are {\it reduced}
\cite[Section 12]{Postnikov}. Although it is not 
easy to characterize reduced plabic graphs (they are defined to 
be plabic graphs whose {\it move-equivalence class} contains no 
graph to which one can apply a {\it reduction}), they are 
very important because of their application to cluster
algebras \cite{Scott} and parameterizations of cells \cite{Postnikov}.

In this section, after recalling definitions,
we will state and prove a new characterization of reduced plabic graphs.
We then use this characterization to prove that soliton graphs
for TP Schubert cells which have no $X$-crossings
are in fact reduced plabic graphs.  
Using Corollary \ref{no-X}, 
we deduce that the set of dominant exponentials labeling any soliton
graph for the TP Grassmannian is a cluster for the cluster algebra associated
to the Grassmannian.  We conjecture that the coordinate ring
of each Schubert variety has a cluster algebra structure
in which the set of dominant exponentials
labeling a soliton graph without $X$-crossings for the corresponding TP Schubert cell is a 
cluster.

\subsection{Reduced plabic graphs}\label{sec:moves}

We will always assume that a plabic graph is {\it leafless}, i.e. that 
it has no non-boundary leaves, and that it has no isolated components.
In order to define {\it reduced}, we first
define some local transformations of plabic graphs.

(M1) SQUARE MOVE.  If a plabic graph has a square formed by
four trivalent vertices whose colors alternate,
then we can switch the
colors of these four vertices.
\begin{figure}[h]
\centering
\includegraphics[height=.5in]{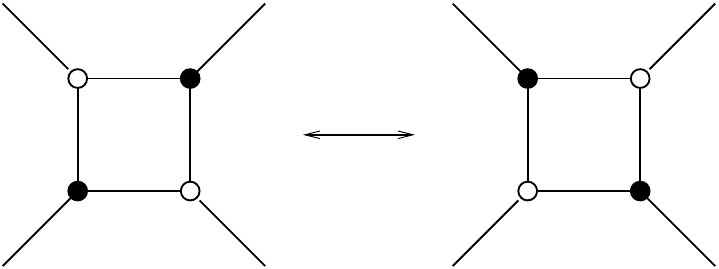}
\caption{Square move}
\label{M1}
\end{figure}

(M2) UNICOLORED EDGE CONTRACTION/UNCONTRACTION.
If a plabic graph contains an edge with two vertices of the same color,
then we can contract this edge into a single vertex with the same color.
We can also uncontract a vertex into an edge with vertices of the same
color.
\begin{figure}[h]
\centering
\includegraphics[height=.3in]{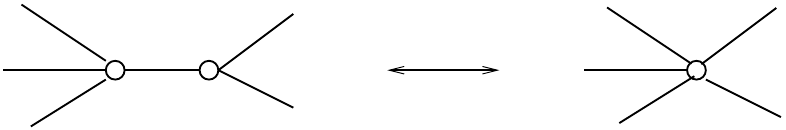}
\caption{Unicolored edge contraction}
\label{M2}
\end{figure}

(M3) MIDDLE VERTEX INSERTION/REMOVAL.
If a plabic graph contains a vertex of degree 2,
then we can remove this vertex and glue the incident
edges together; on the other hand, we can always
insert a vertex (of any color) in the middle of any edge.

\begin{figure}[h]
\centering
\includegraphics[height=.07in]{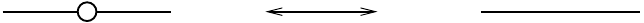}
\caption{Middle vertex insertion/ removal}
\label{M3}
\end{figure}

(R1) PARALLEL EDGE REDUCTION.  If a network contains
two trivalent vertices of different colors connected
by a pair of parallel edges, then we can remove these
vertices and edges, and glue the remaining pair of edges together.

\begin{figure}[h]
\centering
\includegraphics[height=.25in]{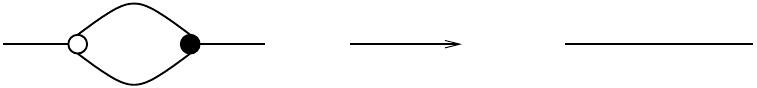}
\caption{Parallel edge reduction}
\label{R1}
\end{figure}

\begin{definition}\cite{Postnikov}
Two plabic graphs are called \emph{move-equivalent} if they can be obtained
from each other by moves (M1)-(M3).  The \emph{move-equivalence class} 
of a given plabic graph $G$ is the set of all plabic graphs which are move-equivalent
to $G$.
A leafless plabic graph without isolated components
is called \emph{reduced} if there is no graph in its move-equivalence 
class to which we can apply (R1).
\end{definition}

\begin{theorem}\cite[Theorem 13.4]{Postnikov}
Two reduced plabic graphs which each have $n$ boundary vertices
are move-equivalent if and only if they have the same 
trip permutation.
\end{theorem}

\subsection{A new characterization of reduced plabic graphs}

\begin{definition}\label{def:resonance}
We say that a (generalized) plabic graph has the 
\emph{resonance property}, if after labeling edges via Definition 
\ref{labels}, the set $E$ of edges incident to a given vertex  has
the following property:
\begin{itemize}
\item  there exist numbers $i_1<i_2<\dots<i_m$ such that when 
we read the labels of $E$,  we see the labels
$[i_1,i_2],[i_2,i_3],\dots,[i_{m-1},i_m],[i_1,i_m]$ appear in 
counterclockwise order.
\end{itemize}
\end{definition}

We call this the {\it resonance property} by  analogy with the resonance
of solitons (see Section \ref{resonance}).

\begin{remark}
Note that the  graphs in 
Figure \ref{blackwhite} satisfy the resonance property.
\end{remark}

\begin{theorem}\label{th:reduced}
A plabic graph is reduced 
if and only if it has the resonance property.\footnote{Recall 
from Definition \ref{def:plabic} that our convention is to label 
boundary vertices of a plabic graph $1,2,\dots,n$ in counterclockwise
order.  If one chooses the opposite convention,
then one must replace the word \emph{counterclockwise} in 
Definition \ref{def:resonance}
by \emph{clockwise}.}
\end{theorem}

\begin{remark}
In fact, our proof below also proves that if a generalized plabic
graph has the resonance property, then it is reduced.
\end{remark}

\begin{proof}
By Proposition \ref{Le-slopes} below, for every positroid cell $\S_{L}^{tnn}$
there is 
a reduced plabic graph $G_L^{\Le}$ satisfying the resonance property,
whose trip permutation equals $\pi(L)$.
We will show that the moves (M1), (M2), and (M3)
preserve the resonance property.  This will show that 
the entire move-equivalence class of $G_L$ satisfies the resonance property.

\begin{figure}[h]
\centering
\includegraphics[height=.8in]{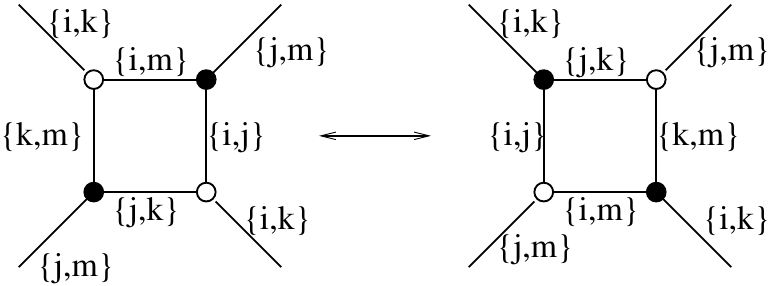}
\caption{The edge-labeling of the square move}
\label{LabeledM1}
\end{figure}
By consideration of the rules of the road, the edges of 
a plabic graph with a local configuration as in the left of 
Figure \ref{M1} must be labeled by the pairs of integers 
as shown in the left of Figure \ref{LabeledM1}, for some
$i,j,k,m$.  The edge labeling after a square move is shown
in the right of Figure \ref{LabeledM1}.  
But now note that if we compare the top left vertex of the left 
figure with the bottom right vertex of the right figure,
their edge labels together with the circular order on them concide.
Similarly we can match the other three vertices of the left figure
with the other three vertices of the right figure in Figure \ref{LabeledM1}.
Therefore the labels around each vertex 
at the left of Figure \ref{LabeledM1}
satisfy the resonance property if and only if the labels around each 
vertex at the right of Figure \ref{LabeledM1} satisfy the resonance property.

\begin{figure}[h]
\centering
\includegraphics[height=.6in]{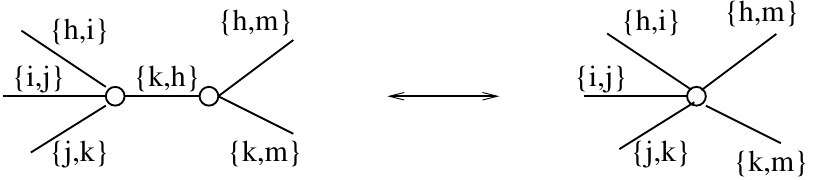}
\caption{The edge-labeling of the unicolored edge contraction}
\label{LabeledM2}
\end{figure}
Similarly, the edge labels of a plabic graph with a local configuration
as in the left of Figure \ref{M2} must be as in the left of 
Figure \ref{LabeledM2}, for some integers $h,i,j,k,m$.  
The right of Figure \ref{LabeledM2} show the new edge labels we'd get
after a unicolored edge contraction.
Note that in order for the configuration on the left 
to satisfy the resonance property,
we must have $h,i,j,k,m$ be cyclically ordered, e.g.
$h<i<j<k<m$ or $i<j<k<m<h$ or $j<k<m<h<i$ or .... 
Similarly, in order for 
the configuration at the 
right of Figure \ref{LabeledM2} to satisfy the resonance property,
we must have $h,i,j,k,m$ be cyclically ordered.
Therefore the move (M2) preserves the resonance property.

The move (M3) trivially preserves the resonance property.
All edges in Figure \ref{M3} will be labeled $[i,j]$ for some $i$ and $j$.
Therefore we have shown that moves (M1), (M2), and (M3) preserve
the resonance property.

By \cite[Theorem 13.4]{Postnikov}, for any two  reduced plabic graphs 
$G$ and $G'$ with 
the same number of boundary vertices, the following claims are 
equivalent:
\begin{itemize}
\item $G$ can be obtained from $G'$ by moves (M1)-(M3)
\item $G$ and $G'$ have the same (decorated) trip permutation.
\end{itemize}
Therefore it follows that all reduced plabic graphs with the (decorated) trip
permutation $\pi$ satisfy the resonance property.  Letting $L$ vary
over all $\Le$-diagrams, we see that all reduced plabic graphs
satisfy the resonance property.

Now we need to show that if a plabic graph $G$ has the resonance property, 
then it must be reduced.
Assume for the sake of contradiction that it is not reduced.  
Then by \cite[Lemma 12.6]{Postnikov},
there is another graph $G'$ in its move-equivalence class to which one can
apply a parallel edge reduction (R1).  
Since applications of (M1), (M2), (M3) preserve the resonance property,
$G'$ has the resonance property.
But then, it is impossible to apply (R1).
\begin{figure}[h]
\centering
\includegraphics[height=.6in]{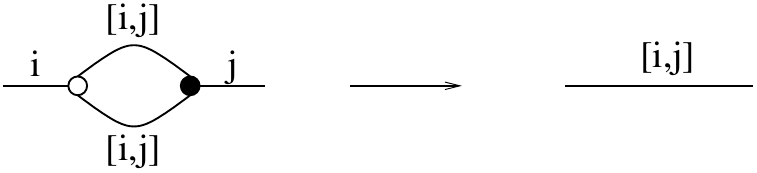}
\caption{The edge-labeling of the parallel edge reduction}
\label{LabeledR1}
\end{figure}
This is because an edge-labeling of a local configuration to which 
one can apply (R1) must be as in the left hand side of Figure \ref{LabeledR1}.
However, this local configuration violates the resonance property: 
it has a trivalent vertex with two incident edges which
 have the same edge-label.
Therefore one cannot apply (R1) to $G'$ so $G$ must be reduced.
\end{proof}

We now provide an algorithm from \cite[Section 20]{Postnikov} 
for associating a reduced plabic graph $G_L^{\Le}$ to 
any $\Le$-diagram $L$.  The plabic graph $G_L^{\Le}$ will have
the trip permutation $\pi(L)$.
\begin{algorithm}\label{Post-plabic}
\cite[Section 20]{Postnikov}
\begin{enumerate}
\item Start with a $\Le$-diagram $L$ contained in a $k\times (n-k)$
rectangle, and label its southeast border from $1$ to $n$,
starting from the northeast corner of the rectangle.
Reflect the figure over the horizontal axis.  
\item From the center of each box containing a $+$, drop a ``hook" up
and to the right, so that the arm and leg of the hook extend beyond
the east and north boundary of the Young diagram.  Consider the 
{\it hook graph} $H(L)$
formed by the set of all such hooks.
\item 
Make local modifications to $H(L)$ as 
in Figure \ref{local2}.
\begin{figure}[h]
\centering
\includegraphics[height=.7in]{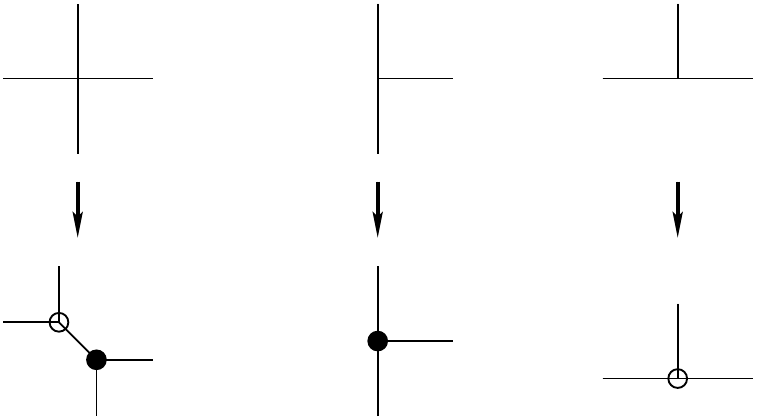}
\caption{}
\label{local2}
\end{figure}
\item The border labels from the first picture become labeled 
 ``boundary vertices;" they are labeled 
$1$ to $n$ in counterclockwise
order.  
After embedding the figure in a disk,
we have a plabic graph which we denote by $G_L^{\Le}$.
\end{enumerate}
\end{algorithm}
Figure \ref{LePlabic2} illustrates the steps of Algorithm \ref{Post-plabic}.
\begin{figure}[h]
\centering
\includegraphics[height=2.5in]{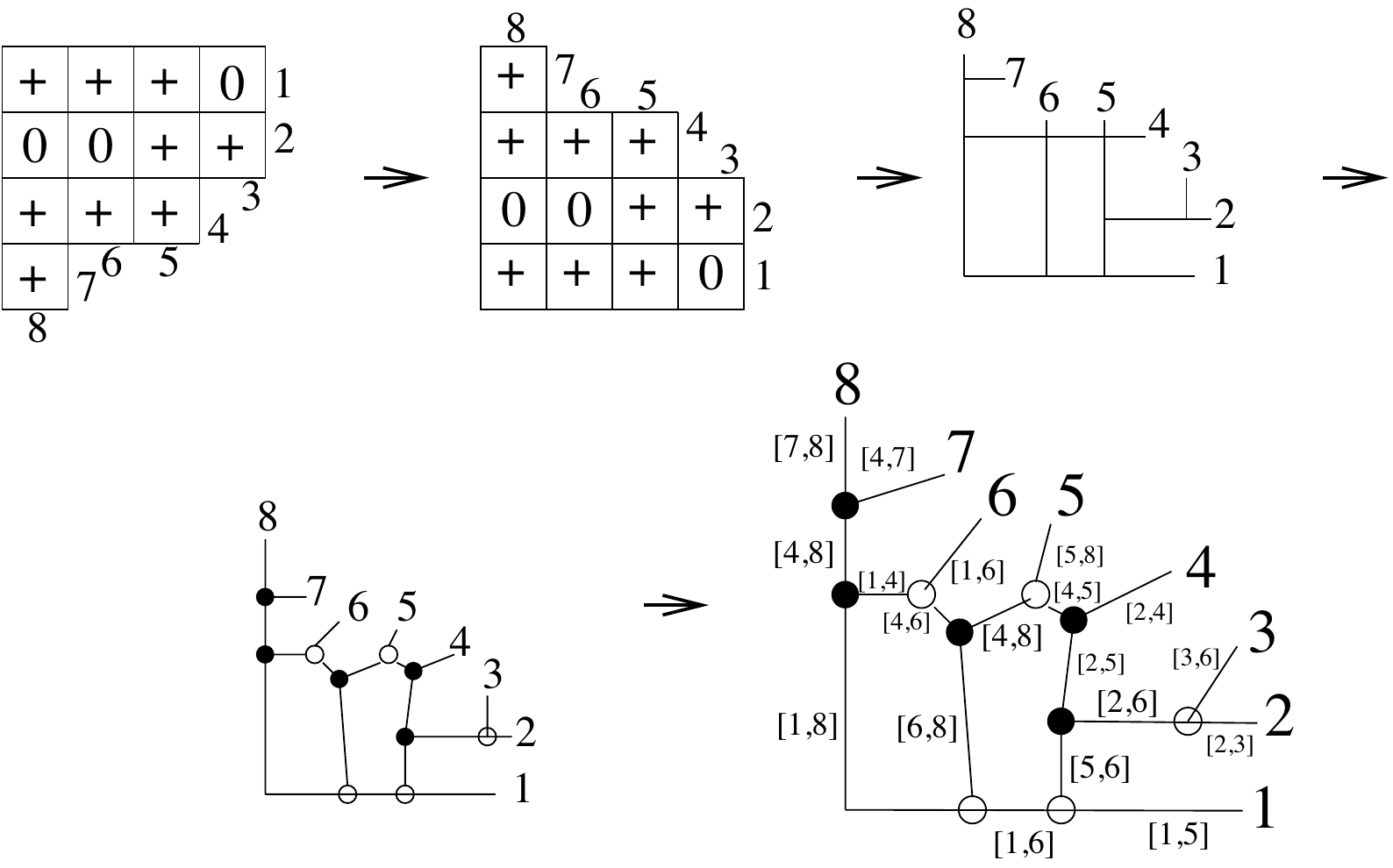}
\caption{A plabic graph for $\S_{\pi}^{tnn}$ with 
$\pi=(6,4,2,7,1,3,8,5)$}
\label{LePlabic2}
\end{figure}

\begin{proposition}\label{Le-slopes}
The plabic graph $G_L^{\Le}$
has the resonance property.
\end{proposition}
\begin{proof}
We prove this directly by analyzing the trips in the graph.
One can collapse the plabic graph $G_L^{\Le}$ back to the 
hook graph $H(L)$, and consider how the trips look in $H(L)$.
In $H(L)$, the trips have the following form:
\begin{itemize}
\item a trip which starts from the label of a vertical edge
in the original $\Le$-diagram first goes west as far as possible,
and then takes a zigzag path north and east, turning whenever 
possible.
\item a trip which starts from the label of a horizontal edge
in the original $\Le$-diagram first goes south as far as possible,
and then takes a zigzag path east and north, turning whenever
possible.
\end{itemize}
See Figure \ref{Zigzag} for a depiction of the general form
of the trips, as well as the trip which begins at $8$ in 
the example from Figure \ref{LePlabic2}.
\begin{figure}[h]
\centering
\includegraphics[height=.85in]{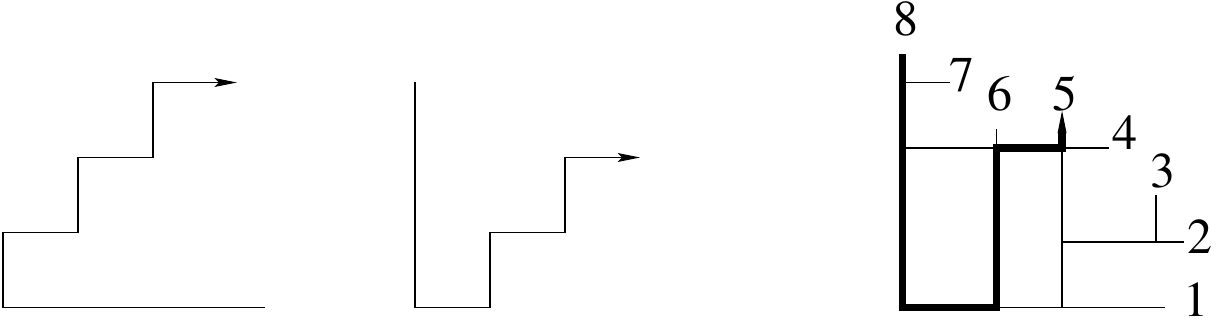}
\caption{}
\label{Zigzag}
\end{figure}

We now analyze the edge labeling around a black vertex in 
$G_L^{\Le}$  which comes from a trivalent vertex in the 
hook graph $H(L)$, see Figure \ref{Trips1}.  
By consideration of the zigzag shape of the trips, 
the trip which approaches the black vertex from above must come
straight south from some boundary vertex labeled $j$, while the trip
which approaches the black vertex from the right must come straight
west from some boundary vertex $i$. The trip which approaches the 
black vertex from below could have started from a boundary vertex $k$
which is either southeast of $i$ or west of $j$, see the first
two pictures in Figure \ref{Trips1}.  Either way, the resulting 
edge labeling will be as shown in the third picture in Figure \ref{Trips1}.
Clearly $i<j$, and either $k<i$ or $k>j$.  Therefore the edge labeling
in the third picture in Figure \ref{Trips1} satisfies the resonance
property.  

\begin{figure}[h]
\centering
\includegraphics[height=.8in]{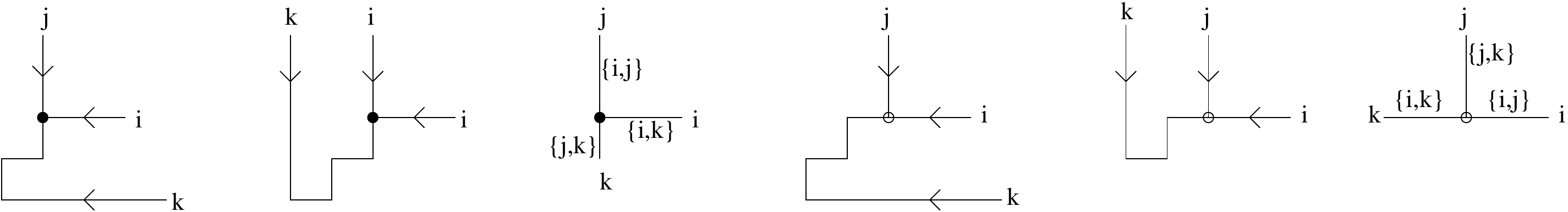}
\caption{}
\label{Trips1}
\end{figure}

The argument for a white vertex in
$G_L^{\Le}$  which comes from a trivalent vertex in the
hook graph $H(L)$ is analogous; see the fourth, fifth, and sixth 
pictures in Figure \ref{Trips1}.

Finally we analyze the edge labeling around a pair of vertices in 
$G_L^{\Le}$ which came from a degree $4$ vertex in $H(L)$.
In $H(L)$, the trip which approaches the vertex from above must come
straight south from some boundary vertex labeled $j$, while the trip
which approaches the vertex from the right must come straight
west from some boundary vertex $i$. As before, $i<j$.  
Let $k$ and $\ell$ denote the boundary vertices
whose trips approach the degree $4$ vertex from the left and below,
respectively.  
The resulting edge-labeling is shown 
at the right of Figure \ref{Trips2}.  

There are multiple possibilities for the trips
starting from $k$ and $\ell$;  
Figure \ref{Trips2} shows several of them.
The only restriction is that 
neither $k$ nor $\ell$
lies in between $i$ and $j$.  
In particular, one of the following must be true:
\begin{itemize}
\item $k<i<j$ and $\ell<i<j$
\item $i<j<k$ and $i<j<\ell$
\item $k<i<j<\ell$
\item $\ell<i<j<k$
\end{itemize}
In all cases, the edge labeling  in
Figure \ref{Trips2} satisfies the resonance property.
\begin{figure}[h]
\centering
\includegraphics[height=1.2in]{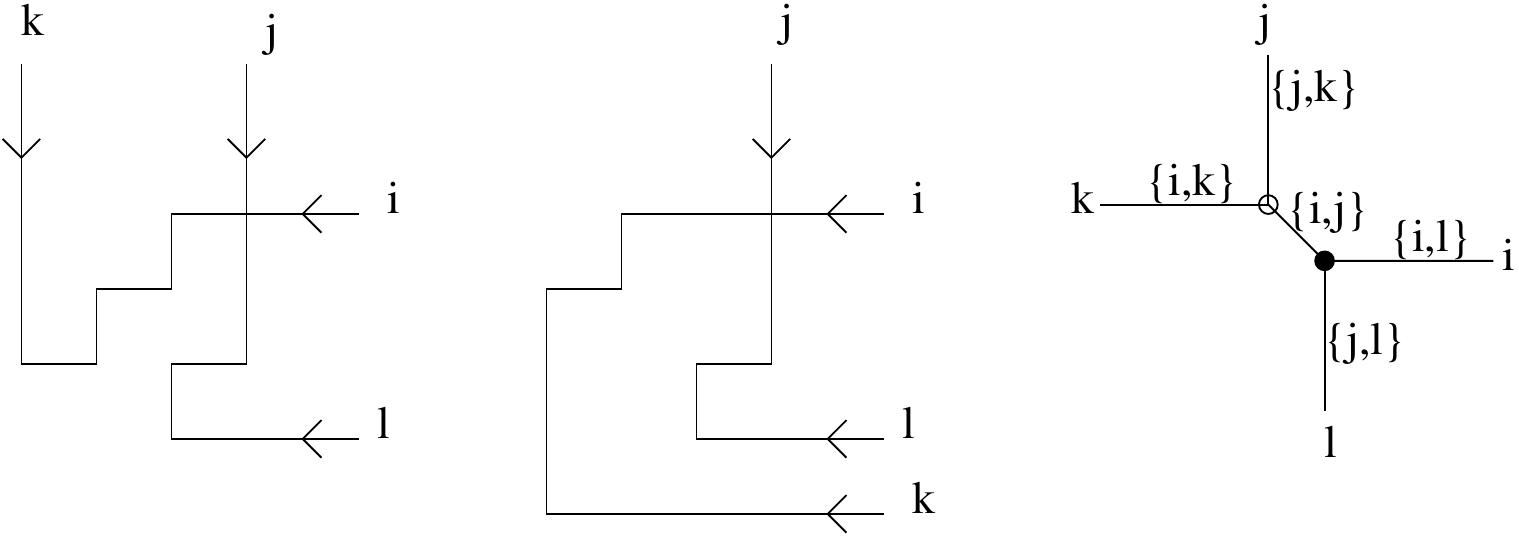}
\caption{}
\label{Trips2}
\end{figure}
\end{proof}


\begin{corollary}\label{cor:reduced}
Let $S_{\mathcal{M}}^{tnn}$ be a TP Schubert cell.  Then if 
the soliton graph $G_{\a}(\mathcal{M})$ is generic and has no $X$-crossings, 
it is a 
reduced plabic graph.  Moreover, \emph{every} generic soliton graph 
coming from the TP Grassmannian $(Gr_{k,n})_{>0}$ is a reduced plabic graph.
\end{corollary}

\begin{proof}
By Remark \ref{Schubert-labeling} and Theorem \ref{soliton-plabic},
every such graph is a plabic graph which satisfies the resonance property.  
The second statement is now a consequence of Corollary \ref{no-X}, which says
that a soliton graph from the TP Grassmannian has no $X$-crossings.
\end{proof}

\subsection{The connection to cluster algebras}
Cluster algebras are a class of commutative rings, introduced
by Fomin and Zelevinsky \cite{FZ}, which have a 
remarkable combinatorial structure.
Many coordinate rings of homogeneous spaces have
a cluster algebra structure: as shown by Scott \cite{Scott},
the Grassmannian is one such example.

\begin{theorem} \cite{Scott}\label{cluster-Grassmannian}
The coordinate ring of 
(the affine cone over) $Gr_{k,n}$  has  a 
cluster algebra structure.  Moreover, the set of Pl\"ucker coordinates
whose indices come from the labels of the regions
of a reduced plabic graph for $(Gr_{k,n})_{>0}$ 
comprises a \emph{cluster} for this cluster algebra.
\end{theorem}


\begin{remark}
In fact \cite{Scott} used the combinatorics of 
{\it alternating strand diagrams}, not reduced plabic graphs,
to describe clusters.
However, alternating strand 
diagrams are easily seen to be 
in bijection with reduced plabic graphs \cite{Postnikov}.
\end{remark}

\begin{theorem}\label{soliton-cluster}
The set of Pl\"ucker coordinates labeling regions of a
generic soliton graph for the TP Grassmannian
is a cluster for the cluster algebra
associated to the Grassmannian.
\end{theorem}

\begin{proof}
This follows from Corollary \ref{cor:reduced} and Theorem \ref{cluster-Grassmannian}.
\end{proof}

Conjecturally, every positroid cell $\S_{\pi}^{tnn}$
of the totally non-negative
Grassmannian also carries a cluster algebra structure, 
and the Pl\"ucker coordinates labeling the regions of any reduced
plabic graph for $\S_{\pi}^{tnn}$ should be a cluster for that cluster
algebra.  In particular, the TP Schubert cells should carry cluster algebra
structures.  Therefore we conjecture that  Theorem
\ref{soliton-cluster} holds with ``Schubert cell" replacing 
``Grassmannian."  Finally, there should be a suitable generalization
of Theorem \ref{soliton-cluster} for arbitrary 
positroid cells.

\section{The inverse problem for soliton graphs}\label{sec:inverse}
\label{inverse}
The {\it inverse} problem for soliton solutions of the KP equation
is the following:
given a time $\t$ together with the contour plot 
$\CC(u_A, \t)$ of a soliton 
solution, can one reconstruct the point $A$ of 
$Gr_{k,n}$ which gave rise to the solution?
Note that solving for $A$ is desirable, 
because this information would allow us to 
completely reconstruct the soliton solution.

In order to address the inverse problem we must 
work with the contour plots 
$\CC(u_A,\t)$ for finite times $\t$, as opposed to their limits, 
the asymptotic contour plots,  because
the former include information regarding
the values of the Pl\"ucker coordinates.
However, our results on the corresponding asymptotic contour plots
will be a crucial tool in the proofs of our results.

We will use the notation $|\t|\gg 0$ to indicate that 
 $\|\t\|$ is ``large enough" so that 
the contour plot
$\CC(u_A,\t)$ has the same topology as 
the corresponding asymptotic contour plot
$\CC_{\t}(\M)$ for $A \in S_{\M}$, i.e.
there is a bijection between the regions of the complements 
of the contour plots with the property that corresponding regions are labeled
by the same dominant exponential.


We will solve the inverse problem in two different situations:
\begin{itemize}
\item When $A \in (Gr_{k,n})_{>0}$ and $|\t| \gg 0$
(see Theorem \ref{inverse2}), and 
\item When $A \in (Gr_{k,n})_{\geq 0}$, $|t_3| \gg 0$, and 
$t_4 = \dots = t_m = 0$ (see Theorem \ref{inverse1}).
\end{itemize}

We will write $t$ instead of $\t$ when we are assuming that 
$t_i = 0$ for $i \geq 4$.

\begin{lemma}\label{lem:reconstructPlucker}
Fix generic real parameters $\kappa_1<\dots<\kappa_n$, and 
consider a generic contour plot $\CC(u_A, \t)$ of a soliton solution
coming from a point $A$ of $(Gr_{k,n})_{\geq 0}$ with $|\t |\gg 0$.
Then from the $\kappa_i$'s, 
the contour plot, and  $\t$, we 
can identify what cell $S_{\pi}^{tnn}$ the element $A$ comes from,
and reconstruct the labels of the dominant exponentials and 
the values of all Pl\"ucker coordinates corresponding
to these dominant exponentials in the contour plot.
\end{lemma}

\begin{proof}
Since the pairwise sums of the $\kappa_i$'s are all distinct, 
we can determine from the contour plot precisely how to label
each line-soliton with a pair $[i,j]$.   
Note here that some of the edges in $\CC(u_A,t)$
correspond to the phase shifts. However, those can be easily identified by 
checking the types of solitons at the intersection point, since
a phase shift appears as the interaction of
two solitons of $[i,j]$- and $[\ell,m]$-types with either
$i<j<\ell<m$ or $i<\ell<m<j$, and its length goes to $0$ when
we take the limit $\lim_{s \to \infty} \CC(u_A, s\a)$.
 From the labels of the 
unbounded line-solitons, we can use Theorem \ref{perm-asymp}
to determine
which positroid cell $S_{\pi}^{tnn}$ the element $A$ comes from.
Finally, we can label the dominant exponentials by using Lemma \ref{separating},
together with the fact that for 
$x\ll0$, the dominant exponential is $E_I$, where $\Delta_I$ is the lexicographically
minimal Pl\"ucker coordinate which is nonzero on $S_{\pi}^{tnn}$ ($I$ is the
set of excedance positions of $\pi$).

Now recall that the equation of each line-soliton
is given by (\ref{eq-soliton}). Since each line-soliton
has been labeled by $[i,j]$, and we know the $\kappa_i$'s,
we can solve for all ratios of Pl\"ucker coordinates
which appear as labels of adjacent regions of the contour plot.
The Pl\"ucker coordinates are only defined up to a simultaneous
scalar, so without loss of generality we set the Pl\"ucker coordinate $\Delta_I$ 
with lexicographically
minimal index set $I$ equal to $1$.
\end{proof}

Using the cluster algebra structure 
for Grassmannians, we can now  prove the following.

\begin{theorem}\label{inverse2}
Consider a generic contour plot $\CC(u_A,\t)$ of a soliton solution
which
comes from a point $A$ of the TP Grassmannian
and a multi-time vector $\t$ such that $|\t|\gg 0$.
Then from the contour plot together with $\t$ we 
can uniquely reconstruct the point $A$.
\end{theorem}

\begin{proof}
By Lemma \ref{lem:reconstructPlucker}, we can 
compute the labels of the dominant exponentials in the contour plot and 
the values of the corresponding Pl\"ucker coordinates.
And by Theorem \ref{soliton-cluster}, the set of dominant exponentials
labeling $\CC(u_A,t)$ forms a cluster $\mathbf c$ for the cluster algebra 
$\mathcal A$
associated to the Grassmannian.  Since the coordinate ring of the 
Grassmannian is a cluster algebra (whose cluster variables include 
the set of all Pl\"ucker coordinates), it follows that we can 
express each Pl\"ucker coordinate of $A$ as a Laurent polynomial
in the elements of $\mathbf c$.  Therefore we can determine 
the element $A \in (Gr_{k,n})_{>0}$ itself.  
\end{proof}

\begin{remark}
We have written the proof of Theorem \ref{inverse2} using the language
of cluster algebras, because we expect  it should be possible
to generalize some of the results of this paper to other integrable systems
associated to cluster algebras.  However, the combinatorial heart
of the proof is the following argument (which indeed is part of 
Scott's proof \cite{Scott} 
that the Grassmannian has a cluster algebra structure): 
any two reduced plabic graphs for a given positroid cell
are connected via a sequence of moves (M1), (M2), (M3)
\cite[Theorem 12.7]{Postnikov}, and the non-trivial move (M1)
corresponds to a three-term Pl\"ucker relation.  Also, every Pl\"ucker
coordinate occurs in some reduced plabic graph for the TP Grassmannian
\cite{OPS}.  
Therefore from the values of the Plucker coordinates
$\Delta_I(A)$ for all $I$ labeling the faces of some reduced plabic
graph for the TP Grassmannian, we can reconstruct $A$.
\end{remark}

\begin{theorem}\label{inverse1}
Fix generic real parameters $\kappa_1<\dots<\kappa_n$.
Consider a generic contour plot $\CC(u_A,t)$ of a soliton solution
coming from a point $A$ of a positroid cell $S_{\pi}^{tnn}$,
for $|t|\gg0$.  Then from the contour plot together with $t$ we 
can uniquely reconstruct the point $A$.
\end{theorem}

Theorem \ref{t<<0} will be instrumental in proving 
Theorem \ref{inverse1}.  
However, we first remind the reader (see Remark \ref{rem:slide})
that the combinatorics of the $X$-crossings in the contour 
plot may differ from the combinatorics of the $X$-crossings in the 
graph $G_-(L)$.  We will describe these differences using the 
following notion of \emph{slide}.

\begin{definition}\label{def:slide}
Consider a generalized plabic graph $G$ with at least one $X$-crossing.
Let $v_{a,b,c}$ be a trivalent vertex 
(with edges labeled $[a,b]$, $[a,c]$, and $[b,c]$)
which has a small neighborhood $N$ containing one or two $X$-crossings
with a line labeled $[i,j]$, but no other trivalent vertices or $X$-crossings.
Here $\{a,b,c\}$ and $\{i,j\}$ must be disjoint.  
Then a \emph{slide} is a local deformation of the graph $G$ which moves 
the line so that it intersects a different set of edges of $v_{a,b,c}$,
creating or destroying at most one region in the process.
\end{definition}

\begin{figure}[h]
\centering
\includegraphics[height=1.6in]{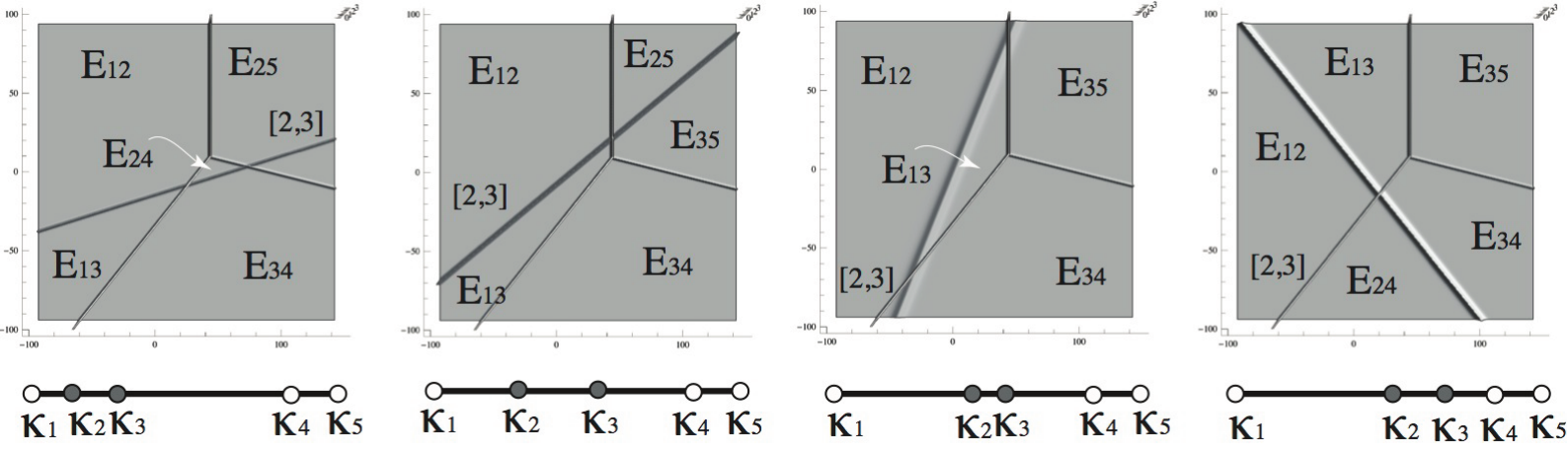}
\caption{Examples of slides. These contour plots correspond to the same Le-diagram $L$ with $\pi(L) = (5,3,2,1,4)$, 
but these plots all differ from $G_-(L)$.}
\label{fig:slide}
\end{figure}

\begin{lemma}\label{lem:slide}
Consider the contour plot $\CC(u_A,t)$ of a soliton solution
coming from $A\in S_{L}^{tnn}$ at $t\ll0$.  Then its soliton graph 
either coincides with $G_-(L)$ or differs from it via a series of slides.
Similarly, if one considers two contour plots of a soliton solution 
coming from $A\in S_{L}^{tnn}$ at $t\ll0$, which are computed using different
sets of parameters
$(\kappa_1,\dots,\kappa_n)$, 
then one can be obtained from the other 
via a sequence of slides.
\end{lemma}
\begin{proof}
Theorem \ref{t<<0} gives a precise description of the trivalent vertices
which appear in the contour plot, together with the topology of 
how they are connected to each other (including the circular order
of the three line-solitons 
$[a,b]$, $[a,c]$, and $[b,c]$ incident to a given vertex $v_{a,b,c}$).
In particular, the combinatorics of the trivalent vertices and their incident 
edges in the contour plot
is precisely that which appears in $G_-(L)$.  

The only feature of the contour plot which 
Theorem \ref{t<<0} does not determine is which pairs of line-solitons 
form an $X$-crossing.  Therefore  
the topology of the contour plot may differ from that of $G_-(L)$
via a sequence of slides.  In each slide, a line-soliton of type
$[i,j]$ may pass across a trivalent vertex $v_{a,b,c}$, changing the 
location of the $X$-crossing, or possibly replacing one $X$-crossing
with two $X$-crossings (or vice-versa).
Note that the indices $\{i,j\}$ and $\{a,b,c\}$ must be disjoint, 
since an $X$-crossing involves two lines-solitons with disjoint indices.
\end{proof}

\begin{lemma}\label{lem:slide2}
Consider two contour plots which differ by a single slide.
Let $S_1$ and $S_2$ denote the two sets of Pl\"ucker coordinates
corresponding to the dominant exponentials in the two contour plots.
Then from the values of the Pl\"ucker coordinates in $S_1$,
one can reconstruct the values of the Pl\"ucker coordinates in $S_2$,
and vice-versa.
\end{lemma}
\begin{proof}
Recall from Theorem \ref{th:X-crossing} that the four Pl\"ucker coordinates
incident to an $X$-crossing are dependent, in particular they satisfy 
a ``two-term" Pl\"ucker relation.
Now it is easy to verify the lemma by inspection, since each slide only creates or removes 
one region, and there is a dependence among the Pl\"ucker coordinates
labeling the dominant exponentials.  The reader may wish to check this 
by looking at the first and second, or the second and third, or the third and fourth 
contour plots in Figure 
\ref{fig:slide}.
\end{proof}

We now turn to the proof of Theorem \ref{inverse1}.

\begin{proof}
We consider the case that $t \ll 0$.
By Lemma \ref{lem:reconstructPlucker}, we can 
identify the cell $S_{\pi}^{tnn} = S_{L}^{tnn}$ containing $A$.
We can also compute the labels of the dominant exponentials 
in the contour plot and 
the values of the corresponding Pl\"ucker coordinates.

Define 
\[
\mathcal{T} = \left\{~ I \in {[n] \choose k}~\Big|~
E_I \text{ is a dominant exponential in }
\CC(u_A,t)~\right\}.
\]
We now claim that from the values of each 
$\Delta_I(A)$ for $I \in \mathcal{T}$,  we can determine
$A$. 

To prove the claim, note that 
since $t\ll 0$ sufficiently small, the contour plot 
$\CC(u_A,t)$ has the same topology as $\CC_{-}(\mathcal{M})$.
Therefore $\CC(u_A,t)$ will either have the same topological structure
as the graph $G_-(L)$, or by Lemma \ref{lem:slide}, it will differ from
$G_-(L)$ via a sequence of slides.  And so by Lemma \ref{lem:slide2}, 
we can determine the values of the Pl\"ucker coordinates labeling the regions 
in $G_-(L)$.  

We now use
Theorem  \ref{identifyPlucker}, below, 
which shows that from the values of the 
Pl\"ucker coordinates labeling the regions 
in $G_-(L)$, 
one can reconstruct all nonzero Pl\"ucker coordinates.
This determines the point $A \in 
(Gr_{k,n})_{\geq 0}$.

This completes the proof of Theorem \ref{inverse1} for $t \ll 0$.
The proof of the theorem for $t \gg 0$ just follows the 
same argument, with the dual $\Le$-diagram (which is a relabeled
$\Le$-diagram) replacing the $\Le$-diagram in the construction of the 
asymptotic contour plot.
\end{proof}


\subsection{Reconstructing a point of $(Gr_{k,n})_{\geq 0}$ from a minimal 
set of Pl\"ucker coordinates}

Talaska \cite{Talaska} studied the problem of 
how to reconstruct an element $A\in (Gr_{k,n})_{\geq 0}$ from a subset
of its Pl\"ucker coordinates $\Delta_I(A)$.
For each cell $S_{\mathcal M}^{tnn}$, she 
characterized a minimal 
set of Pl\"ucker coordinates  which suffice to reconstruct 
the corresponding element of $S_{\mathcal M}^{tnn}$.  Her 
proof worked by explicitly inverting Postnikov's {\it boundary
measurement map} \cite{Postnikov}.
In this section we review some of Talaska's work.  

Given a $\Le$-diagram $L$ of shape $\lambda$ which fits inside a 
$k \times (n-k)$ rectangle, we construct a planar
network $N_L$ as follows.  
First we draw a disk whose boundary consists of  the north
and west edges of the $k \times (n-k)$ box and the path 
determining the southeast boundary of $\lambda$.  Place a vertex, 
called a {\it boundary source}, at the end of each row,
and a vertex, called a {\it boundary sink}, at the end of each column
of $\lambda$.  Label these in sequence 
with the integers $\{1,2,\dots,n\}$, following the path from 
the northeast corner to the southwest corner which determines
$\lambda$.  Let $I=\{i_1<i_2<\dots<i_k\}$ be the set of boundary 
sources, so that $[n]\setminus I = \{j_1<j_2<\dots<j_{n-k}\}$
is the set of boundary sinks.

Given a box $b$ in $L$ which contains a $+$,
we drop a hook down and to the 
right, extending the two segments comprising the hook all the way to 
the boundary of the disk.  We direct the horizontal segment left and 
the vertical segment down.  After forgetting the $+$'s and $0$'s in $L$,
we now have a planar directed network.  There is exactly one face
for each box $b$ in $\lambda$ containing a $+$, and in addition,
there is one face whose northwest boundary is the boundary of the disk.
See the first two pictures in 
Figure \ref{network}, which shows this construction
applied to the $\Le$-diagram from Figure \ref{LePlabic}.
\begin{figure}[h]
\centering
\includegraphics[height=1.3in]{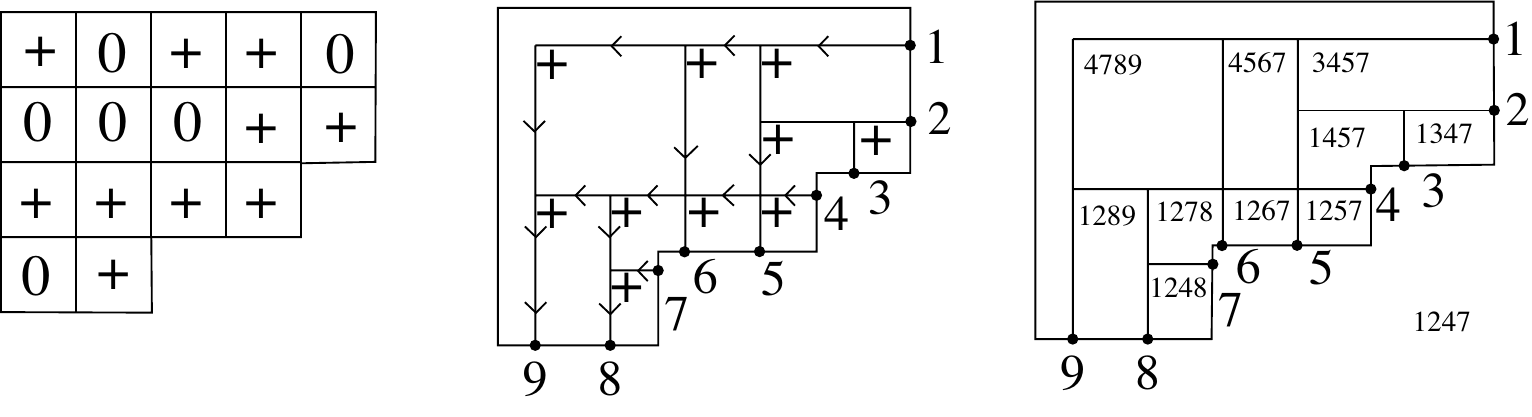}
\caption{A minimal set of Pl\"ucker coordinates for $\pi=(7,4,2,9,1,3,8,6,5)$.}
\label{network}
\end{figure}

Talaska proved the following.

\begin{theorem}\label{destination-set} \cite[Corollary 4.2]{Talaska}, 
\cite[Lemma 2.1]{Talaska}
Let $L$ be the $\Le$-diagram of a positroid cell $S_{L}^{tnn}$,
and $A \in S_{L}^{tnn}$. 
Define a set of Pl\"ucker coordinates
$T(L):=\{\Delta_I(A)\} \cup \{\Delta_{J(b)}(A) \}_b$, where $b$ ranges over all boxes containing a $+$ in the 
$\Le$-diagram, and $J(b)$ is the destination set
of the northwest-most path collection
lying weakly southeast of the $b$ hook.
Then $T(L)$ is a \emph{totally positive base}; that is,
each nonzero Pl\"ucker coordinate on $S_{L}^{tnn}$ can be written as a 
subtraction-free expression in the elements of $T(L)$.  It follows that
one can reconstruct $A$ from $T(L)$.
\end{theorem}

The right side of Figure \ref{network} shows the Pl\"ucker coordinate
$\Delta_{J(b)}$ associated with each box $b$ which contained a $+$
in the $\Le$-diagram.

\begin{definition}
Given a $\Le$-diagram $L$ and a box $b$ containing a $+$, let 
$L_b$ be the $\Le$-diagram obtained from $L$ by changing the content
of each box to a $0$ if the box  lies in a row above $b$ or a column left of 
$b$.
\end{definition}

\begin{definition}
We define a total order on elements of ${[n] \choose k}$ 
by saying that 
$J^1 = \{j^1_1 > \dots > j^1_k\} > J^2 = \{j^2_1>\dots > j^2_k\}$
if in the first position $i$ where they differ, $j^1_i > j^2_i$.
Then if $\mathcal M \subset {[n] \choose k}$, we refer to the largest
element of $\mathcal M$ as \emph{lexicographically maximal}.
\end{definition}

The following lemma is implicit in \cite{Talaska}.
\begin{lemma}\label{inv1}
Let $\mathcal M:=\mathcal M(L_b)$ be the collection of Pl\"ucker coordinates which 
are positive on the cell $S_{L_b}^{tnn}$.  Then 
$\Delta_{J(b)}$ is the lexicographically maximal element of $\mathcal M$.
\end{lemma}

The following lemma is obvious.
\begin{lemma}\label{inv2}
The generalized plabic graph $G_-(L_b)$ that Algorithm \ref{LeToPlabic}
associates to $L_b$ is contained inside $G_-(L)$.
\end{lemma}

\begin{lemma} \label{inv3}
Consider the graph $G_-(L)$ constructed
by Algorithm \ref{LeToPlabic}.  Let $J$ be the $k$-element 
subset of $[n]$ which labels the region $R$ that comes from the 
northwest corner of the $\Le$-diagram $L$ (see the 
bottom left picture in Figure \ref{LePlabic}).
Then $J$ is the lexicographically maximal element of $\mathcal M(L)$.
\end{lemma}

\begin{proof}
One may prove directly that the label of $R$ coincides with the destination
set of the northwest-most path collection in $N_L$.

Alternatively, we may prove this using a soliton argument.
Let $A\in S_{L}^{tnn}$.
When $x \gg0$, we have that 
$E_n > E_{n-1} > \dots > E_1$.  Therefore if 
$J$ is the lexicographically maximal element of $\mathcal M(L)$,
then the term $\Delta_J(A) E_J$ dominates the $\tau$-function
$\tau_A$.  It follows that $E_J$ is the dominant exponential
of any contour plot $\CC(u_A,t)$ in the region where
$x\gg0$. By Lemma \ref{lem:slide}, for $t \ll 0$,
this contour
plot coincides  with $G_-(L)$ up to a series of slides,
none of which will change the label of the region at $x\gg 0$.
And this region corresponds to the region coming from the 
northwest corner of the $\Le$-diagram $L$.
\end{proof}

\begin{theorem}\label{identifyPlucker}
Let $L$ be a $\Le$-diagram.
Consider the set $S$ of Pl\"ucker coordinates
associated to the dominant exponentials in
the regions of the graph $G_-(L)$,
as constructed
in Algorithm \ref{LeToPlabic}. 
Then $S$ contains $T(L)$.
In particular,
one may reconstruct the element $A\in S_{L}^{tnn}$ from 
the values of the Pl\"ucker coordinates labeling 
$G_-(L)$ for $t\ll 0$.
\end{theorem} 

\begin{proof}
By Lemmas \ref{inv1}, \ref{inv2}, and \ref{inv3},
for each box $b$ in $L$, 
$E_{J(b)}$ is a dominant exponential labeling a region of 
$G_-(L)$.  Also, if $I$ is the lexicographically minimal 
element of $\mathcal M(L)$, then $E_I$ labels 
the region at $x\ll 0$ in $G_-(L)$.  Therefore 
$T(L) \subset S$.  The proof now follows from 
Theorem \ref{destination-set}.
\end{proof}


\section{Triangulations of $n$-gon and soliton graphs for $(Gr_{2,n})_{>0}$} 
\label{sec:triangulation}

The main result of this section is
an algorithm for producing all generic soliton graphs
(up to the (M2)-equivalence of Section \ref{sec:moves})
that come from $(Gr_{2,n})_{>0}$.  We consider the generic 
asymptotic contour plot
$\CC_{\a}(\mathcal{M})$ for each fixed $\a=(a_3,\ldots,a_m)$, and
classify those plots by taking various choices of $\a$ in the $\R^{m-2}$-space.
Here we will show that it is sufficient to consider $m=n-1$ for $n\ge4$.
Note that by Corollary \ref{no-X}, these asymptotic contour plots 
do not have any $X$-crossings, and hence their generic soliton graphs have only trivalent vertices.
We state the algorithm and main theorem
in Section \ref{sec:theorem}, and then prove it in the rest of the section.

\subsection{Algorithm to produce soliton graphs}\label{sec:theorem}


\begin{algorithm}  \label{TriangulationA}
Let $T$ be a triangulation of an $n$-gon $P$, whose $n$ vertices are labeled
by the numbers $1,2,\dots,n$, in counterclockwise order. Therefore
each edge of $P$ and each diagonal of $T$ is specified by a pair of distinct integers
between $1$ and $n$. The following
procedure yields a labeled graph  $\Psi(T)$.
\begin{enumerate}
\item Put a black vertex in the interior of each triangle in $T$.
\item Put a white vertex at each of the $n$ vertices of $P$
which is incident to a diagonal of $T$; put a black vertex
at the remaining vertices of $P$.
\item Connect each vertex which is
inside a triangle of $T$ to the three vertices
of that triangle.
\item Erase the edges of $T$, and contract every pair of  adjacent vertices
which have the same color.  This produces a new graph $G$
with $n$ boundary vertices, in bijection with the vertices of the original
$n$-gon $P$.
\item Add one unbounded ray to each of the boundary vertices of
$G$, so as to produce a new (planar) graph  $\Psi(T)$.  Note that
$\Psi(T)$ divides the plane into regions; the bounded regions correspond to
the diagonals of $T$, and the unbounded regions correspond to the edges of $P$.
More specifically, a region of $\Psi(T)$ is labeled by $E_{ij}$,
where $i$ and $j$ are the endpoints of the corresponding diagonal
or edge of $T$.
\end{enumerate}
\end{algorithm}
See Figure \ref{Psi}.
Our main result in this section is the following.

\begin{theorem}\label{theorem:maintriangulation}
Up to (M2)-equivalence, the graphs $\Psi(T)$ constructed above are 
soliton graphs
for $(Gr_{2,n})_{>0}$, and conversely, up to (M2)-equivalence,
any generic soliton graph
for $(Gr_{2,n})_{>0}$ comes from this construction.
Moreover, one can realize each graph $\Psi(T)$ by 
choosing 
$\a=(a_3,\dots,a_n)$ appropriately.\end{theorem}
\begin{figure}[h]
\centering
\includegraphics[height=2.2in]{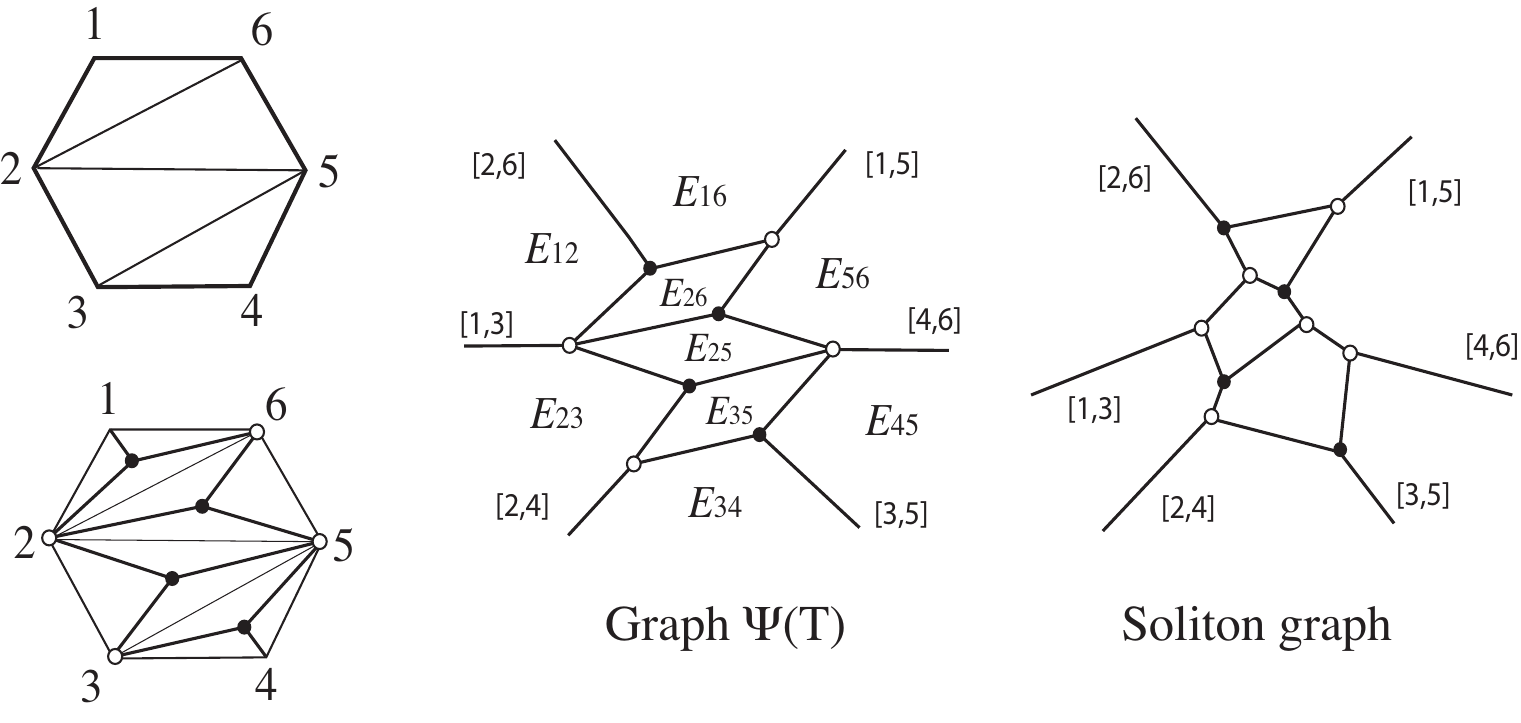}
\caption{Algorithm \ref{TriangulationA}, starting from a triangulation of
a hexagon.  The right figure shows the corresponding soliton graph 
from the contour plot for a soliton solution
from  $(Gr_{2,6})_{>0}$.  The graph $\Psi(T)$ is (M2)-equivalent to
the soliton graph.
\label{Psi}}
\end{figure}
\begin{remark}
The process of flipping a diagonal in the triangulation corresponds
to a mutation in the cluster algebra.  In the terminology of reduced plabic
graphs, a mutation corresponds to the square move (M1) (see Section \ref{sec:moves}).
In the setting of KP solitons, each mutation may be considered as an
evolution along a particular flow of the KP hierarchy.  For example,
the 
contour plot for $(Gr_{2,4})_{>0}$ for $t_3\ll 0$ has one bounded region with the dominant exponential $E_{1,3}$, and as $t_3$ increases,
the bounded region closes at some $t_3$, then for $t_3\gg 0$, the contour plot
has again one bounded region with the dominant exponential $E_{2,4}$.
(This can be easily verified using the construction
given in Section \ref{plabic-soliton}.)
Thus, the time $t_3$ can be considered as a mutation parameter for $A_1$
cluster, i.e. $\Delta_{1,3}\to\Delta_{2,4}$ as $t_3$ increases.
\end{remark}

\begin{remark}\label{rem:suffices}
It is known already that the reduced plabic graphs for 
$(Gr_{2,n})_{>0}$ all have the form given by Algorithm \ref{TriangulationA}
(up to (M2)-equivalence).  And 
by Corollary \ref{cor:reduced},
every generic soliton graph is a reduced plabic
graph.  Therefore 
it follows immediately that every
generic soliton graph for $(Gr_{2,n})_{>0}$ must have the form of Algorithm 
\ref{TriangulationA}.
Therefore in order to prove 
Theorem \ref{theorem:maintriangulation}, we must show that each 
graph that one can obtain from Algorithm \ref{TriangulationA} 
is realizable as a soliton
graph (up to (M2)-equivalence).
\end{remark}


\subsection{Sequence of asymptotic contour plots}

Let $\mathcal{M} = {[n] \choose 2}$, the set of all two-element subsets
of $[n]$.
Recall that for fixed $\a=(a_3,\ldots,a_m)$ the asymptotic contour plot $\CC_{\a}(\mathcal{M})$ 
for  $S_{\mathcal{M}}^{tnn} = (Gr_{2,n})_{>0}$ is defined as the 
locus in the $\bar{x}\bar{y}$-plane where 
\[
f_{\mathcal{M}}(\bar{x},\bar{y},\a):=\underset{1\le i<j\le n}{\text{max}}\left\{\Theta_{i,j}(\bar{x},\bar{y},\a)\right\}
\]

is not linear.  Here $\Theta_{i,j}=\theta_i+\theta_j$ and
\[
\theta_{j}(\bar{x},\bar{y},\a)=\kappa_j\bar{x}+\kappa_j^2\bar{y}+\sum_{p=3}^m\kappa_j^pa_p,\qquad\text{for}\quad j=1,\ldots,n.
\]
Since $\mathcal{M}= {[n] \choose 2}$ will be fixed throughout this section,
we will hereafter  refer to 
$\CC_{\a}(\mathcal{M})$ as simply $\CC_{\a}$.

For each fixed $\a$, we 
consider the $\binom{n}{2}$ planes defined by $z=\Theta_{i,j}(\bar{x},\bar{y},\a)$
in the $(\bar{x},\bar{y},z)$-space. 
Then the asymptotic contour plot $\CC_{\a}$ represents
the $\bar{x}\bar{y}$-projection of the intersection lines of those planes 
which take the largest values of $z$.  We call those planes the \emph{dominant planes}. Note that:
\begin{itemize}
\item If a point 
$(\bar{x},\bar{y})$ 
is not in $\CC_{\a}$
then there is a unique dominant plane at that point.
\item Let 
$L_{i,j}$ be the line in the $\bar{x} \bar{y}$-plane where
$\theta_i = \theta_j$.  Every line segment
of $\CC_{\a}$ (either bounded or unbounded)
is a portion of some $L_{i,j}$.
\item Given $\{i,j,k\} \subset [n]$,
let $v_{i,j,k} = (v_{i,j,k}^x,v_{i,j,k}^y)$ be the point
where the lines $L_{i,j}$, $L_{i,k}$, and $L_{j,k}$ mutually intersect.
All trivalent vertices of $\CC_{\a}$ have the form
$v_{i,j,k}$ for some $i,j,k$.
\item If a point 
lies on a line of $\CC_{\a}$ then there are
two dominant planes whose intersection gives
the line, which breaks the linearity of $f_{\mathcal{M}}(\bar{x},\bar{y},\a)$.
\end{itemize}

\begin{lemma}\label{lem:order}
Let $(r_1,\dots,r_n)$ be any point in $\R^n$.  Then we can find a unique point $(\bar{x}, \bar{y}, \mathbf{a}) =
(\bar{x},\bar{y}, a_3,\dots,a_{n-1}) \in \R^{n-1}$ and a 
constant $a_0$ such that 
$\theta_i(\bar{x}, \bar{y}, \a) = r_i-a_0$ for all $1 \leq i \leq n$.
\end{lemma}

\begin{proof}
Define the function $\ell_i: \R \times \R^{n-1} \to \R$ by 
$$\ell_i(a_0,\bar{x}, \bar{y}, \a)  = a_0 + \theta_i(\bar{x},\bar{y},\a)
= (a_0,\bar{x},\bar{y},a_3,\dots, a_{n-1}) \cdot \mathsf E_i^{\bf 0},$$ where 
$\mathsf E_i^{\bf 0}$ is the $i$th column vector defined in \eqref{eq:E}.
Recall that the matrix $E^{\bf 0}=(\mathsf E_1^{\bf 0},\ldots,\mathsf E_n^{\bf 0})$ is the Vandermonde matrix in the $\kappa_j$'s, and it is invertible.
Therefore  we can find a unique solution 
$(a_0,\bar{x},\bar{y},\a)$
to the equations $(a_0,\bar{x},\bar{y},\a)\cdot E^{\bf 0}=(r_1,\ldots,r_n)$, and
hence 
$a_0+ \theta_i(\bar{x},\bar{y},\a) = r_i$ for all $i$.
\end{proof}

Lemma \ref{lem:order}
implies that  the relative positions of  the planes $z=\Theta_{i,j}(\bar{x},\bar{y},\a)$ at each point $(\bar{x},\bar{y})\in \R^2$
are determined by the  $n-3$ parameters $\a=(a_3,\ldots,a_{n-1})\in\R^{n-3}$.

In Theorem \ref{thm:triangle} below we will give a procedure
for realizing each graph $\Psi(T)$ as a soliton graph 
(up to (M2)-equivalence).  Note that by Remark \ref{rem:suffices}, Theorem 
\ref{theorem:maintriangulation} is a consequence of Theorem 
\ref{thm:triangle}.

First we need some definitions.
Consider an $n$-gon, whose vertices
are labeled from $1$ to $n$ in counterclockwise order,
and let $I \subset [n]$.
Choose some $j \in [n]\setminus I$.  We say that the
\emph{clockwise} (respectively, \emph{counterclockwise}) \emph{neighbor} 
of $j$ in $I$ is the element of
$I$ which is closest to $j$ when we travel from $j$ in clockwise
(respectively, counterclockwise)
order around the boundary of the polygon.

\begin{definition}
Let $\mathcal{T}_n$ denote the triangulations of an $n$-gon, whose
vertices are labeled from $1$ to $n$ in counterclockwise order.
We define a surjective map $\lambda:S_n \to \mathcal{T}_n$ as follows.
Let $(i_1,\dots,i_n) \in S_n$.  We construct an element of
$\mathcal{T}_n$ by:
\begin{itemize}
\item creating a triangle with vertices $\{i_1,i_2,i_3\}$
\item For $4 \leq k \leq n$, add two edges connecting
$i_k$ to its clockwise and counterclockwise neighbors from
$\{i_1,i_2,\dots,i_{k-1}\}$.
\end{itemize}
\end{definition}

\begin{theorem}\label{thm:triangle}
Choose a point $(r_1,\dots,r_n)\in \R^n$ such that 
\[
0=r_{i_1}=r_{i_2}=r_{i_3}\gg r_{i_4}\gg r_{i_5}\gg\cdots\gg r_{i_n}.
\]
More specifically,
choose some large real number $R$ such that
\begin{align*} 
R &\gg \underset{i\ne j\ne k}{\max}\left\{1,~ \frac{1}{|\kappa_k-\kappa_i||\kappa_k-\kappa_j|}, ~\frac{|\kappa_i+\kappa_j|}{|\kappa_k-\kappa_i||\kappa_k-\kappa_j|} \right\}.
\end{align*}
We then take
\[
r_{i_{\ell}}=-R^{\ell}\qquad\text{for}\quad \ell=4,\ldots n.
\]
Using Lemma \ref{lem:order},  choose the unique point
$(a_0, \overline{x}_0, \overline{y}_0, a_3,\dots,a_{n-1})$ such that
$\theta_i(\overline{x}_0, \overline{y}_0, \a) = r_i - a_0$
for all $1 \leq i \leq n$.
Then the soliton graph associated to 
the asymptotic contour plot $\CC_{\a}$ 
is (M2)-equivalent to the
graph $\Psi(\lambda(i_1,i_2,\dots,i_n))$.
\end{theorem}

\begin{remark}
Other choices of the $r_i$'s in Theorem \ref{thm:triangle} are possible,
but the above choice will be convenient for our proofs.
\end{remark}

From now on we use the hypotheses of Theorem \ref{thm:triangle}.
\begin{definition}
For $\ell = 3,4, \dots, n$, we define
\[
I_\ell:=\{i_1,i_2,\ldots,i_{\ell}\},
\]
and 
\[
f^{(\ell)}(\bar{x},\bar{y},\a):=\underset{i,j\in I_\ell}{\text{max}}\left\{
\Theta_{i,j}(\bar{x},\bar{y},\a)\right\}.
\]
And we define $\CC_{\a}^{(\ell)}$
to be the locus in the $\bar{x}\bar{y}$-plane where $f^{(\ell)}(\bar{x},\bar{y},\a)$ is not linear.
\end{definition}

We will prove Theorem \ref{thm:triangle} by induction on $n$.
In order to justify the fact that we can build each contour plot
inductively, we need the following result.

\begin{proposition}\label{prop:induction}
There is a sequence of  regions $\mathcal{R}_{I_\ell}$
\[
\mathcal{R}_{I_3}\subset \mathcal{R}_{I_4}\subset\cdots\subset \mathcal{R}_{I_n}=\R^2
\]
in the $\bar{x} \bar{y}$-plane
so that at any point $(\bar{x},\bar{y})\in \mathcal{R}_{I_\ell}$, 
we have the order
\begin{equation}\label{eqn:order}
\theta_{j}(\bar{x},\bar{y},\a)~ \gg~
\theta_{i_{\ell+1}}(\bar{x},\bar{y},\a)~\gg ~\dots~ \gg~
\theta_{i_n}(\bar{x},\bar{y},\a)
\end{equation}
for any $j\in I_\ell$.
Furthermore, every trivalent vertex $v_{i,j,k}$  of 
$\CC_{\a}$ is contained in the region 
$\mathcal{R}_{I_k}$.
\end{proposition}

Note that since $\theta_i(\bar{x}_0,\bar{y}_0,{\bf a})=r_i-a_0$,
we have \[
\theta_{i}(\bar{x},\bar{y},{\bf a})=\kappa_i(\bar{x}-\bar{x}_0)+\kappa_i^2(\bar{y}-\bar{y}_0)+r_i-a_0.
\]
Since each function $\theta_i$ is linear in $\bar{x}$ and 
$\bar{y}$, and because the asymptotic contour plots
only depend on the orders of the values of the 
functions $\Theta_{i,j}$, we can assume without loss of generality
that $\bar{x}_0=\bar{y}_0 = a_0 = 0$.
Or equivalently, we can always shift our coordinate system
so that 
$(\bar{x}_0,\bar{y}_0,-a_0)$
is the origin in the three-dimensional $\bar{x}\bar{y}z$-space.
Therefore from now on we simply write
\begin{equation}\label{eq:newcoords}
z=\theta_i(\bar{x},\bar{y})=\theta_i(\bar{x},\bar{y},\a) = 
\kappa_i\bar{x}+\kappa_i^2\bar{y}+r_i.
\end{equation}

Using \eqref{eq:newcoords}, the following lemma is easily verified.
\begin{lemma}\label{lem:linespoints}
The line $L_{i,j}$ in the $\bar{x} \bar{y}$-plane at which 
$\theta_i = \theta_j$ has the equation
\begin{equation*}
\bar{x} + (\kappa_i+\kappa_j)\bar{y} + \frac{r_i-r_j}{\kappa_i-\kappa_j} = 0.
\end{equation*}
Given $\{i,j,k\} \subset [n]$, 
the vertex $v_{i,j,k} = (v_{i,j,k}^x,v_{i,j,k}^y)$
where the lines $L_{i,j}$, $L_{i,k}$, and $L_{j,k}$ mutually
intersect has coordinates given by 
\begin{align*}
v_{i,j,k}^x&=\left(\frac{\kappa_j+\kappa_k}{(\kappa_i-\kappa_j)(\kappa_i-\kappa_k)}r_{i}+\frac{\kappa_k+\kappa_i}{(\kappa_{j}-\kappa_i)(\kappa_{j}-\kappa_k)}r_{j}+\frac{\kappa_i+\kappa_j}{(\kappa_{k}-\kappa_i)(\kappa_{k}-\kappa_j)}r_{k}\right)\\
v_{i,j,k}^y&=-\left(\frac{r_i}{(\kappa_i-\kappa_j)(\kappa_i-\kappa_k)}+\frac{r_j}{(\kappa_j-\kappa_i)(\kappa_j-\kappa_k)}+\frac{r_k}{(\kappa_k-\kappa_i)(\kappa_k-\kappa_j)}\right).
\end{align*}
\end{lemma}

We now prove Proposition \ref{prop:induction}.
\begin{proof}
By our choice of the constant $R$ and the $r_i$'s in Theorem \ref{thm:triangle},
Lemma \ref{lem:linespoints} implies  that the distance 
$\|v_{i,j,k}\|$ from the point $v_{i,j,k}$ to the origin
is approximately
\[
\|v_{i,j,k}\|\approx  \text{max}\{|r_i|,|r_j|,|r_k|\}.
\]
Then for $3 \leq \ell \leq n-1$, we let $\mathcal{R}_{I_{\ell}}$ denote the circle around the origin
with radius equal to 
$R^{\ell+\frac{1}{2}}.$
(We let $\mathcal{R}_{I_n}=\R^2$.)
Then it follows that each vertex
$v_{i,j,k}$ for $i,j,k\in I_{\ell}$ lies in $\mathcal{R}_{I_{\ell}}$,
and all others are outside this region.
It is also immediate that in $\mathcal{R}_{I_{\ell}}$, 
 \eqref{eqn:order} holds
for any $j\in I_\ell$.
\end{proof}
\begin{figure}[h]
\begin{centering}
\includegraphics[height=5.7cm]{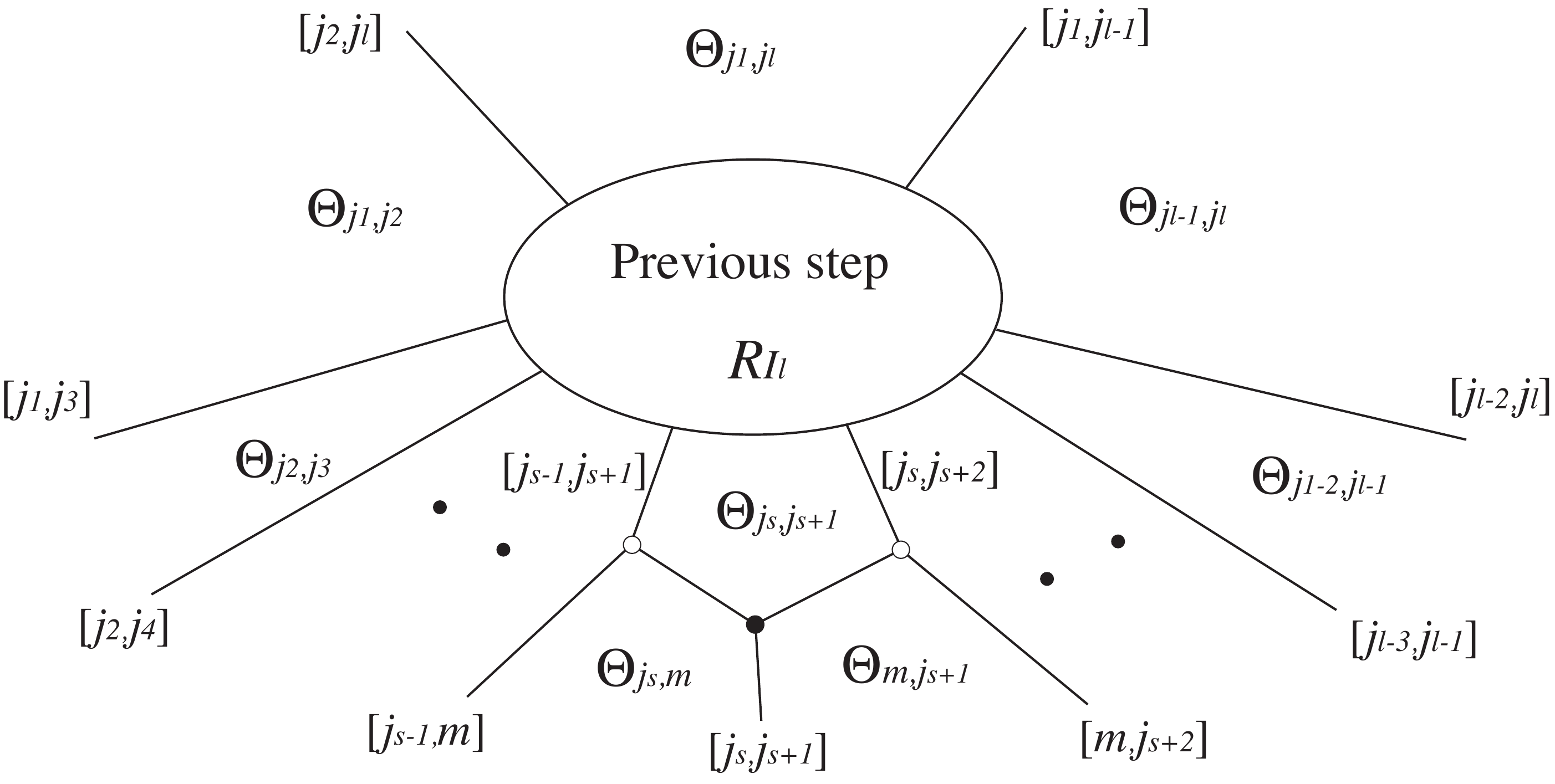}
\par\end{centering}
\caption{The unbounded region labeled by $\Theta_{j_s,j_s+1}$ in $\CC_{\a}^{(\ell)}$ becomes bounded in
$\CC_{\a}^{(\ell+1)}$.  \label{fig:L2}
Here we have
$ j_1 < \dots < j_s < m < j_{s+1} < \dots < j_{\ell}$
with $I_{\ell}=\{i_1,\ldots,i_\ell\}$ and $I_{\ell+1}=I_{\ell}\cup\{m\}$.}
\end{figure}

We're now ready to prove Theorem \ref{thm:triangle}.
\begin{proof}
Note that 
$\CC_{\a}^{(\ell)}$ is an asymptotic contour plot
for $(Gr_{2,\ell})_{>0}$, where the indices of the line-solitons come 
from the set $I_{\ell}$.
By Proposition \ref{prop:induction}, we see 
the ``whole" contour plot
for $\CC_{\a}^{(\ell)}$ within the region 
$\mathcal{R}_{I_{\ell}}$, i.e. we see 
every trivalent vertex and a portion of every unbounded line-soliton.
Moreover, \eqref{eqn:order} guarantees
that if $\Theta_{i,j}$ is a dominant plane 
in some region of $\CC_{\a}^{(\ell)}$, then 
it remains a dominant plane
in some region of $\CC_{\a}^{(\ell+1)}$.

Write $I_{\ell} = \{j_1, j_2, \dots, j_{\ell}\}$,
where $j_1 < j_2 < \dots < j_{\ell}$, and write $m=i_{\ell+1}$.
Since  $\CC_{\a}^{(\ell+1)}$ is an asymptotic contour plot
for $(Gr_{2,\ell+1})_{>0}$ with labels on the line-solitons
and regions coming from the set $I_{\ell+1}$, the dominant 
planes labeling the unbounded regions of 
$\CC_{\a}^{(\ell+1)}$ 
correspond to all cyclically adjacent pairs in the set
$I_{\ell} \cup \{m\}$.  For example, if 
$j_s < m < j_{s+1}$, then those dominant planes are
$\{\Theta_{j_1,j_2}, \Theta_{j_2,j_3}, \dots, 
\Theta_{j_s, m}, \Theta_{m,j_{s+1}}, \dots, \Theta_{j_{\ell-1}, j_{\ell}},
\Theta_{j_1, j_{\ell}}\}$.
Among these, 
$\CC_{\a}^{(\ell+1)}$ 
contains two new dominant planes that 
$\CC_{\a}^{(\ell)}$ 
did not, namely 
$\Theta_{j_s, m}$ and $\Theta_{m,j_{s+1}}$ -- 
see Figure \ref{fig:L2}.

Now we know that:
\begin{itemize}
\item the soliton graph for 
$\CC_{\a}^{(\ell+1)}$ has the form $\Psi(T)$
for some triangulation $T$ of a polygon with vertices $I_{\ell+1}$.
\item If $T' = \lambda(i_1,\dots,i_{\ell})$, then since
$\CC_{\a}^{(\ell)} = \Psi(T')$ up to (M2)-equivalence
(by induction), and each dominant plane of 
$\CC_{\a}^{(\ell)}$ remains a dominant plane of 
$\CC_{\a}^{(\ell+1)}$, it follows that 
$T$ contains all the edges and diagonals that $T'$ does.
\item 
$\CC_{\a}^{(\ell+1)}$  contains 
$\Theta_{j^-, m}$ and $\Theta_{m,j^+}$ as dominant planes,
where $j^-$ and $j^+$ are the clockwise and counterclockwise
neighbors of $m$ in the set $I_{\ell}$, and hence
$T$ contains two edges or diagonals labeled by 
$(j^-,m)$ and $(m,j^+)$.  (In Figure \ref{fig:L2},
where we have $j_s < m < j_{s+1}$, we have
$\{j^+, j^-\} = \{j_s, j_{s+1}\}$.)
\end{itemize}
\begin{figure}[h]
\begin{centering}
\includegraphics[height=3.8cm]{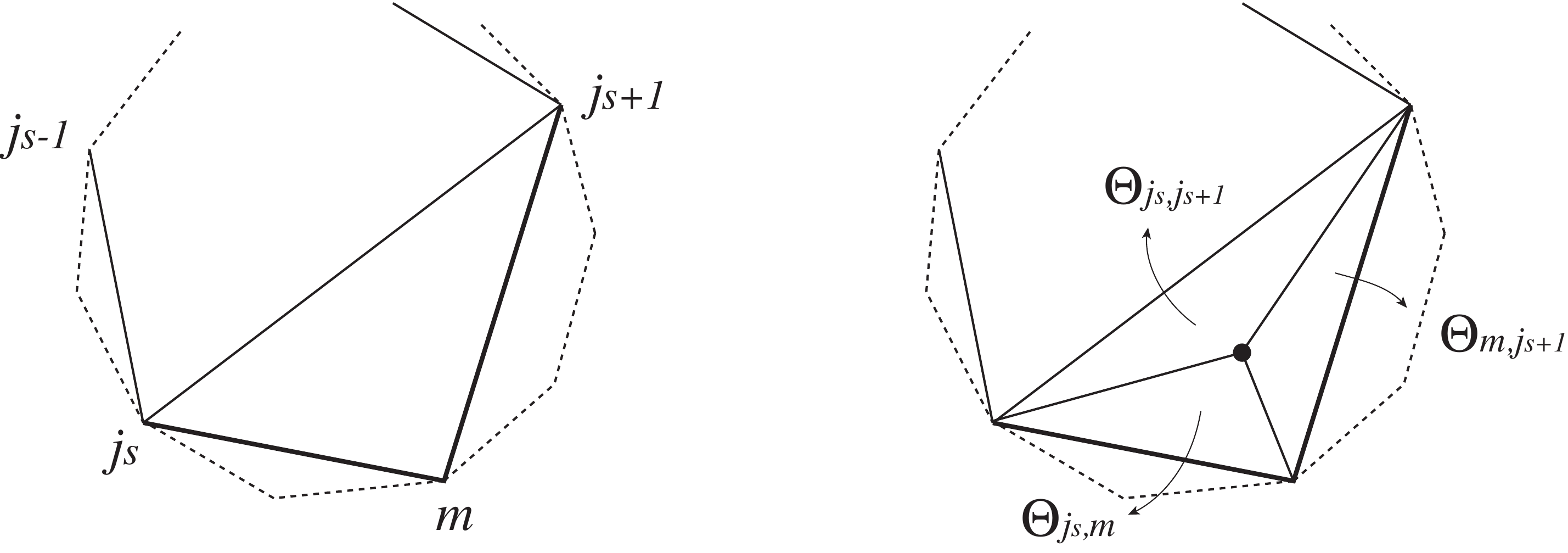}
\par\end{centering}
\caption{Adding a new vertex $m$ to the triangulation together
with one new triangle with the vertices $\{j_s,m,j_{s+1}\}$; and the effect on the  graph $\Psi(T)$ produced
from Algorithm \ref{TriangulationA}.  The regions
 are labeled by the dominant planes $\Theta_{i,j}$.
\label{fig:newtriangulation}}
\end{figure}

The only $T$ with all the properties above is the triangulation obtained
from $T'$ by inserting a new vertex $m=i_{\ell+1}$ in between
its clockwise and counterclockwise neighbors $j^-$ and $j^+$,
and connecting it by edges to those two vertices, see
Figure \ref{fig:newtriangulation}.
This completes the proof.
\end{proof}


\raggedright

\end{document}